\documentclass{amsart}
\usepackage{latexsym}
\usepackage{amsmath}
\usepackage{amsfonts}
 \usepackage[colorlinks=true]{hyperref}
\hypersetup{urlcolor=blue, citecolor=red}
\usepackage{hyperref}

\newtheorem{theorem}{Theorem}[section]
\newtheorem{corollary}{Corollary}

\newtheorem{lemma}[theorem]{Lemma}
\newtheorem{proposition}{Proposition}

\theoremstyle{definition}

\newtheorem{remark}{Remark}

\title[Global solutions to quasi-linear wave equations] 
{Remarks on a system of quasi-linear wave equations 
in $3$D satisfying the weak null condition}

\author[K. Hidano and D. Zha]{}


 \email{hidano@edu.mie-u.ac.jp}
 \email{ZhaDongbing@163.com}
\thanks{$^*$ Corresponding author}%: Kunio Hidano}

\DeclareMathOperator{\esssup}{ess\sup}%Added by the author
\begin{document}
\maketitle

% Enter the first author's name and address:
\centerline{\scshape Kunio Hidano$^*$}
\medskip
{\footnotesize
% please put the address of the first author
 \centerline{Department of Mathematics, Faculty of Education}
   \centerline{Mie University}
   \centerline{1577 Kurima-machiya-cho Tsu, Mie Prefecture 514-8507, Japan}
} % Do not forget to end the {\footnotesize by the sign }

\medskip

\centerline{\scshape Dongbing Zha}
\medskip
{\footnotesize
 % please put the address of the second  and third author
 \centerline{Department of Applied Mathematics}
   \centerline{Donghua University}
   \centerline{Shanghai 201620, PR China}
}

\bigskip

% The name of the associate editor will be entered by an editorial staff
% "Communicated by the associate editor name" is not needed for special issue.
 %\centerline{(Communicated by the associate editor name)}

%The abstract of your paper
\begin{abstract}
We give an alternative proof of the global existence result 
originally due to Hidano and Yokoyama for the Cauchy problem 
for a system of quasi-linear wave equations 
in three space dimensions satisfying the weak null condition. 
The feature of the new proof lies in that it never uses 
the Lorentz boost operator in the energy integral argument. 
The proof presented here has an advantage over 
the former one in that 
the assumption of compactness of the support of data 
can be eliminated 
and the amount of regularity of data can be lowered 
in a straightforward manner. 
A recent result of Zha for the scalar unknowns is also refined.
\end{abstract}

%The title of your section 1
\section{Introduction}
We consider the Cauchy problem for a system of 
quasi-linear wave equations in three space dimensions 
satisfying the weak null condition 
given by Lindblad and Rodnianski \cite{LR}
\begin{equation}\label{1}
\begin{cases}
\displaystyle{
\Box u_1
+
G_1^{11,\alpha\beta\gamma}
(\partial_\gamma u_1)
(\partial_{\alpha\beta}^2 u_1)
+
G_1^{21,\alpha\beta\gamma}
(\partial_\gamma u_2)
(\partial_{\alpha\beta}^2 u_1)
             }\\
%%%%%%%%%%%%%%%%%%%%%
\displaystyle{
+
H_1^{11,\alpha\beta}
(\partial_\alpha u_1)
(\partial_\beta u_1)
+
H_1^{12,\alpha\beta}
(\partial_\alpha u_1)
(\partial_\beta u_2)
+
H_1^{22,\alpha\beta}
(\partial_\alpha u_2)
(\partial_\beta u_2)
=0
               },\\
%%%%%%%%%%%%%%%%%
\displaystyle{
\Box u_2
+
G_2^{12,\alpha\beta\gamma}
(\partial_\gamma u_1)
(\partial_{\alpha\beta}^2 u_2)
+
G_2^{22,\alpha\beta\gamma}
(\partial_\gamma u_2)
(\partial_{\alpha\beta}^2 u_2)
                }\\
%%%%%%%%%%%%%%%%%%%%%%%
\displaystyle{
+
H_2^{12,\alpha\beta}
(\partial_\alpha u_1)
(\partial_\beta u_2)
+
H_2^{11,\alpha\beta}
(\partial_\alpha u_1)
(\partial_\beta u_1)
+
H_2^{22,\alpha\beta}
(\partial_\alpha u_2)
(\partial_\beta u_2)
=0
               }
\end{cases}
\end{equation}
with data given at $t=0$. Here, 
$\Box:=\partial_t^2-\Delta$, 
$\partial_0:=\partial/\partial t$, 
$\partial_i:=\partial/\partial x_i$, $i=1,2,3$. 
We always use the summation convention: 
when the same index is above and below, 
summation over this index is assumed from $0$ to $3$. 
Since our concern is in classical solutions, 
we may assume the symmetry condition 
without loss of generality: there hold 
$G_1^{11,\alpha\beta\gamma}=G_1^{11,\beta\alpha\gamma}$, 
$G_1^{21,\alpha\beta\gamma}=G_1^{21,\beta\alpha\gamma}$, 
and 
$G_2^{12,\alpha\beta\gamma}=G_2^{12,\beta\alpha\gamma}$, 
$G_2^{22,\alpha\beta\gamma}=G_2^{22,\beta\alpha\gamma}$ 
for all $\alpha,\beta,\gamma=0,\dots,3$.  
Using the energy inequality obtained by the method of ghost weight due to Alinhac 
(see \cite{Al2001}, \cite{Al2010}), 
Yokoyama and the first author have proved in \cite{HY2017}:
\begin{theorem}\label{theorem1.1}
Suppose 
\begin{align}
&
G_1^{11,\alpha\beta\gamma}
X_\alpha X_\beta X_\gamma
=
G_1^{21,\alpha\beta\gamma}
X_\alpha X_\beta X_\gamma
=
G_2^{22,\alpha\beta\gamma}
X_\alpha X_\beta X_\gamma
=0,\label{2}\\
%%%%%%%%%%%%%%%%%%%%
&
H_1^{11,\alpha\beta}X_\alpha X_\beta
=
H_1^{12,\alpha\beta}X_\alpha X_\beta
=
H_1^{22,\alpha\beta}X_\alpha X_\beta
=
H_2^{22,\alpha\beta}X_\alpha X_\beta
=0\label{3}
\end{align}
for any $X=(X_0,\dots,X_3)\in{\mathbb R}^4$ 
with 
$X_0^2=X_1^2+X_2^2+X_3^2$. 
Let 
$0<\eta<1/6$, 
$0<\delta<1/6$ so that 
$\eta+2\delta<1/2$. 
Then, there exist constants 
$C>0$, $0<\varepsilon<1$ 
depending only on 
the coefficients of the system $(\ref{1})$, 
$\delta$, and $\eta$ 
such that 
if compactly supported smooth data satisfy 
$W_4(u_1(0))+W_4(u_2(0))<\varepsilon$, 
then the Cauchy problem for $(\ref{1})$ 
admits a unique global smooth solution 
$(u_1(t,x),u_2(t,x))$ satisfying for all 
$t>0$, $T>0$
\begin{equation}\label{4}
\begin{split}
&
W_4(u_1(t))
+
(1+t)^{-\delta}
W_4(u_2(t))\\
&
\hspace{0.2cm}
+
\sum_{i=1}^3
\sum_{|a|\leq 3}
\biggl(
\|
\langle t-r\rangle^{-(1/2)-\eta}
T_i \Gamma^a u_1
\|_{L^2((0,\infty)\times{\mathbb R}^3)}\\
&
\hspace{2.2cm}
+
(1+T)^{-\delta}
\|
\langle t-r\rangle^{-(1/2)-\eta}
T_i \Gamma^a u_2
\|_{L^2((0,T)\times{\mathbb R}^3)}
\biggr)\\
&
\hspace{0.1cm}
\leq
C\bigl(
W_4(u_1(0))+W_4(u_2(0))
\bigr).
\end{split}
\end{equation}
Here $T_i=\partial_i+(x_i/|x|)\partial_t$. 
\end{theorem}
%%%%%%%%%%%%%%%%%%%%%%%%
\begin{remark}\label{remark1}
In Section 3 of \cite{HY2017}, 
thanks to compactness of the support of initial data 
together with the finite speed of propagation, 
the proof of Theorem \ref{theorem1.1} was able to employ 
the standard local existence theorem 
in solving locally (in time) the Cauchy problem with data given at $t=0$ 
and in continuing the local solutions to a larger strip, 
though some partial differential operators 
with ``weight'' (see just below) were naturally used. 
We should remark that the constant $\varepsilon$ in the above theorem 
is independent of the ``radius'' of the support of given data 
$(u_i(0),\partial_t u_i(0))=(f_i,g_i)$ $(i=1,2)$, 
that is, 
$R_*:=
\inf
\bigl\{\,r>0\,:
\,{\rm supp}\,\{f_1,g_1,f_2,g_2\}
\subset 
\{
x\in{\mathbb R}^3:|x|<r
\}
\bigr\}$.
\end{remark}
Here we explain the notation used in the statement of 
Theorem \ref{theorem1.1}. We set 
\begin{align}
&
E_1(u(t))
:=
\frac12
\int_{{\mathbb R}^3}
\bigl(
(\partial_t u(t,x))^2
+
|\nabla u(t,x)|^2
\bigr)dx,\label{5}\\
&
W_\kappa(u(t))
:=
\sum_{|a|\leq \kappa-1}
E_1^{1/2}(\Gamma^a u(t)),
\quad \kappa=2,3,\dots\label{6}
\end{align}
By $\Gamma$, we mean the set of the operators 
$\partial_\alpha$ $(\alpha=0,\dots,3)$, 
$\Omega_{ij}:=x_i\partial_j-x_j\partial_i$ 
$(1\leq i<j\leq 3)$, 
$L_k:=x_k\partial_t+t\partial_k$ 
$(k=1,2,3)$, 
and $S:=t\partial_t+x\cdot\nabla$. 
Also, for a multi-index $a$, 
$\Gamma^a$ stands for any product 
of the $|a|$ these operators. 
We remark that $\partial_t^ku_i(0,x)$ for $i=1,2$ and $k=2,3,4$ 
can be calculated with the help of the equation (\ref{1}), 
and thus the quantity $W_4(u_1(0))+W_4(u_2(0))$ appearing in (\ref{1}) 
is determined by the given initial data. 

We note that 
the proof of Theorem \ref{theorem1.1} 
fully exploits the Lorentz invariance 
in the sense that it uses 
the operators $\Omega_{ij}$ and $L_k$, 
in addition to $\partial_\alpha$ and $S$. 
When it comes to the Cauchy problem for a nonrelativistic system 
satisfying the weak null condition 
(see, e.g., (2.8) of \cite{PS2013}) or 
the initial-boundary value problems 
in a domain exterior to an obstacle 
(see, e.g., \cite{ST}, \cite{KSS}, \cite{MS}), 
the use of $L_k$ should be avoided. 
The purpose of this paper is to 
revisit the Cauchy problem for (\ref{1}) 
and prove global existence without relying upon $L_k$. 
Moreover, we also aim at eliminating compactness 
of the support of data 
and lowering the amount of regularity of data. 
To state the main theorem precisely, 
we set the notation. 
As in \cite{H2016}, 
we define
\begin{equation}\label{7}
\begin{split}
&
N_1(u(t))
:=
\sqrt{E_1(u(t))},\,
N_2(u(t))
:=
\left(
\sum_{|a|+|b|+d\leq 1}
E_1(\partial_x^a \Omega^b S^d u(t))
\right)^{1/2},\\
&
N_4(u(t))
:=
\left(
\sum_{{|a|+|b|+d\leq 3}\atop{d\leq 1}}
E_1(\partial_x^a \Omega^b S^d u(t))
\right)^{1/2},
\end{split}
\end{equation}
where, for $a=(a_1,a_2,a_3)$ and $b=(b_1,b_2,b_3)$, 
$\partial_x^a:=\partial_1^{a_1}\partial_2^{a_2}\partial_3^{a_3}$, 
$\Omega^b:=\Omega_{12}^{b_1}\Omega_{13}^{b_2}\Omega_{23}^{b_3}$. 
We also define for a pair of 
time-independent functions $(v(x),w(x))$
\begin{equation}\label{8}
\begin{split}
&D(v,w)\\
:=&
\left(
\sum_{{|a|+|b|+d\leq 3}\atop{d\leq 1}}
\left(
\int_{{\mathbb R}^3}
|\nabla\partial_x^a\Omega^b\Lambda^d v(x)|^2dx
+
\int_{{\mathbb R}^3}
|\partial_x^a\Omega^b\Lambda^d w(x)|^2dx
\right)
\right)^{1/2}.
\end{split}
\end{equation}
Here, we have set $\Lambda:=x\cdot\nabla$, which can be regarded 
as a time-independent analogue of $S$. 
Since $\partial_t Su=\partial_t u+\Lambda\partial_t u$ 
at $t=0$, there obviously exists a numerical constant 
$C_D>0$ such that 
$N_4(u(0))\leq C_DD(u(0),\partial_tu(0))$ 
for smooth functions $u(t,x)$.
%%%%%%%%%%%%%%%%%%%%%%%%%%%%%%%%%%%%%%%%%%%
%%%%%%%%%%%%%%%%%%%%%%%%%%%%%%%%%%%%%%%%%%
\begin{theorem}\label{theorem1.2}
Suppose $(\ref{2})$, $(\ref{3})$ 
for any $X=(X_0,\dots,X_3)\in{\mathbb R}^4$ 
with 
$X_0^2=X_1^2+X_2^2+X_3^2$. 
Then, there exists $\varepsilon\in (0,1)$ 
such that if $f_1,f_2\in L^6({\mathbb R}^3)$ and 
$D(f_1,g_1)+D(f_2,g_2)<\varepsilon$, 
then the Cauchy problem for $(\ref{1})$ 
with data 
$(u_i,\partial_t u_i)=(f_i,g_i)$ 
$(i=1,2)$ 
given at $t=0$ 
admits a unique global solution 
$u(t,x)=(u_1(t,x),u_2(t,x))$ satisfying 
\begin{equation}
\esssup\displaylimits_{t>0}
{\mathcal N}(u(t))
+
{\mathcal G}_T(u)
+
{\mathcal L}_T(u)
\leq
C
\sum_{i=1}^2
D(f_i,g_i)\label{9}
\end{equation}
for all $T>0$, with a constant $C>0$ independent of $T$. 
\end{theorem}
%%%%%%%%%%%%%%%%%%%%%%%%%%
For the definition of 
${\mathcal N}(u)$, 
${\mathcal G}_T(u)$, and ${\mathcal L}_T(u)$, 
see (\ref{eqn:SN1}), (\ref{87}), and (\ref{88}), respectively. 
(We remark that the constant $\delta$ appearing in 
(\ref{eqn:SN1}), (\ref{87})--(\ref{88}) is smaller than in 
Theorem \ref{theorem1.1}.) 
Compared with $W_4(u_1(0))$ (see Theorem \ref{theorem1.1} above), 
the semi-norm $D(v,w)$ has a couple of advantages. 
Firstly, 
by the standard way we can easily find a sequence 
$\{(v_j,w_j)\}\in C_0^\infty({\mathbb R}^3)\times C_0^\infty({\mathbb R}^3)$ 
such that 
$D(v-v_j,w-w_j)\to 0$ as $j\to\infty$ 
when $v\in L^6({\mathbb R}^3)$ and 
$\nabla\partial_x^a\Omega^b\Lambda^d v$, 
$\partial_x^a\Omega^b\Lambda^d w
\in 
L^2({\mathbb R}^3)$ 
for any $|a|+|b|+d\leq 3$ with $d\leq 1$. 
(We remark that the corresponding procedure becomes rather complicated 
when we employ $W_4$ (see (\ref{6})), as in \cite{HY2017}, to measure the size of data.) 
We are naturally led to 
proving Theorem \ref{theorem1.2} first for compactly supported smooth data 
(because the proof of global existence becomes easier for 
such initial data), 
and then we use this helpful property 
to complete its proof by passing to the limit of 
a sequence of compactly supported 
(for any fixed time) smooth solutions. 
See Section 8. 
Secondly, thanks to the limitation 
of the number of $\Lambda$ to $1$ 
in the definition of $D(v,w)$, 
we easily see that 
the size condition in Theorem \ref{theorem1.2} 
is satisfied whenever 
the initial data is radially symmetric about $x=0$ and 
its norm with the low weight 
%%%%%%%%%%%%%%%%%%%%%%%%%%%%%%%%%
%%%%%%%%%%%%%%%%%%%%%%%%%%%%%%%%%%
$\langle x\rangle:=\sqrt{1+|x|^2}$ 
\begin{equation}
\sum_{i=1,2}
\biggl(
\sum_{1\leq |a|\leq 4}
\|\langle x\rangle\partial_x^a f_i\|_{L^2}
+
\sum_{|a|\leq 3}
\|\langle x\rangle\partial_x^a g_i\|_{L^2}
\biggr)\label{10}
\end{equation}
%%%%%%%%%%%%%%%%%%% 
%%%%%%%%%%%%%%%%%%%%
is small enough. 
Note that, thanks to its low weight, 
we can allow such an oscillating and slowly decaying data as 
$g_1(x)=\langle x\rangle^{-d}\sin\langle x\rangle$ with 
$d>5/2$. 
Naturally, it results from 
the limitation 
of the number of $S$ to $1$ 
in the definition of ${\mathcal N}(u)$, 
${\mathcal G}_T(u)$, and ${\mathcal L}_T(u)$. 
%%%%%%%%%%%%%%%%%%%%%%%%%%%%%%%%%%%%%%%%%%%%%%%%%%%%%

%%%%%%%%%%%%%%%%%%%%%%%%%%%%%%%%%%%%%%%%%
Note that 
by setting $H_2^{11,\alpha\beta}=0$ for all $\alpha$, $\beta$ 
and choosing the trivial data 
$u_2(0,x)=\partial_t u_2(0,x)=0$ 
and thus considering the trivial solution 
$u_2(t,x)\equiv 0$, 
we can go back to the wave equation for the scalar unknowns 
%%%%%%%%%%%%%%%%%%%%%%%%%%%%%%%
%%%%%%%%%%%%%%%%%%%%%%%%%%%%%%%%
\begin{equation}
\Box u
+
G^{\alpha\beta\gamma}
(\partial_\gamma u)
(\partial_{\alpha\beta}^2 u)
+
H^{\alpha\beta}
(\partial_\alpha u)
(\partial_\beta u)
=0,
\,\,t>0,\,x\in{\mathbb R}^3\label{11}
\end{equation}
and thus obtain: 
\begin{theorem}\label{theorem1.3}
Suppose the symmetry condition 
$G^{\alpha\beta\gamma}=G^{\beta\alpha\gamma}$. 
Also, suppose the null condition: there holds 
\begin{equation}
G^{\alpha\beta\gamma}X_\alpha X_\beta X_\gamma
=
H^{\alpha\beta}X_\alpha X_\beta=0
\label{12}
\end{equation}
for any $X=(X_0,\dots,X_3)\in{\mathbb R}^4$ 
with $X_0^2=X_1^2+X_2^2+X_3^2$.  
Let $\delta$, $\eta$ and $\mu$ be 
sufficiently small positive constants. 
Then, there exists $\varepsilon\in (0,1)$ 
such that if $f\in L^6({\mathbb R}^3)$ and 
$D(f,g)\leq\varepsilon$, 
then 
the Cauchy problem $(\ref{11})$ with initial data 
$(f,g)$ given at $t=0$ admits 
a unique global solution $u(t,x)$ 
satisfying 
%%%%%%%%%%%%%%%%%%%%
%%%%%%%%%%%%%%%%%%%
\begin{equation}\label{13}
\begin{split}
&
\esssup\displaylimits_{t>0}N_4(u(t))\\
&
\hspace{0.2cm}
+
%%%%%%%%%%%%%%%
\biggl(
\int_0^\infty
\sum_{i=1}^3
\sum_{{|a|+d\leq 3}\atop{d\leq 1}}
\|
\langle t-r\rangle^{-(1/2)-\eta}
T_i{\bar Z}^aS^d u(t)
\|_{L^2({\mathbb R}^3)}^2
dt
\biggr)^{1/2}\\
&
\hspace{0.2cm}
+
\biggl(
\int_0^\infty
\sum_{i=1}^3
\sum_{{|a|+d\leq 2}\atop{d\leq 1}}
\|
\langle t-r\rangle^{-(1/2)-\eta}
T_i\partial_t{\bar Z}^aS^d u(t)
\|_{L^2({\mathbb R}^3)}^2
dt
\biggr)^{1/2}\\
&
\hspace{0.2cm}
+
\sup_{t>0}
\langle t\rangle^{-\mu-\delta}
\biggl(
\int_0^t
\sum_{{|a|+d\leq 3}\atop{d\leq 1}}
\biggl(
\|
r^{-(3/2)+\mu}
{\bar Z}^a S^d u(\tau)
\|_{L^2({\mathbb R}^3)}^2\\
&
\hspace{4.5cm}
+
\|
r^{-(1/2)+\mu}
\partial{\bar Z}^a S^d u(\tau)
\|_{L^2({\mathbb R}^3)}^2
\biggr)
d\tau
\biggr)^{1/2}\\
&
\hspace{0.1cm}
\leq
C
D(f,g).
\end{split}
\end{equation}
%%%%%%%%%%%%%%%%%%%%%%
%%%%%%%%%%%%%%%%%%%%%%
\end{theorem}
See the beginning of the next section for 
the definition of ${\bar Z}$. 
This improves Theorem 1.1 of the second author \cite{Zha} 
which says global existence of solutions to (\ref{11}) 
in the absence of the semi-linear term 
$H^{\alpha\beta}(\partial_\alpha u)(\partial_\beta u)$ 
for small data with higher regularity than is assumed in 
Theorem \ref{theorem1.3}. 
Since we no longer assume compactness of the support 
of initial data, 
Theorem \ref{theorem1.3} is also an improvement of global existence 
results of \cite{Al2010} and \cite{HY2017} for (\ref{11}) 
(see \cite[p.\,94]{Al2010} and \cite[Theorem 1.5]{HY2017}).

The operators $L_k$ together with the other elements of $\Gamma$ 
played an essential role 
in the proof of Theorem \ref{theorem1.1}. 
Namely, the use of all the elements of $\Gamma$ was crucial 
for the purpose of getting time decay estimates 
for local solutions with the help of 
the inequality of Klainerman \cite{Kl87} and its $H^1$--$L^q$ version 
due to Ginibre and Velo \cite{GV}. 
Since we avoid the use of the operators $L_k$, 
some good substitutes for these inequalities are necessary. 
In fact, 
there already exist two major ways of 
obtaining time decay estimates 
without relying upon $L_k$. 
One is to use point-wise decay estimates 
for homogeneous and inhomogeneous wave equations 
(see, e.g., \cite{Y}). 
The other is to use the Klainerman-Sideris inequality \cite{KS} 
in combination with 
some Sobolev-type inequalities with weights such as 
$\langle t-r\rangle$, 
$\langle r\rangle\langle t-r\rangle^{1/2}$ 
(see, e.g., \cite{SiderisTu}). 
As in \cite{Zha}, 
we proceed along the latter approach 
to compensate for the absence of $L_k$ in the list 
of the available differential operators 
and intend to combine the ghost weight method of Alinhac 
with the Klainerman-Sideris method. 
Actually, such an attempt of combining these two methods 
has been already made in \cite{Zha}. 
With the help of some observations in \cite{H2016} and \cite{HY2017}, 
we adjust the machinery thereby assembled in \cite{Zha}, 
in order to reduce the amount of regularity of initial data, 
and also to discuss the system (\ref{1}) violating 
the standard null condition but satisfying the weak null condition. 
We hope that in the future, this machinery will be useful 
in discussing the Cauchy problem for a nonrelativistic system 
satisfying the weak null condition 
or the initial-boundary value problems 
in a domain exterior to an obstacle. 

This paper is organized as follows. 
In the next section, we prove some basic inequalities. 
In Section 3, we consider the bound for 
the weighted $L^2$ norm of the second or higher-order 
derivatives of local solutions. 
Sections 4--5 and 6--7 are devoted to the energy estimate 
and the space-time $L^2$ estimate for local solutions, 
respectively. 
In Section 8, we complete the proof of Theorem \ref{theorem1.2} 
by the continuity argument. 
%%%%%%%%%%%%%%%%%%%%%%%%%%%%%%%%%
%%%%%%%%%%%%%%%%%%%%%%%%%%%%%%%%
%%%%%%%%%%%%%%%%%%%%%%%%%%%%%%%%%%
\section{Preliminaries}
As mentioned in Section 1, 
we use $\partial_1$, $\partial_2$, $\partial_3$, 
$\Omega_{12}$, $\Omega_{23}$, $\Omega_{13}$ and $S$, 
and we denote these by 
$Z_1$, $Z_2,\dots, Z_7$ in this order. 
The set $\{Z_1, Z_2,\dots, Z_7\}$ is denoted by $Z$. 
Note that $\partial_t\notin Z$. 
For a multi-index $a=(a_1,\dots,a_7)$, 
we set $Z^a:=Z_1^{a_1}\cdots Z_7^{a_7}$. 
We also set 
${\bar Z}:=\{Z_1, Z_2,\dots, Z_6\}
=Z\setminus\{S\}$, 
and 
${\bar Z}^a:=Z_1^{a_1}\cdots Z_6^{a_6}$ 
for $a=(a_1,\dots,a_6)$. 

We need the commutation relations. 
Let $[\cdot,\cdot]$ be the commutator: 
$[A,B]:=AB-BA$. 
It is easy to verify that 
\begin{eqnarray}
& &
[Z_i,\Box]=0\,\,\,\mbox{for $i=1,\dots,6$},\,\,\,
[S,\Box]=-2\Box,
\label{eqn:comm1}\\
& &
[Z_j,Z_k]=\sum_{i=1}^\mu C^{j,k}_i Z_i,\,\,\,
j,\,k=1,\dots,7,
\label{eqn:comm2}\\
& &
[Z_j,\partial_k]
=
\sum_{i=1}^n C^{j,k}_i\partial_i,\,\,\,j=1,\dots,7,\,\,k=1,2,3,
\label{eqn:comm3}\\
& &
[Z_j,\partial_t]=0,\,j=1,\dots,6,\quad [S,\partial_t]=-\partial_t
\label{eqn:comm4}.
\end{eqnarray}
Here $C^{j,k}_i$ denotes a constant depending on 
$i$, $j$, and $k$.

The next lemma states that the null form is preserved 
under the differentiation. 
\begin{lemma}\label{lemma2.1}
Suppose that $\{G^{\alpha\beta\gamma}\}$ and 
$\{H^{\alpha\beta}\}$ satisfy the null condition 
$($see $(\ref{2})$, $(\ref{3})$, and $(\ref{12})$ above$)$. 
For any $Z_i$ $(i=1,\dots,7)$, 
the equality 
\begin{equation}\label{18}
\begin{split}
&
Z_i
G^{\alpha\beta\gamma}
(\partial_\gamma v)
(\partial_{\alpha\beta}^2 w)\\
=&
G^{\alpha\beta\gamma}
(\partial_\gamma Z_i v)
(\partial_{\alpha\beta}^2 w)
+
G^{\alpha\beta\gamma}
(\partial_\gamma v)
(\partial_{\alpha\beta}^2 Z_i w)
+
{\tilde G}_i^{\alpha\beta\gamma}
(\partial_\gamma v)
(\partial_{\alpha\beta}^2 w)
\end{split}
\end{equation}
holds 
with the new coefficients 
$\{{\tilde G}_i^{\alpha\beta\gamma}\}$ 
also satisfying the null condition. 
Also, the equality 
\begin{equation}\label{19}
\begin{split}
&
Z_i
H^{\alpha\beta}
(\partial_\alpha v)
(\partial_\beta w)\\
=&
H^{\alpha\beta}
(\partial_\alpha Z_i v)
(\partial_\beta w)
+
H^{\alpha\beta}
(\partial_\alpha v)
(\partial_\beta Z_i w)
+
{\tilde H}_i^{\alpha\beta}
(\partial_\alpha v)
(\partial_\beta w)
\end{split}
\end{equation}
holds 
with the new coefficients 
$\{{\tilde H}_i^{\alpha\beta}\}$ 
also satisfying the null condition. 
\end{lemma}
%%%%%%%%%%%%%%%%%%%%%%
For the proof, see, e.g., \cite[p.\,91]{Al2010}.
%%%%%%%%%%%%%%%%%%%%%%%%%%%%%%%%%
It is possible to show the following lemma 
essentially in the same way as in \cite[pp.\,90--91]{Al2010}. 
Together with it, we will later exploit the fact that 
for local solutions $u$, 
the special derivatives $T_i u$ have 
better space-time $L^2$ integrability 
and improved time decay property of their $L^\infty({\mathbb R}^3)$ norms. 
%%%%%%%%%%%%%%%%%%%%%%%%%%%%%%%%%%%%%%
\begin{lemma}\label{lemma2.2}
Set $\omega_0=-1$, $\omega_k=x_k/|x|$, $k=1,2,3$. 
Suppose that 
$\{G^{\alpha\beta\gamma}\}$, $\{H^{\alpha\beta}\}$ satisfy 
the null condition. Then, we have 
for smooth functions $w_i(t,x)$ $(i=1,2,3)$
\begin{align}
&
|
G^{\alpha\beta\gamma}
(\partial_\gamma w_1)
(\partial_{\alpha\beta}^2 w_2)
|
\leq
C
\bigl(
|T w_1|
|\partial^2 w_2|
+
|\partial w_1|
|T\partial w_2|
\bigr),\label{eqn:null1}\\
&
|
G^{\alpha\beta\gamma}
(\partial_{\alpha\gamma}^2 w_1)
(\partial_\beta w_2)
|
\leq
C
\bigl(
|T\partial w_1|
|\partial w_2|
+
|\partial^2 w_1|
|T w_2|
\bigr),\label{eqn:null2}
\end{align}
%%%%%%%%%%%%%%%%%%%%%%%%%%%%%%%%
\begin{equation}\label{eqn:null3}
\begin{split}
&|
G^{\alpha\beta\gamma}
(\partial_\gamma w_1)
(\partial_\beta w_2)
(\partial_\alpha w_3)
|,\,
|
G^{\alpha\beta\gamma}
(\partial_\gamma w_1)
(\partial_\beta w_2)
(-\omega_\alpha)
(\partial_t w_3)
|
\\
\leq&
C
\bigl(
|T w_1|
|\partial w_2|
|\partial w_3|
+
|\partial w_1|
|T w_2|
|\partial w_3|
+
|\partial w_1|
|\partial w_2|
|T w_3|
\bigr),
\end{split}
\end{equation}
%%%%%%%%%%%%%%%%%%%%%%%%%%%%%%%%%
\begin{equation}\label{eqn:null4}
|
H^{\alpha\beta}
(\partial_\alpha v)
(\partial_\beta w)
|
\leq
C
\bigl(
|T v|
|\partial w|
+
|\partial v|
|T w|
\bigr).
\end{equation}
%%%%%%%%%%%%%%%%%%%%%%%%%%%%%%%%%
\end{lemma}
%%%%%%%%%%%%%%%%%%%%%%%%%%%%%%%%
Here, and in the following, we use the notation 
$\partial v:=(\partial_0v,\dots,\partial_3v)$, 
%%%%%%%%%%%%%%%%%%%%%%%%%%%%%%%%%%%
\begin{equation}
|Tv|
:=
\biggl(
\sum_{k=1}^3
|T_k v|^2
\biggr)^{1/2},\quad
|T\partial v|
:=
\biggl(
\sum_{k=1}^3
\sum_{\gamma=0}^3
|T_k\partial_\gamma v|^2
\biggr)^{1/2}\label{24}
\end{equation}
%%%%%%%%%%%%%%%%%%%%%%%
\begin{lemma}[Lemma 2.2 of \cite{Zha}]\label{lemma2.3}
The inequality
\begin{equation}\label{eqn:Zha}
\begin{split}
|Tv(t,x)|
\leq
C
\langle t\rangle^{-1}
\bigl(&
|\partial_x v(t,x)|
+
|\partial_t v(t,x)|
+
\sum_{|b|=1}
|\Omega^b v(t,x)|\\
&+
|S v(t,x)|
+
\langle t-r\rangle|\partial_x v(t,x)|
\bigr)
\end{split}
\end{equation}
holds for smooth functions $v(t,x)$. 
\end{lemma}
%%%%%%%%%%%%%%%%%%%%%%%%%%%%%%%%%%%%%
The following lemma is concerned with Sobolev-type 
or trace-type inequalities. 
We use these inequalities in combination with 
the Klainerman-Sideris inequality (see (\ref{eqn:KSineq}) below). 
The auxiliary norms 
\begin{align}
M_2(v(t))
&=
\sum_{{0\leq\delta\leq 3}\atop{1\leq j\leq 3}}
\|\langle t-|x|\rangle
\partial_{\delta j}^2 v(t)\|_{L^2({\Bbb R}^3)},\label{26}\\
M_4(v(t))
&=\sum_{|a|\leq 2}M_2({\bar Z}^a v(t)),\label{27}
\end{align}
which appear in the following discussion, play an intermediate role. 
We remark that $S$ and $\partial_t^2$ are absent in the 
right-hand side above. 
We also use the notation $\partial_r:=(x/|x|)\cdot\nabla$, 
\begin{align}
&
\|w\|_{L_r^\infty L_\omega^p({\mathbb R}^3)}
:=
\sup_{r>0}
\|w(r\cdot)\|_{L^p(S^2)},\label{28}\\
&
\|w\|_{L_r^2 L_\omega^p({\mathbb R}^3)}
:=
\biggl(
\int_0^\infty \|w(r\cdot)\|_{L^p(S^2)}^2 r^2dr
\biggr)^{1/2}.\label{29}
\end{align}

\begin{lemma}\label{lemma2.4}
Suppose that $v$ decays sufficiently fast as $|x|\to\infty$. 
The following inequalities hold for $\alpha=0,1,2,3$
\begin{align}
&
\|
\langle t-r\rangle\partial_\alpha v(t)
\|_{L^6({\mathbb R}^3)}
\leq
C
\bigl(
N_1(v(t))
+
M_2(v(t))
\bigr),\label{eqn:ell6}\\
&
\langle t-r\rangle
|\partial_\alpha v(t,x)|
\leq
C
\biggl(
\sum_{|a|\leq 1}N_1(\partial_x^a v(t))
+
\sum_{|a|\leq 1}M_2(\partial_x^a v(t))
\biggr).\label{eqn:ellinfty}
\end{align}
Moreover, we have 
\begin{align}
&
\|r\partial_\alpha v(t)\|_{L_r^\infty L_\omega^4({\mathbb R}^3)}
\leq
C
\sum_{|a|\leq 1}
N_1({\bar Z}^a v(t)),\label{eqn:j2}
\\
&
\langle r\rangle
|\partial_\alpha v(t,x)|
\leq
C\sum_{|a|\leq 2}
N_1({\bar Z}^a v(t)).
\label{eqn:j3}
\end{align}
\end{lemma}
These inequalities have been already employed in the literature. 
For the proof of (\ref{eqn:ell6}), see (2.10) of \cite{H2016}. 
For the proof of (\ref{eqn:ellinfty}), 
see (37) of \cite{Zha}, (2.13) of \cite{H2016}. 
See (3.19) of \cite{Sideris2000} for the proof of (\ref{eqn:j2}). 
Finally, combining (3.14b) of \cite{Sideris2000} 
with the Sobolev embedding $H^2({\mathbb R}^3)\hookrightarrow 
L^\infty({\mathbb R}^3)$, 
we obtain (\ref{eqn:j3}).

We also need the following inequality. 
\begin{lemma}\label{lemma2.5}
Suppose that $v$ decays sufficiently fast as $|x|\to\infty$. 
For any $\theta$ with $0\leq \theta\leq 1/2$, 
there exists a constant $C>0$ such that the inequality 
\begin{equation}
r^{(1/2)+\theta}
\langle t-r\rangle^{1-\theta}
\|\partial_\alpha v(t,r\cdot)\|_{L^4(S^2)}
\leq
C
\biggl(
\sum_{|a|\leq 1}N_1(\Omega^a v(t))
+
M_2(v(t))
\biggr)
\label{eqn:interp}
\end{equation}
holds.
\end{lemma} 
Following the proof of (3.19) in \cite{Sideris2000}, 
we are able to obtain this inequality for $\theta=1/2$. 
The next lemma with 
$v=\langle t-r\rangle\partial_\alpha w$ immediately 
yields (\ref{eqn:interp}) for $\theta=0$. 
We follow the idea in Section 2 of \cite{MNS-JJM2005} and 
obtain (\ref{eqn:interp}) for $\theta\in (0,1/2)$ by interpolation. 

In our proof, the trace-type inequality also plays an important role. 
(For the proof, see, e.g., (3.16) of \cite{Sideris2000}.)
\begin{lemma}\label{lemma2.6}
There exists a positive constant $C$ such that 
if $v=v(x)$ decays sufficiently fast as $|x|\to\infty$, 
then the inequality
\begin{equation}
r^{1/2}
\|v(r\cdot)\|_{L^4(S^2)}
\leq
C\|\nabla v\|_{L^2({\mathbb R}^3)}
\label{eqn:hoshiro}
\end{equation}
holds.
\end{lemma}
%%%%%%%%%%%%%%%%%%%%%%%%%%%%%%%%%%
%%%%%%%%%%%%%%%%%%%%%%%%%%%%%%%%%%%%
We also need the space-time $L^2$ estimates for 
the variable-coefficient operator $P$ defined as 
\begin{equation}
P
:=
\partial_t^2-\Delta
+h^{\alpha\beta}(t,x)\partial_{\alpha\beta}^2. 
\end{equation}
Let $h^{\alpha\beta}\in C^\infty((0,T)\times{\mathbb R}^3)$  
$(\alpha, \beta = 0,1,2,3)$, 
and suppose the symmetry condition 
$h^{\alpha\beta}=h^{\beta\alpha}$ 
and 
the size condition $\sum |h^{\alpha\beta}(t,x)|\leq 1/2$. 
We have the following:
\begin{lemma}[Theorem 2.1 of \cite{HWY2012Adv}]\label{lemma2.7}
For $0<\mu<1/2$, there exists a positive constant $C$ such that 
the inequality 
\begin{equation}\label{37}
\begin{split}
(1&+T)^{-2\mu}
\left(
\|
r^{-(3/2)+\mu}u
\|^2_{L^2((0,T)\times{\mathbb R}^3)}
+
\|
r^{-(1/2)+\mu}
\partial u
\|^2_{L^2((0,T)\times{\mathbb R}^3)}
\right)\\
&
\leq
C\|\partial u(0,\cdot)\|^2_{L^2({\mathbb R}^3)}\\
&
\hspace{0.4cm}
+C
\int_0^T\!\!\!\int_{{\mathbb R}^3}
\left(
|\partial u||Pu|
+
\frac{|u||Pu|}{r^{1-2\mu}\langle r\rangle^{2\mu}}
+
|\partial h||\partial u|^2 \right.\\
&
\hspace{2.4cm}
\left.
+ 
\frac{|\partial h||u\partial u|}{r^{1-2\mu}\langle r\rangle^{2\mu}}
+ 
\frac{|h||\partial u|^2}{r^{1-2\mu}\langle r\rangle^{2\mu}}
+ 
\frac{|h||u \partial u|}{r^{2 - 2\mu}\langle r\rangle^{2\mu}} 
\right)dxdt
\end{split}
\end{equation}
holds for smooth and compactly supported 
$($for any fixed time$)$ functions $u(t,x)$.
\end{lemma}
See also \cite{MS-SIAM} for an earlier and related estimate. 
The estimate (\ref{37}) was proved by the geometric multiplier method 
of Rodnianski (see Appendix of \cite{Ster}). 
At first sight, the above estimate may appear useless 
for the proof of global existence, 
because of the presence of the factor $(1+T)^{-2\mu}$. 
Combined with Lemma \ref{lemma2.5} 
and the useful idea of dyadic decomposition 
of the time interval (see (\ref{124}) below), 
the estimate (\ref{37}) actually works effectively for the proof of 
global existence with no use of $L_j$ and with limitation of 
the occurrence of $S$ to $1$ in the definition of $N_4(u(t))$. 

The following was proved by Klainerman and Sideris, 
and will be used in the proof of Proposition \ref{proposition1} below. 
By setting $t=0$ in (\ref{eqn:KSineq}), we get the simple inequality 
$M_2(v(0))\leq C_{KS}N_2(v(0))$ which, 
together with Proposition \ref{proposition1}, 
will be used in the proof of Proposition \ref{proposition6} below. 
%%%%%%%%%%%%%%%%%%%%%%%%%%%%%%%
\begin{lemma}[Klainerman-Sideris inequality \cite{KS}]\label{lemma2.8}
%(Klainerman--Sideris inequality)} 
There exists a constant $C_{KS}>0$ such that 
the inequality
\begin{equation}
M_2(v(t))
\leq
C_{KS}
\bigl(
N_2(v(t))
+
t
\|\Box v(t)\|_{L^2({\mathbb R}^3)}
\bigr)\label{eqn:KSineq}
\end{equation}
holds for smooth functions $v=v(t,x)$ 
decaying sufficiently fast as $|x|\to\infty$. 
\end{lemma}
%%%%%%%%%%%%%%%%%%%%%%%%%%%%%%%%%%%%%%%
%%%%%%%%%%%%%%%%%%%%%%%%%%%%%%%%%%%%%%%%
%%%%%%%%%%%%%%%%%%%%%%%%%%%%%%%%%%%%%%%%%
\section{Bound for $M_4(u(t))$}
Since the second order quasi-linear hyperbolic system (\ref{1}) 
can be written in the form of 
the first order quasi-linear symmetric hyperbolic system 
(see, e.g., (5.9) of Racke \cite{Racke}), 
the standard local existence theorem 
(see, e.g., Theorem 5.8 of \cite{Racke}) 
applies to the Cauchy problem for (\ref{1}). 
To begin with, we assume that 
the initial data are smooth, compactly supported, 
and small so that 
\begin{equation}\label{eqn:3a}
\begin{split}
&
\sum_{i=1}^2N_4(u_i(0),\partial_t u_i(0))
\leq
C_D\sum_{i=1}^2D(f_i,g_i)
%%%%%%%%%%%%%%%%%%%%%
%%%%%%%%%%%%%%%%%%%%%%%
\\
\leq&
\varepsilon_0
:=
\min
\biggl\{
\frac{\min\{\varepsilon_1^*,\varepsilon_2^*\}}{2AC_2(1+2A)},
\frac{\varepsilon_3^*}{2AC_2(1+2A)},\\
&
\hspace{1.8cm}
\frac{1}{2C_2C_3(2+3A)(1+2A)},
\frac{1-\frac{4C_0}{9A^2}}{2AC_1C_2(1+2A)+3AC_1}
\biggr\}
\end{split}
\end{equation}
may hold. 
See the inequality following (\ref{8}) for the constant $C_D$. 
See (\ref{eqn:3b}), (\ref{eqn:3f}), and (\ref{eqn:daiji1}) for the constants 
$\varepsilon_1^*$, $\varepsilon_2^*$, and $\varepsilon_3^*$, respectively. 
See (\ref{169}) for $A$, and see (\ref{165}) for $C_0$ and $C_1$. 
See (\ref{172}) and the inequality following it for $C_2$ and $C_3$. 
Note that $\varepsilon_0$ is independent of 
$R_*$ (see Remark \ref{remark1}). 

We know that 
a unique, smooth solution to (\ref{1}) exists 
at least for a short time interval, 
and it is compactly supported at any fixed time 
by the finite speed of propagation. 

Before entering into the energy estimate 
in the next section, 
we must refer to an elementary result concerning 
point-wise estimates for $u_1$ and $u_2$. 
It compensates for the absence of 
$\partial_t^i v(t,x)$ $(i=2,3,4)$ in the definition of 
the norms $N_4(v(t))$, $M_4(v(t))$, $G(v(t))$, and $L(v(t))$ 
(see (\ref{7}), (\ref{27}), (\ref{84})--(\ref{85})). 
%%%%%%%%%%%%%%%%%%%%%%%%%%%%%%%%
%%%%%%%%%%%%%%%%%%%%%%%%%%%%%%%%%
\begin{lemma}\label{lemma3.1}
There exists a constant 
$\varepsilon^*_1>0$ depending on 
the coefficients of $(\ref{1})$ 
with the following property$:$ 
whenever smooth solutions $u=(u_1,u_2)$ to 
$(\ref{1})$ satisfy
\begin{equation}
\max\{\,|\partial_\alpha{\bar Z}^b u_k(t,x)|\,:\,
|b|\leq 1,\,0\leq\alpha\leq 3,\,k=1,2\,\}
\leq
\varepsilon^*_1,
%%%%%%%%%%%%%%%%%%%%%%%%
\label{eqn:3b}
%%%%%%%%%%%%%%%%%%%%%%%%
\end{equation}
the following point-wise inequalities 
$({\rm i})${\rm --}$({\rm iv})$ hold for $i=1,2$. 

\noindent $({\rm i})$ The inequalities
\begin{align}
&
|\partial_t^2 u_i(t,x)|
\leq
C
%\sum_{{1\leq m\leq 3}\atop{0\leq\alpha\leq 3}}
|\partial\partial_x u_i(t,x)|
+
C
\sum_{k=1}^2
%\sum_{\alpha=0}^3
|\partial u_k(t,x)|,
%%%%%%%%%%%%%%%%%%%%%%%%%%
\label{eqn:3c}\\
&
|\partial_t^3 u_i(t,x)|
\leq
C\sum_{|a|=1}^2
|\partial\partial_x^a u_i(t,x)|
+
C\sum_{k=1}^2|\partial u_k(t,x)|
\label{eqn:ad1}
%%%%%%%%%%%%%%%%%%%%%%%%
\end{align}
hold. 

\noindent $({\rm ii})$
There hold 
\begin{equation}\label{eqn:3yokutsukau}
|\partial_t^2{\bar Z}^a u_i(t,x)|
\leq
C
\sum_{|b|\leq |a|}
\biggl(
|\partial\partial_x{\bar Z}^b u_i(t,x)|
+
\sum_{k=1}^2
|\partial{\bar Z}^b u_k(t,x)|
\biggr), \,\,
|a|=1,2,
\end{equation}
%%%%%%%%%%%%%%%%
\begin{equation}\label{eqn:ad5}
\begin{split}
&
|\partial_t^3{\bar Z}^a u_i(t,x)|\\
\leq&
C\sum_{{|b|=2}\atop{|c|\leq 1}}
|\partial\partial_x^b{\bar Z}^c u_i(t,x)|
+
C\sum_{|b|,|c|\leq 1}\sum_{k=1}^2
|\partial\partial_x^b{\bar Z}^c u_k(t,x)|,\,\,|a|=1.
\end{split}
\end{equation}
%%%%%%%%%%%%%%%%%%%%%%%%%%%%%%%%
\noindent $({\rm iii})$ The inequality 
\begin{equation}
|\partial_t^2S u_i(t,x)|
\leq
C
\sum_{d\leq 1}
\biggl(
|\partial\partial_x S^du_i(t,x)|
+
\sum_{k=1}^2
|\partial S^d u_k(t,x)|
\biggr),\,\,i=1,2.
%%%%%%%%%%%%%%%%%%%%%
\label{eqn:3yokutsukau2}
\end{equation}
holds.

%%%%%%%%%%%%%%%%%%%%
\noindent $({\rm iv})$ The inequality
\begin{equation}
|T_j\partial_t^2 u_i(t,x)|
\leq
C|T_j\partial\partial_x u_i(t,x)|
+
C\sum_{k=1}^2|T_j\partial u_k(t,x)|
\label{eqn:ad3}
\end{equation}
holds for $j=1,2,3$. 
Also, for $|a|=1$
\begin{equation}\label{eqn:adad1}
\begin{split}
|T_j\partial_t^2{\bar Z}^a u_i(t,x)|
\leq&
C|T_j\partial\partial_x{\bar Z}^a u_i(t,x)|
+
C\sum_{k=1}^2
\sum_{|b|\leq 1}
|T_j\partial{\bar Z}^b u_k(t,x)|\\
&
+
C\biggl(
\sum_{k=1}^2
T_j\partial u_k(t,x)
\biggr)
\biggl(
\sum_{k=1}^2
\sum_{|b|\leq 1}
|\partial\partial_x{\bar Z}^b u_k(t,x)|
\biggr).
\end{split}
\end{equation}
\end{lemma}
%%%%%%%%%%%%%%%%%%%%%%%%%
We must not assume smallness of 
$|\partial Su_2(t,x)|$ 
(see, e.g., (\ref{54})--(\ref{55}) below, where we allow 
$\|\partial Su_2(t)\|_{L^\infty({\mathbb R}^3)}$ to grow with $t$), 
and therefore we treat point-wise estimates for 
$\partial_t^2{\bar Z}^aS u_i(t,x)$ $(|a|=1)$, 
$\partial_t^3 S u_i(t,x)$, 
and $T_j\partial_t^2 S u_i(t,x)$, separately. 
%%%%%%%%%%%%%%%%%%%%%%%%%%%%%%%%%%%%%%
\begin{lemma}\label{lemma3.2}
There exists a constant 
$\varepsilon^*_2>0$ depending on 
the coefficients of $(\ref{1})$ 
with the following property$:$ 
whenever smooth solutions $u=(u_1,u_2)$ to 
$(\ref{1})$ satisfy
\begin{equation}
\max\{\,|\partial_\alpha{\bar Z}^b u_k(t,x)|\,:\,
|b|\leq 1,\,0\leq\alpha\leq 3,\,k=1,2\,\}
\leq
\varepsilon^*_2,
%%%%%%%%%%%%%%%%%%%%%%%
\label{eqn:3f}
%%%%%%%%%%%%%%%%%%%%%%%
\end{equation}
then the inequality
\begin{equation}\label{eqn:3yokutsukau3}
\begin{split}
&
|\partial_t^2{\bar Z}^a S u_i(t,x)|\\
\leq&
C
\sum_{|b|,\,d\leq 1}
\biggl(
|\partial\partial_x{\bar Z}^b S^d u_i(t,x)|
+
\sum_{k=1}^2
|\partial{\bar Z}^b S^d u_k(t,x)|
\biggr)\\
&
\hspace{0.4cm}
+
C
\biggl(
\sum_{k=1}^2
|\partial S u_k(t,x)|
\biggr)
\biggl(
\sum_{|b|=1}
|\partial\partial_x{\bar Z}^b u_i(t,x)|
\biggr)
\end{split}
\end{equation}
holds for $|a|=1$ and $i=1,2$. 
Also, we have
\begin{equation}\label{eqn:ad6}
\begin{split}
|\partial_t^3Su_i(t,x)|
\leq&
C\sum_{|a|=2}|\partial\partial_x^a Su_i(t,x)|
+
C\sum_{k=1}^2\sum_{{|a|+d\leq 2}\atop{d\leq 1}}
|\partial\partial_x^a S^d u_k(t,x)|\\
&
+
C
\biggl(
\sum_{k=1}^2
|\partial Su_k(t,x)|
\biggr)
\biggl(
\sum_{|a|=1}^2
|\partial\partial_x^a u_i(t,x)|
+
\sum_{k=1}^2
|\partial u_k(t,x)|
\biggr
),
\end{split}
\end{equation}
%%%%%%%%%%%%%%%%%%%%%%%
\begin{equation}\label{eqn:adad2}
\begin{split}
&|T_j\partial_t^2 S u_i(t,x)|\\
\leq&
C\sum_{d\leq 1}
|T_j\partial\partial_x S^d u_i(t,x)|
+
C\sum_{k=1}^2
\sum_{d\leq 1}
|T_j\partial S^d u_k(t,x)|\\
&
+
C
\biggl(
\sum_{k=1}^2
|\partial S u_k(t,x)|
\biggr)
\biggl(
|T_j\partial\partial_x u_i(t,x)|
+
\sum_{k=1}^2
|T_j\partial u_k(t,x)|
\biggr)\\
&
+
C
\biggl(
\sum_{k=1}^2
|T_j\partial u_k(t,x)|
\biggr)
\biggl(
|\partial\partial_x S u_i(t,x)|
+
\sum_{k=1}^2
|\partial S u_k(t,x)|
\biggr).
\end{split}
\end{equation}
%%%%%%%%%%%%%%%%%%%%%%%%%%
\end{lemma}
For the proof of Lemmas \ref{lemma3.1} and \ref{lemma3.2}, 
we have only to repeat essentially the same argument 
as in the proof of Lemma 3.2 of \cite{H2016}. 
We thus omit the proof. 
%%%%%%%%%%%%%%%%%%%%%%%%%%%%%%%%%%%%

Using the above point-wise inequalities, 
let us next consider the bound for 
$M_4(u_1(t))$ and $M_4(u_2(t))$. 
Taking (\ref{eqn:comm1}) into account, 
we have for $|a|+d\leq 3$
\begin{equation}\label{eqn:3basic1}
\begin{split}
&
\Box{\bar Z}^a S^d u_1\\
&
+
G_1^{11,\alpha\beta\gamma}
(\partial_\gamma u_1)
(\partial_{\alpha\beta}^2{\bar Z}^a S^d u_1)
+
G_1^{21,\alpha\beta\gamma}
(\partial_\gamma u_2)
(\partial_{\alpha\beta}^2{\bar Z}^a S^d u_1)\\
&
+
\sum\!{}^{'}
{\tilde G}^{\alpha\beta\gamma}
(\partial_\gamma{\bar Z}^{a'}S^{d'} u_1)
(\partial_{\alpha\beta}^2{\bar Z}^{a''}S^{d''}u_1)\\
&
+
\sum\!{}^{'}
{\hat G}^{\alpha\beta\gamma}
(\partial_\gamma{\bar Z}^{a'}S^{d'}u_2)
(\partial_{\alpha\beta}^2{\bar Z}^{a''}S^{d''}u_1)\\
&
+
\sum\!{}^{''}
{\tilde H}^{\alpha\beta}
(\partial_\alpha{\bar Z}^{a'}S^{d'}u_1)
(\partial_\beta{\bar Z}^{a''}S^{d''}u_1)\\
&
+
\sum\!{}^{''}
{\hat H}^{\alpha\beta}
(\partial_\alpha{\bar Z}^{a'}S^{d'}u_1)
(\partial_\beta{\bar Z}^{a''}S^{d''}u_2)\\
&
+
\sum\!{}^{''}
{\bar H}^{\alpha\beta}
(\partial_\alpha{\bar Z}^{a'}S^{d'}u_2)
(\partial_\beta{\bar Z}^{a''}S^{d''}u_2)
=0.
\end{split}
\end{equation}
Here $\sum\!{}^{'}$ stands for the summation over 
$a'$, $a''$, $d'$ and $d''$ satisfying 
$|a'|+|a''|+d'+d''\leq |a|+d$, 
$|a''|+d''<|a|+d$, 
and $d'+d''\leq d$. 
Similarly, $\sum\!{}^{''}$ stands for the summation over 
$|a'|+|a''|+d'+d''\leq |a|+d$ and $d'+d''\leq d$. 
Just for simplicity of notation, 
we have omitted dependence of the coefficients 
${\tilde G}^{\alpha\beta\gamma}
=
{\tilde G}^{11,\alpha\beta\gamma}_1,\dots,
{\bar H}^{\alpha\beta}={\bar H}^{22,\alpha\beta}_1$ 
on $a'$, $a''$, $d'$ and $d''$. 
Similarly, we have for $|a|+d\leq 3$
\begin{equation}\label{eqn:3basic2}
\begin{split}
&
\Box{\bar Z}^a S^d u_2\\
&
+
G_2^{12,\alpha\beta\gamma}
(\partial_\gamma u_1)
(\partial_{\alpha\beta}^2{\bar Z}^a S^d u_2)
+
G_2^{22,\alpha\beta\gamma}
(\partial_\gamma u_2)
(\partial_{\alpha\beta}^2{\bar Z}^a S^d u_2)\\
&
+
\sum\!{}^{'}
{\tilde G}^{\alpha\beta\gamma}
(\partial_\gamma{\bar Z}^{a'}S^{d'} u_1)
(\partial_{\alpha\beta}^2{\bar Z}^{a''}S^{d''}u_2)\\
&
+
\sum\!{}^{'}
{\hat G}^{\alpha\beta\gamma}
(\partial_\gamma{\bar Z}^{a'}S^{d'}u_2)
(\partial_{\alpha\beta}^2{\bar Z}^{a''}S^{d''}u_2)\\
&
+
\sum\!{}^{''}
{\tilde H}^{\alpha\beta}
(\partial_\alpha{\bar Z}^{a'}S^{d'}u_1)
(\partial_\beta{\bar Z}^{a''}S^{d''}u_2)\\
&
+
\sum\!{}^{''}
{\hat H}^{\alpha\beta}
(\partial_\alpha{\bar Z}^{a'}S^{d'}u_1)
(\partial_\beta{\bar Z}^{a''}S^{d''}u_1)\\
&
+
\sum\!{}^{''}
{\bar H}^{\alpha\beta}
(\partial_\alpha{\bar Z}^{a'}S^{d'}u_2)
(\partial_\beta{\bar Z}^{a''}S^{d''}u_2)
=0.
\end{split}
\end{equation}
Here, ${\tilde G}^{\alpha\beta\gamma}
=
{\tilde G}^{12,\alpha\beta\gamma}_2,\dots,
{\bar H}^{\alpha\beta}={\bar H}^{22,\alpha\beta}_2$. 
%%%%%%%%%%%%%%%%%%%%%%%%%%%%%%%%%%%%%%%%%%%%
In what follows, by $\delta$, $\eta$, and $\mu$, 
we mean sufficiently small positive constants 
such that $\delta<1/9$, $\eta<5/18$, and $\mu<1/4$. 
We use the following quantities for local solutions 
$u=(u_1,u_2)$:
\begin{equation}\label{54}
\begin{split}
&\langle\!\langle u(t)\rangle\!\rangle\\
:=&
\langle\!\langle\!\langle u_1(t)\rangle\!\rangle\!\rangle
+
\sum_{|b|\leq 1}
\||x|\partial{\bar Z}^b u_1(t)\|_{L^\infty({\mathbb R}^3)}
+
\sum_{|b|\leq 2}
\||x|\partial{\bar Z}^b u_1(t)\|_{L^\infty_r L^4_\omega({\mathbb R}^3)}\\
&
+
(1+t)^{-\delta}
\langle\!\langle\!\langle u_2(t)\rangle\!\rangle\!\rangle,
\end{split}
\end{equation}
where for a scalar function $w(t,x)$ 
\begin{equation}\label{55}
\begin{split}
&\langle\!\langle\!\langle w(t)\rangle\!\rangle\!\rangle\\
:=&
(1+t)
\sum_{|b|\leq 1}
\|\partial{\bar Z}^b w(t)\|_{L^\infty({\mathbb R}^3)}
+
\sum_{|b|\leq 1}
\|
|x|\langle t-r\rangle^{1/2}
\partial{\bar Z}^b w(t)
\|_{L^\infty({\mathbb R}^3)}\\
&
+
\sum_{|b|\leq 1}
\|
|x|^{1/2}\partial{\bar Z}^b w(t)
\|_{L^\infty({\mathbb R}^3)}
+
\sum_{|b|\leq 2}
\|
|x|^{1/2}\partial{\bar Z}^b w(t)
\|_{L^\infty_r L^4_\omega({\mathbb R}^3)}\\
&
+
\sum_{|b|\leq 2}
\|
|x|^{1/2}{\bar Z}^b w(t)
\|_{L^\infty({\mathbb R}^3)}
+
\sum_{|b|\leq 3}
\|
|x|^{1/2}{\bar Z}^b w(t)
\|_{L^\infty_r L^4_\omega({\mathbb R}^3)}\\
&
+
\sum_{|b|\leq 1}
\|
|x|^{1/2}S{\bar Z}^b w(t)
\|_{L^\infty({\mathbb R}^3)}
+
\sum_{|b|\leq 2}
\|
|x|^{1/2}S{\bar Z}^b w(t)
\|_{L^\infty_r L^4_\omega({\mathbb R}^3)}\\
&
+
\sum_{|b|\leq 1}
\|
|x|^{1/2}\langle t-r\rangle\partial{\bar Z}^b w(t)
\|_{L^\infty({\mathbb R}^3)}\\
&
+
\sum_{|b|\leq 2}
\|
|x|^{1/2}\langle t-r\rangle\partial{\bar Z}^b w(t)
\|_{L^\infty_r L^4_\omega({\mathbb R}^3)}\\
&
+
\sum_{i=1}^2\sum_{|b|\leq 1}
\|
|x|^{(1/2)+\theta_i}\langle t-r\rangle^{1-\theta_i}\partial{\bar Z}^b w(t)
\|_{L^\infty({\mathbb R}^3)}\\
&
+
\sum_{i=1}^2\sum_{|b|\leq 2}
\|
|x|^{(1/2)+\theta_i}\langle t-r\rangle^{1-\theta_i}\partial{\bar Z}^b w(t)
\|_{L^\infty_r L^4_\omega({\mathbb R}^3)}\\
&
+
\|
\partial Sw(t)
\|_{L^\infty({\mathbb R}^3)}
+
\|
\langle r\rangle\partial Sw(t)
\|_{L^\infty({\mathbb R}^3)}
+
\sum_{|b|\leq 1}
\|
\langle r\rangle\partial{\bar Z}^b Sw(t)
\|_{L^\infty_r L^4_\omega({\mathbb R}^3)},
\end{split}
\end{equation}
where $\theta_1:=(1/2)-2\mu$, $\theta_2:=(1/2)-\eta$, 
%%%%%%%%%%%%%%%%%%%%%%%
\begin{align}
&
{\mathcal N}(u(t))
:=
N_4(u_1(t))
+
\langle t\rangle^{-\delta}
N_4(u_2(t)),
\label{eqn:SN1}
\\
&
{\mathcal M}(u(t))
:=
M_4(u_1(t))
+
\langle t\rangle^{-\delta}
M_4(u_2(t)).
\label{eqn:adadScM}
\end{align}

%%%%%%%%%%%%%%%%%%%%%%%%%%%%%%%%%%%%%%%%%
\begin{proposition}\label{proposition1}
Smooth local solutions $u=(u_1,u_2)$ to $(\ref{1})$ 
defined in $(0,T)\times{\mathbb R}^3$ for some $T>0$ satisfy 
the inequality 
\begin{equation}
{\mathcal M}(u(t))
\leq
C_{KS}
{\mathcal N}(u(t))
+
C
\langle\!\langle u(t)\rangle\!\rangle
\bigl(
{\mathcal M}(u(t))+{\mathcal N}(u(t))
\bigr)
\label{eqn:nov1}
\end{equation}
for every $t\in(0,T)$, 
provided that 
they satisfy 
$$
\sup_{0<t<T}
\langle\!\langle 
u(t)
\rangle\!\rangle
\leq
\min\{\varepsilon^*_1,\,\varepsilon^*_2\}.
$$ 
\end{proposition}
For $\varepsilon^*_1$ and $\varepsilon^*_2$, 
see Lemmas \ref{lemma3.1} and \ref{lemma3.2}. 
%%%%%%%%%%%%%%%%%%%%%%%%
%%%%%%%%%%%%%%%%%%%%%%%%%
\begin{corollary}\label{corollary1}
There exists a small constant $\varepsilon^*_3$ 
$(0<\varepsilon^*_3<\min\{\varepsilon^*_1,\,\varepsilon^*_2\})$ 
such that 
as long as smooth local solutions $u=(u_1,u_2)$ 
satisfy 
\begin{equation}
\langle\!\langle u(t)\rangle\!\rangle
\leq
\varepsilon^*_3,
\label{eqn:daiji1}
\end{equation}
the estimate
\begin{equation}
M_4(u_1(t)),\,
\langle t\rangle^{-\delta}M_4(u_2(t))
\leq
C{\mathcal N}(u(t))
\label{eqn:mn}
\end{equation}
holds.
\end{corollary}
%%%%%%%%%%%%%%%%%%%%%%
\begin{proof}
We prove Proposition \ref{proposition1}. 
Corollary \ref{corollary1} is an immediate consequence of it, 
because 
$
C\langle\!\langle u(t)\rangle\!\rangle
{\mathcal M}(u(t))
$
can be absorbed into the left-hand side of (\ref{eqn:nov1}) 
for small $\langle\!\langle u(t)\rangle\!\rangle$. 
We use (\ref{eqn:3basic1}), (\ref{eqn:3basic2}) with 
$|a|\leq 2$, $d=0$. 
Obviously, it suffices to explain how to bound 
$M_2({\bar Z}^a u_i(t))$ for 
$|a|=2$, $i=1,2$. 

We first bound $M_2({\bar Z}^a u_1(t))$. 
In view of the Klainerman-Sideris inequality (\ref{eqn:KSineq}), 
our task is to bound the $L^2({\mathbb R}^3)$ norm of 
the 2nd, 3rd, $\dots$, and 8th terms 
on the left-hand side of (\ref{eqn:3basic1}) for 
$|a|=2$, $d=0$. 
In fact, it is enough to bound 
the 5th and 8th terms for $|a'|+|a''|=2$ 
because the others can be handled similarly. 
For any fixed $t\in (0,T)$, 
we bound their $L^2$ norm 
over the set 
$\{x\in{\mathbb R}^3\,:\,
|x|\leq (t+1)/2\}$ 
and its complement, separately. 
Let $\chi_1(x)$ be the characteristic function of 
this set, 
and we set $\chi_2(x):=1-\chi_1(x)$. 
Recall that we now have 
$|a'|+|a''|\leq 2$ 
at the $5$th term on the left-hand side of (\ref{eqn:3basic1}). 
For $|a'|\leq 1$, we get by (\ref{eqn:3yokutsukau}) 
\begin{equation}\label{61}
\begin{split}
&\|
\chi_1
(\partial{\bar Z}^{a'}u_2(t))
(\partial^2{\bar Z}^{a''}u_1(t))
\|_{L^2({\mathbb R}^3)}\\
\leq&
C\sum_{|b|\leq |a''|}
\langle t\rangle^{-2+\delta}
\bigl(
\langle t\rangle^{1-\delta}
\|
\partial{\bar Z}^{a'}u_2(t)
\|_{L^\infty}
\bigr)
\|
\langle
t-r
\rangle
\partial\partial_x{\bar Z}^b u_1(t)
\|_{L^2}\\
&
+
C
\sum_{{|b|\leq |a''|}\atop{k=1,2}}
\langle t\rangle^{-(3/2)+2\delta}
\bigl(
\langle t\rangle^{-\delta}
\|
|x|\langle t-r\rangle^{1/2}
\partial{\bar Z}^{a'}u_2(t)
\|_{L^\infty}
\bigr)\\
&
\hspace{3.7cm}
\times
\langle t\rangle^{-\delta}
\|
|x|^{-1}\langle t-r\rangle
\partial{\bar Z}^b u_k(t)
\|_{L^2}\\
\leq&
C
\langle t\rangle^{-(3/2)+2\delta}
\langle\!\langle u(t)\rangle\!\rangle
\bigl(
{\mathcal M}(u(t))
+
{\mathcal N}(u(t))
\bigr),
\end{split}
\end{equation}
where we have used the Hardy inequality at the last inequality. 
For $|a'|=2$ (therefore $|a''|=0$), 
we get by (\ref{eqn:3c}) and the Hardy inequality
\begin{equation}\label{62}
\begin{split}
&\|
\chi_1
(\partial{\bar Z}^{a'}u_2(t))
(\partial^2 u_1(t))
\|_{L^2({\mathbb R}^3)}\\
\leq&
C
\langle t\rangle^{-(3/2)+\delta}
\bigl(
\langle t\rangle^{-\delta}
\|
|x|^{-1}\langle t-r\rangle
\partial{\bar Z}^{a'}u_2(t)
\|_{L^2}
\bigr)
\|
|x|\langle t-r\rangle^{1/2}
\partial\partial_x u_1(t)
\|_{L^\infty}\\
&+
C
\langle t\rangle^{-(3/2)+2\delta}
\bigl(
\langle t\rangle^{-\delta}
\|
|x|^{-1}\langle t-r\rangle
\partial{\bar Z}^{a'}u_2(t)
\|_{L^2}
\bigr)\\
&
\hspace{2.0cm}
\times
\bigl(
\|
|x|\langle t-r\rangle^{1/2}
\partial u_1(t)
\|_{L^\infty}
+
\langle t\rangle^{-\delta}
\|
|x|\langle t-r\rangle^{1/2}
\partial u_2(t)
\|_{L^\infty}
\bigr)\\
\leq&
C
\langle t\rangle^{-(3/2)+2\delta}
\langle\!\langle u(t)\rangle\!\rangle
\bigl(
{\mathcal M}(u(t))
+
{\mathcal N}(u(t))
\bigr).
\end{split}
\end{equation}
For the 8th term on the left-hand side of (\ref{eqn:3basic1}), 
we get, 
assuming $|a'|\leq |a''|$ 
(therefore, $|a'|\leq 1$) 
without loss of generality
\begin{equation}\label{63}
\begin{split}
&\|
\chi_1
(\partial{\bar Z}^{a'}u_2(t))
(\partial{\bar Z}^{a''}u_2(t))
\|_{L^2}\\
\leq&
C
\langle t\rangle^{-(3/2)+2\delta}
\bigl(
\langle t\rangle^{-\delta}
\|
|x|\langle t-r\rangle^{1/2}
\partial{\bar Z}^{a'}u_2(t)
\|_{L^\infty}
\bigr)\\
&
\hspace{2.0cm}
\times
\bigl(
\langle t\rangle^{-\delta}
\|
|x|^{-1}\langle t-r\rangle
\partial{\bar Z}^{a''}u_2(t)
\|_{L^2}
\bigr)\\
\leq&
C
\langle t\rangle^{-(3/2)+2\delta}
\langle\!\langle u(t)\rangle\!\rangle
\bigl(
\langle t\rangle^{-\delta}M_4(u_2(t))
+
\langle t\rangle^{-\delta}N_4(u_2(t))
\bigr)
\end{split}
\end{equation}
in the same way as above. 

Next, let us consider the estimate over the set 
$\{x\in{\mathbb R}^3:|x|>(t+1)/2\}$ 
for any fixed $t\in(0,T)$. 
Recall that 
$\chi_2(x)=1-\chi_1(x)$. 
Since the coefficients 
${\hat G}^{\alpha\beta\gamma}$ 
satisfy the null condition 
thanks to Lemma \ref{lemma2.1}, 
we can first use Lemma \ref{lemma2.2} to get
\begin{equation}\label{64}
\begin{split}
&\|
\chi_2{\hat G}^{\alpha\beta\gamma}
(\partial_\gamma{\bar Z}^{a'}u_2(t))
(\partial_{\alpha\beta}^2{\bar Z}^{a''}u_1(t))
\|_{L^2({\mathbb R}^3)}\\
\leq&
C
\bigl(
\|
\chi_2
(T{\bar Z}^{a'}u_2(t))
(\partial^2{\bar Z}^{a''}u_1(t))
\|_{L^2}
+
\|
\chi_2
(\partial{\bar Z}^{a'}u_2(t))
(T\partial{\bar Z}^{a''}u_1(t))
\|_{L^2}
\bigr)
\end{split}
\end{equation}
and then we use Lemma \ref{lemma2.3} to get
\begin{equation}\label{65}
\begin{split}
&\|
\chi_2
(T{\bar Z}^{a'}u_2(t))
(\partial^2{\bar Z}^{a''}u_1(t))
\|_{L^2}\\
\leq&
C\langle t\rangle^{-1}
\bigl(
\|
\chi_2
(\partial{\bar Z}^{a'}u_2(t))
(\partial^2{\bar Z}^{a''}u_1(t))
\|_{L^2}\\
&\hspace{1.2cm}
+
\sum_{|b|=1}
\|
\chi_2
(\Omega^b{\bar Z}^{a'}u_2(t))
(\partial^2{\bar Z}^{a''}u_1(t))
\|_{L^2}\\
&
\hspace{1.2cm}
+
\|
\chi_2
(S{\bar Z}^{a'}u_2(t))
(\partial^2{\bar Z}^{a''}u_1(t))
\|_{L^2}\\
&
\hspace{1.2cm}
+
\|
\chi_2
\langle t-r\rangle
(\partial_x{\bar Z}^{a'}u_2(t))
(\partial^2{\bar Z}^{a''}u_1(t))
\|_{L^2}
\bigr),
\end{split}
\end{equation}
%%%%%%%%%%%%%%%%%%%%%%%%%%%%%%%%%5
%%%%%%%%%%%%%%%%%%%%%%%%%%%%%%%
\begin{equation}\label{66}
\begin{split}
&\|
\chi_2
(\partial{\bar Z}^{a'}u_2(t))
(T\partial{\bar Z}^{a''}u_1(t))
\|_{L^2}\\
\leq&
C\langle t\rangle^{-1}
\bigl(
\|
\chi_2
(\partial{\bar Z}^{a'}u_2(t))
(\partial^2{\bar Z}^{a''}u_1(t))
\|_{L^2}\\
&
\hspace{1.2cm}
+
\sum_{|b|\leq 1}
\|
\chi_2
(\partial{\bar Z}^{a'}u_2(t))
(\partial\Omega^b{\bar Z}^{a''}u_1(t))
\|_{L^2}\\
&
\hspace{1.2cm}
+
\sum_{d\leq 1}
\|
\chi_2
(\partial{\bar Z}^{a'}u_2(t))
(\partial S^d{\bar Z}^{a''}u_1(t))
\|_{L^2}\\
&
\hspace{1.2cm}
+
\|
\chi_2
\langle t-r\rangle
(\partial{\bar Z}^{a'}u_2(t))
(\partial\partial_x{\bar Z}^{a''}u_1(t))
\|_{L^2}
\bigr).
\end{split}
\end{equation}
It suffices to show how to treat 
the 3rd and 4th terms on the right-hand side of 
the second last inequality, 
because 
the other terms can be estimated in a similar way. 
Recall that we are assuming 
$|a'|+|a''|\leq 2$, $|a''|\leq 1$. 
If $|a'|\leq 1$, then 
we get 
\begin{equation}\label{67}
\begin{split}
&\|
\chi_2
(S{\bar Z}^{a'}u_2(t))
(\partial^2{\bar Z}^{a''}u_1(t))
\|_{L^2}\\
\leq&
C
\langle t\rangle^{-1/2}
\|
|x|^{1/2}S{\bar Z}^{a'}u_2(t)
\|_{L^\infty}\\
&
\hspace{1.5cm}
\times
\sum_{|b|\leq |a''|}
\bigl(
\|
\partial\partial_x{\bar Z}^bu_1(t)
\|_{L^2}
+\sum_{k=1,2}
\|
\partial{\bar Z}^b u_k(t)
\|_{L^2}
\bigr)\\
\leq&
C
\langle t\rangle^{-(1/2)+2\delta}
\langle\!\langle u(t)\rangle\!\rangle
{\mathcal N}(u(t)).
\end{split}
\end{equation}
If $|a'|=2$ (hence $|a''|=0$), 
then we get 
\begin{equation}\label{68}
\begin{split}
&\|
\chi_2
(S{\bar Z}^{a'}u_2(t))
(\partial^2 u_1(t))
\|_{L^2}\\
\leq&
C
\langle t\rangle^{-1/2}
\|
|x|^{1/2}S{\bar Z}^{a'}u_2(t)
\|_{L^\infty_r L^4_\omega}
\bigl(
\|
\partial\partial_x u_1(t)
\|_{L^2_r L^4_\omega}
+\sum_{k=1,2}
\|
\partial u_k(t)
\|_{L^2_r L^4_\omega}
\bigr)\\
\leq&
C
\langle t\rangle^{-(1/2)+2\delta}
\langle\!\langle u(t)\rangle\!\rangle
{\mathcal N}(u(t)).
\end{split}
\end{equation}
If $|a'|\leq 1$, then we obtain
\begin{equation}\label{69}
\begin{split}
&\|
\chi_2
\langle t-r\rangle
(\partial_x {\bar Z}^{a'}u_2(t))
(\partial^2{\bar Z}^{a''}u_1(t))
\|_{L^2}\\
\leq&
C
\langle t\rangle^{-1/2}
\|
|x|^{1/2}\langle t-r\rangle
\partial_x{\bar Z}^{a'}u_2(t)
\|_{L^\infty}\\
&
\hspace{1.3cm}
\times
\sum_{|b|\leq |a''|}
\bigl(
\|\partial\partial_x{\bar Z}^b u_1(t)\|_{L^2}
+
\sum_{k=1,2}
\|\partial{\bar Z}^b u_k(t)\|_{L^2}
\bigr)\\
\leq&
C
\langle t\rangle^{-(1/2)+2\delta}
\langle\!\langle u(t)\rangle\!\rangle
{\mathcal N}(u(t)).
\end{split}
\end{equation}
If $|a'|=2$, then we use the $L^4_\omega$ norm as above to obtain
\begin{equation}\label{70}
\begin{split}
&\|
\chi_2
\langle t-r\rangle
(\partial_x {\bar Z}^{a'}u_2(t))
(\partial^2 u_1(t))
\|_{L^2}\\
\leq&
C
\langle t\rangle^{-1/2}
\|
|x|^{1/2}\langle t-r\rangle
\partial_x{\bar Z}^{a'}u_2(t)
\|_{L^\infty_r L^4_\omega}\\
&
\hspace{1.3cm}
\times
\bigl(
\|\partial\partial_x u_1(t)\|_{L^2_r L^4_\omega}
+
\sum_{k=1,2}
\|\partial u_k(t)\|_{L^2_r L^4_\omega}
\bigr)\\
\leq&
C
\langle t\rangle^{-(1/2)+2\delta}
\langle\!\langle u(t)\rangle\!\rangle
{\mathcal N}(u(t)).
\end{split}
\end{equation}
We thus conclude that 
\begin{equation}\label{71}
\|
\chi_2
{\hat G}^{\alpha\beta\gamma}
(\partial_\gamma{\bar Z}^{a'} u_2(t))
(\partial^2_{\alpha\beta}{\bar Z}^{a''}u_1(t))
\|_{L^2({\mathbb R}^3)}
\leq
C
\langle t\rangle^{-(3/2)+2\delta}
\langle\!\langle u(t)\rangle\!\rangle
{\mathcal N}(u(t)).
\end{equation}
Similarly, 
the coefficients ${\bar H}^{\alpha\beta}$ 
satisfy the null condition and thus we can use (\ref{eqn:null4}) to get 
the inequality
\begin{equation}\label{72}
\|
\chi_2
{\bar H}^{\alpha\beta}
(\partial_\alpha{\bar Z}^{a'} u_2(t))
(\partial_\beta{\bar Z}^{a''}u_2(t))
\|_{L^2({\mathbb R}^3)}
\leq
C
\langle t\rangle^{-(3/2)+2\delta}
\langle\!\langle u(t)\rangle\!\rangle
{\mathcal N}(u(t)).
\end{equation}
for $|a'|+|a''|\leq 2$ in the same way as above. 

Let us turn our attention to 
the bound for $M_2({\bar Z}^a u_2(t))$, $|a|\leq 2$. 
Naturally, 
we may focus on the 2nd, 4th, 6th, and 7th terms 
on the left-hand side of (\ref{eqn:3basic2}) 
whose coefficients do not necessarily satisfy the null condition. 
We will show how to treat the 4th and 6th terms, 
because the 2nd and 7th terms can be handled similarly. 
For the 3rd , 5th, and 8th terms whose coefficients 
satisfy the null condition, 
we have only to proceed as 
we did in the treatment of $M_2({\bar Z}^a u_1(t))$, 
thus we may omit the details. 

Let us resume with the estimate of the 4th and 6th terms. 
Recall that Lemma \ref{lemma2.2} has played no role in 
(\ref{61})--(\ref{63}) 
and it has played an essential role in 
(\ref{64})--(\ref{72}). 
Since we can no longer use Lemma \ref{lemma2.2}, 
our task is to consider their bound over the set 
$\{x\in{\mathbb R}^3:|x|>(t+1)/2\}$. 
If $|a'|\leq 1$, then we get by (\ref{eqn:3yokutsukau})
\begin{equation}\label{73}
\begin{split}
&\|
\chi_2
(\partial{\bar Z}^{a'}u_1(t))
(\partial^2{\bar Z}^{a''}u_2(t))
\|_{L^2}\\
\leq&
C
\langle t\rangle^{-1}
\|
|x|\partial{\bar Z}^{a'}u_1(t)
\|_{L^\infty}\\
&
\hspace{1.1cm}
\times
\sum_{|b|\leq |a''|}
\bigl(
\|
\partial\partial_x{\bar Z}^b u_2(t)
\|_{L^2}
+
\sum_{k=1,2}
\|
\partial{\bar Z}^b u_k(t)
\|_{L^2}
\bigr)\\
\leq&
C
\langle t\rangle^{-1+\delta}
\langle\!\langle u(t)\rangle\!\rangle
{\mathcal N}(u(t)).
\end{split}
\end{equation}
%%%%%%%%%%%%%%%%%%%%%%%%%%%%
If $|a'|=2$, then we get by (\ref{eqn:3c})
\begin{equation}\label{74}
\begin{split}
&\|
\chi_2
(\partial{\bar Z}^{a'}u_1(t))
(\partial^2 u_2(t))
\|_{L^2}\\
\leq&
C
\langle t\rangle^{-1}
\|
|x|\partial{\bar Z}^{a'}u_1(t)
\|_{L^\infty_r L^4_\omega}
\bigl(
\|
\partial\partial_x u_2(t)
\|_{L^2_r L^4_\omega}
+
\sum_{k=1,2}
\|
\partial u_k(t)
\|_{L^2_r L^4_\omega}
\bigr)\\
\leq&
C
\langle t\rangle^{-1+\delta}
\langle\!\langle u(t)\rangle\!\rangle
{\mathcal N}(u(t)).
\end{split}
\end{equation}
Similarly, we obtain
\begin{equation}\label{75}
\|
\chi_2
(\partial{\bar Z}^{a'}u_1(t))
(\partial{\bar Z}^{a''} u_2(t))
\|_{L^2}
\leq
C
\langle t\rangle^{-1+\delta}
\langle\!\langle u(t)\rangle\!\rangle
{\mathcal N}(u(t)).
\end{equation}
Summing up, 
we have finished the proof of Proposition \ref{proposition1}. 
\end{proof}
%%%%%%%%%%%%%%%%%%%%%%%%%%%%%%%%%%%%%%%%%%
%%%%%%%%%%%%%%%%%%%%%%%%%%%%%%%%%%%%%%%
%%%%%%%%%%%%%%%%%%%%%%%%%%%%%%%%%%%%%%%
\section{Energy estimate for $u_1$}
From now on, 
we focus on the energy estimate of the highest order 
$|a|+d=3$; 
the energy estimate of the lower order is easier. 
Following the argument in page 93 of \cite{Al2010}, we obtain 
for the function $g=g(t-r)$ chosen below (see (\ref{91}))
\begin{equation}\label{76}
\begin{split}
&\frac12
\partial_t
\bigl\{
e^g
\bigl(
(\partial_t{\bar Z}^{a}S^{d}u_1)^2
+
|\nabla{\bar Z}^{a}S^{d}u_1|^2\\
&
\hspace{1.2cm}
-
G_1^{i1,\alpha\beta\gamma}
(\partial_\gamma u_i)
(\partial_\beta{\bar Z}^{a}S^{d}u_1)
(\partial_\alpha{\bar Z}^{a}S^{d}u_1)\\
&
\hspace{1.2cm}
+2
G_1^{i1,0\beta\gamma}
(\partial_\gamma u_i)
(\partial_\beta{\bar Z}^{a}S^{d}u_1)
(\partial_t{\bar Z}^{a}S^{d}u_1)
\bigr)
\bigr\}\\
&
+\nabla\cdot\{\cdots\}
+
e^gq
+
e^g(J_{1,1}+J_{1,2}+\cdots+J_{1,5})=0,
\end{split}
\end{equation}
%%%%%%%%%%%%%%%%%%%%%%%%%%%%%%
where, and later on as well, summation over 
the repeated index $i$ is assumed from 1 to 2. 
%%%%%%%%%%%%%%%%%%%%%%%%%%%
Here, $q=q_1-(1/2)g'(t-r)q_2$, 
\begin{equation}\label{eqn:q1}
\begin{split}
q_1=&
\frac12
G_1^{i1,\alpha\beta\gamma}
(\partial_{t\gamma}^2 u_i)
(\partial_\beta{\bar Z}^{a}S^{d}u_1)
(\partial_\alpha{\bar Z}^{a}S^{d}u_1)\\
&-
G_1^{i1,\alpha\beta\gamma}
(\partial_{\alpha\gamma}^2 u_i)
(\partial_\beta{\bar Z}^{a}S^{d}u_1)
(\partial_t{\bar Z}^{a}S^{d}u_1),
\end{split}
\end{equation}
%%%%%%%%%%%%%%%%%%%%%
%%%%%%%%%%%%%%%%%%%%%%
\begin{equation}\label{eqn:q2}
\begin{split}
q_2
=
&\sum_{j=1}^3
(T_j{\bar Z}^{a}S^{d}u_1)^2
-
G_1^{i1,\alpha\beta\gamma}
(\partial_{\gamma} u_i)
(\partial_\beta{\bar Z}^{a}S^{d}u_1)
(\partial_\alpha{\bar Z}^{a}S^{d}u_1)\\
&+
2
G_1^{i1,\alpha\beta\gamma}
(\partial_{\gamma} u_i)
(\partial_\beta{\bar Z}^{a}S^{d}u_1)
(-\omega_\alpha)
(\partial_t{\bar Z}^{a}S^{d}u_1)
\end{split}
\end{equation}
%%%%%%%%%%%%%%%%%%%%%%%%%%%%%%%%%%%%%%%%%
where, as explained in Lemma \ref{lemma2.2} above, 
$\omega_0=-1$, $\omega_k=x_k/|x|$, $k=1,2,3$. 
Also, (see (\ref{eqn:3basic1}) for $\sum\!{}^{'}$, $\sum\!{}^{''}$)
%%%%%%%%%%%%%%%%%%%%%%%%%%%%%%%%%%%%%5
\begin{align}
&
J_{1,1}
=
\sum\!{}^{'}
{\tilde G}^{\alpha\beta\gamma}
(\partial_\gamma{\bar Z}^{a'}S^{d'}u_1)
(\partial_{\alpha\beta}^2{\bar Z}^{a''}S^{d''}u_1)
(\partial_t{\bar Z}^{a}S^{d}u_1),\label{79}\\
&
J_{1,2}
=
\sum\!{}^{'}
{\hat G}^{\alpha\beta\gamma}
(\partial_\gamma{\bar Z}^{a'}S^{d'}u_2)
(\partial_{\alpha\beta}^2{\bar Z}^{a''}S^{d''}u_1)
(\partial_t{\bar Z}^{a}S^{d}u_1),\label{80}\\
&
J_{1,3}
=
\sum\!{}^{''}
{\tilde H}^{\alpha\beta}
(\partial_\alpha{\bar Z}^{a'}S^{d'}u_1)
(\partial_\beta{\bar Z}^{a''}S^{d''}u_1)
(\partial_t{\bar Z}^{a}S^{d}u_1),\label{81}\\
&
J_{1,4}
=
\sum\!{}^{''}
{\hat H}^{\alpha\beta}
(\partial_\alpha{\bar Z}^{a'}S^{d'}u_1)
(\partial_\beta{\bar Z}^{a''}S^{d''}u_2)
(\partial_t{\bar Z}^{a}S^{d}u_1),\label{82}\\
&
J_{1,5}
=
\sum\!{}^{''}
{\bar H}^{\alpha\beta}
(\partial_\alpha{\bar Z}^{a'}S^{d'}u_2)
(\partial_\beta{\bar Z}^{a''}S^{d''}u_2)
(\partial_t{\bar Z}^{a}S^{d}u_1).\label{83}
\end{align}
%%%%%%%%%%%%%%%%%%%%%%%%%%%%%%%%%%%%%%%%
%%%%%%%%%%%%%%%%%%%%%%%%%%%%%%%%%%%%%%%
In the following, we use the following 
$G(v(t))$ and $L(v(t))$ (recall $\eta<5/18$, $\mu<1/4$), 
which are related to the ghost energy 
and localized energy, respectively:
\begin{equation}\label{84}
\begin{split}
&
G(v(t))
:=
\biggl\{
\sum_{j=1}^3
\biggl(
\sum_{{|a|+d\leq 3}\atop{d\leq 1}}
\|
\langle t-r\rangle^{-(1/2)-\eta}
T_j{\bar Z}^aS^d v(t)
\|_{L^2({\mathbb R}^3)}^2\\
&
\hspace{2.8cm}
+
\sum_{{|a|+d\leq 2}\atop{d\leq 1}}
\|
\langle t-r\rangle^{-(1/2)-\eta}
T_j\partial_t{\bar Z}^aS^d v(t)
\|_{L^2({\mathbb R}^3)}^2
\biggr)
\biggr\}^{1/2},
\end{split}
\end{equation}
%%%%%%%%%%%%%%%
%%%%%%%%%%%%%%%%%%
\begin{equation}\label{85}
\begin{split}
&
L(v(t)):=
\biggl\{
\sum_{{|a|+d\leq 3}\atop{d\leq 1}}
\biggl(
\|
r^{-(3/2)+\mu}
{\bar Z}^a S^d v(t)
\|_{L^2({\mathbb R}^3)}^2\\
&
\hspace{3.2cm}
+
\|
r^{-(1/2)+\mu}
\partial{\bar Z}^a S^d v(t)
\|_{L^2({\mathbb R}^3)}^2
\biggr)
\biggr\}^{1/2}.
\end{split}
\end{equation}
We remark that 
the norm 
$\|
\langle t-r\rangle^{-(1/2)-\eta}
T_j\partial_t{\bar Z}^aS^d v(t)
\|_{L^2({\mathbb R}^3)}$ 
$(|a|+d\leq 2,\,d\leq 1)$, 
which requires a separate and careful treatment, 
naturally comes up later. 
See, e.g., (\ref{113}) below. 
Just for simplicity, we denote for local solutions 
$u=(u_1,u_2)$, 
$N(u_k(t)):=N_4(u_k(t))$ and $M(u_k(t)):=M_4(u_k(t))$. 
Also, we use the notation (recall $\delta<1/3$)
\begin{align}
&{\mathcal N}_T(u):=
\sup_{0<t<T}{\mathcal N}(u(t)),\label{86}\\
&{\mathcal G}_T(u):=
\biggl(
\int_0^T
G(u_1(t))^2dt
\biggr)^{1/2}
+
\sup_{0<t<T}
\langle t\rangle^{-\delta}
\biggl(
\int_0^t
G(u_2(\tau))^2d\tau
\biggr)^{1/2},\label{87}
\end{align}
%%%%%%%%%%%%%%%%%%%%%%%%%%%%%%%%%%
\begin{equation}\label{88}
\begin{split}
{\mathcal L}_T(u)
:=&
\sup_{0<t<T}
\langle t\rangle^{-\mu-\delta}
\biggl(
\int_0^t
L(u_1(\tau))^2d\tau
\biggr)^{1/2}\\
&+
\sup_{0<t<T}
\langle t\rangle^{-\mu-(3\delta/2)}
\biggl(
\int_0^t
L(u_2(\tau))^2d\tau
\biggr)^{1/2}.
\end{split}
\end{equation}
%%%%%%%%%%%%%%%%%%%%%%%%%%%%%%%%%
The purpose of this section is to show the following:
\begin{proposition}\label{proposition2}
The following inequality holds 
for smooth local solutions to $(\ref{1})$ 
$u=(u_1,u_2)$, 
as long as they satisfy $(\ref{eqn:daiji1})$ for some time interval 
$(0,T):$
\begin{equation}\label{89}
\begin{split}
&\sup_{0<t<T}N(u_1(t))^2
+
\int_0^T G(u_1(t))^2 dt\\
\leq&
CN(u_1(0))^2
+
C\sup_{0<t<T}
\langle\!\langle u(t)\rangle\!\rangle
\bigl(
{\mathcal N}_T(u)^2
+
{\mathcal G}_T(u)^2
+
{\mathcal L}_T(u)^2
\bigr)\\
&+
C{\mathcal N}_T(u)^3
+
C{\mathcal G}_T(u){\mathcal N}_T(u)^2.
\end{split}
\end{equation}
\end{proposition}
%%%%%%%%%%%%%%%%%%%%%%%%%%%%%%%%%%%%%
The rest of this section is devoted to the proof of this proposition. 
We postpone the proof of 
the required estimate for 
$\|
\langle t-r\rangle^{-(1/2)-\eta}
T_j\partial_t{\bar Z}^aS^d u_1
\|_{L^2((0,T)\times{\mathbb R}^3)}$, 
$|a|+d\leq 2,\,d\leq 1$ 
(see (\ref{84})) until the end of this section 
because it should be treated separately. 

For any fixed time $t$, 
estimates are carried out 
over the set 
$\{x\in{\mathbb R}^3\,|\,|x|<(t+1)/2\}$ 
and its complement, separately. 
We thus use the functions 
$\chi_1(x)$ and $\chi_2(x)$ again. 

\noindent
{\bf 
Estimate over \boldmath$\{x\in{\mathbb R}^3\,|\,|x|<(t+1)/2\}$
}.\,\,
Recall that $q=q_1-(1/2)g'(t-r)q_2$ 
for $q_1$ and $q_2$ defined in (\ref{eqn:q1}), (\ref{eqn:q2}). 

\noindent{\bf $\cdot$ Estimate of \boldmath$\chi_1 q$}. 
Recall $\theta_1:=(1/2)-2\mu$ (see (\ref{55})). 
Using (\ref{eqn:3c}), we estimate $\chi_1q_1$ as follows:
\begin{equation}\label{90}
\begin{split}
&\|
\chi_1
(\partial^2 u_i(t))
(\partial{\bar Z}^{a}S^{d}u_1(t))
(\partial{\bar Z}^{a}S^{d}u_1(t))
\|_{L^1({\mathbb R}^3)}\\
\leq&
C
\langle t\rangle^{-1+\theta_1}
\|
r^{(1/2)+\theta_1}
\langle t-r\rangle^{1-\theta_1}
\partial^2 u_i(t)
\|_{L^\infty}
\|
r^{-(1/4)-(\theta_1/2)}
\partial{\bar Z}^{a}S^{d}u_1(t)
\|_{L^2}^2\\
\leq&
C
\langle t\rangle^{-(1/2)-2\mu}
\biggl(
\|
r^{(1/2)+\theta_1}
\langle t-r\rangle^{1-\theta_1}
\partial\partial_x u_i(t)
\|_{L^\infty}\\
&
\hspace{2.3cm}
+
\sum_{k=1,2}
\|
r^{(1/2)+\theta_1}
\langle t-r\rangle^{1-\theta_1}
\partial u_k(t)
\|_{L^\infty}
\biggr)\\
&
\hspace{1.8cm}
\times
\|
r^{-(1/2)+\mu}
\partial{\bar Z}^{a}S^{d}u_1(t)
\|_{L^2}^2\\
\leq&
C
\langle t\rangle^{-(1/2)-2\mu+\delta}
\langle\!\langle u(t)\rangle\!\rangle
L(u_1(t))^2,\,\,\,i=1,2.
\end{split}
\end{equation}
%%%%%%%%%%%%%%%%%%%%%%%%%%%%%%%%%%%%%%%%%%%
For the estimate of $g'(t-r)q_2$, 
we choose $g=g(\rho)$ $(\rho\in{\mathbb R})$ so that 
%%%%%%%%%%%%%%%%%%%%%%%%%%%%
\begin{equation}\label{91}
g'(\rho)
=
-\langle\rho\rangle^{-1-2\eta}.
\end{equation}
%%%%%%%%%%%%%%%%%%%%%%%%%%
We then obtain
\begin{equation}\label{92}
\begin{split}
&\|
\chi_1
g'(t-r)
(\partial u_i(t))
(\partial{\bar Z}^{a}S^{d}u_1(t))^2
\|_{L^1({\mathbb R}^3)}\\
\leq&
C
\langle t\rangle^{-1-2\eta}
\|
\partial u_i(t)
\|_{L^\infty}
\|
\partial{\bar Z}^{a}S^{d}u_1(t)
\|_{L^2}^2
\leq
C
\langle t\rangle^{-2-2\eta+\delta}
\langle\!\langle u(t)\rangle\!\rangle
N(u_1(t))^2.
\end{split}
\end{equation}
%%%%%%%%%%%%%%%%%%%%%%%%%%%%%%%%%%%%%%
We have finished the estimate of $\chi_1q$. 

\noindent{\bf $\cdot$ Estimate of \boldmath$\chi_1 J_{1,k}$}. 
Next, let us consider the estimate of 
$\chi_1J_{1,k}$ ($k=1,\dots,5$). 
It suffices to deal with $\chi_1J_{1,2}$ 
and $\chi_1J_{1,5}$, 
because the others can be handled similarly. 
For the estimate of $\chi_1J_{1,2}$, 
we must proceed carefully, 
paying attention on the number of 
occurrence of $S$. 
Obviously, we may focus on the case 
$d'+d''=1$; in other cases, the estimate becomes much easier.  
Recall that we are considering the highest-order energy, 
i.e., $|a|+d=3$. 
We thus see $|a|\leq 2$ when $d'+d''=1$. 

\noindent\underline{Case $1$.} $d'=1$, $d''=0$. 

\noindent\underline{Case $1$-$1$. $|a'|=0$, $|a''|\leq 2$.} We use 
(\ref{eqn:3yokutsukau}) and the Sobolev embedding on $S^2$, to get
%%%%%%%%%%%%%%%%%%%%%%%%%%%%%%%%%%%%%%%%%%%
\begin{equation}\label{93}
\begin{split}
&\|
\chi_1
(\partial S u_2(t))
(\partial^2{\bar Z}^{a''} u_1(t))
(\partial_t{\bar Z}^{a}S u_1(t))
\|_{L^1({\mathbb R}^3)}\\
\leq&
C
\sum_{|b|\leq |a''|}
\biggl(
\|
\chi_1
(\partial S u_2(t))
(\partial\partial_x{\bar Z}^b u_1(t))
(\partial_t{\bar Z}^{a}S u_1(t))
\|_{L^1({\mathbb R}^3)}\\
&
\hspace{1.5cm}
+
\sum_{k=1}^2
\|
\chi_1
(\partial S u_2(t))
(\partial{\bar Z}^b u_k(t))
(\partial_t{\bar Z}^{a}S u_1(t))
\|_{L^1({\mathbb R}^3)}
\biggr)\\
\leq&
C
\sum_{|b|\leq |a''|}
\biggl(
\langle t\rangle^{-1}
\|
\langle r\rangle
\partial Su_2(t)
\|_{L^\infty}
\|
\langle t-r\rangle
\partial\partial_x{\bar Z}^b u_1(t)
\|_{L^2}\\
&
\hspace{2.3cm}
\times
\|
\langle r\rangle^{-1}
\partial_t{\bar Z}^a Su_1(t)
\|_{L^2}\\
&
\hspace{1.5cm}
+
\sum_{k=1}^2
\langle t\rangle^{-1+\theta_1}
\|
r^{-(1/2)+\mu}
\partial Su_2(t)
\|_{L^2_r L^4_\omega}\\
&
\hspace{3.7cm}
\times
\|
r^{(1/2)+\theta_1}
\langle t-r\rangle^{1-\theta_1}
\partial{\bar Z}^b u_k(t)
\|_{L^\infty_r L^4_\omega}\\
&
\hspace{3.7cm}
\times
\|
r^{-(1/2)+\mu}
\partial_t{\bar Z}^a Su_1(t)
\|_{L^2}
\biggr)\\
\leq&
C\langle t\rangle^{-1+2\delta}
\langle\!\langle u(t)\rangle\!\rangle
{\mathcal N}(u(t))L(u_1(t))\\
&+
C\langle t\rangle^{-(1/2)-2\mu+\delta}
\langle\!\langle u(t)\rangle\!\rangle
L(u_1(t))L(u_2(t)),
\end{split}
\end{equation}
where we have used (\ref{eqn:mn}) at the last inequality. 

%%%%%%%%%%%%%%%%%%%%%%%%%%%%%%%%%%%%%%%%%%%%%%%%%%%%%
%%%%%%%%%%%%%%%%%%%%%%%%%%%%%%%%%%%%%%%%%%%%%%%%%%%%%%
\noindent\underline{Case $1$-$2$. $|a'|,\,|a''|\leq 1$.} Using 
$\|
\langle r\rangle\partial{\bar Z}^{a'}Su_2(t)
\|_{L^\infty_r L^4_\omega}$ 
and 
$\|
\langle t-r\rangle\partial\partial_x{\bar Z}^{b}u_1(t)
\|_{L^2_r L^4_\omega}$ 
$(|b|\leq |a''|)$, 
we get by suitably modifying the argument in Case 1-1 above
%%%%%%%%%%%%%%%%%%%%%%%%%%%%%%%%%%
\begin{equation}\label{94}
\begin{split}
&\|
\chi_1
(\partial{\bar Z}^{a'}S u_2(t))
(\partial^2{\bar Z}^{a''} u_1(t))
(\partial_t{\bar Z}^{a}Su_1(t))
\|_{L^1({\mathbb R}^3)}\\
\leq&
C\langle t\rangle^{-1+2\delta}
\langle\!\langle u(t)\rangle\!\rangle
{\mathcal N}(u(t))L(u_1(t))\\
&
+
C\langle t\rangle^{-(1/2)-2\mu+\delta}
\langle\!\langle u(t)\rangle\!\rangle
L(u_1(t))L(u_2(t)).
\end{split}
\end{equation}
%%%%%%%%%%%%%%%%%%%%%%%%%%%%%%%%%%%%%%%%
\noindent\underline{Case $1$-$3$. $|a'|\leq 2$, $|a''|=0$.} We use 
$
\|
\chi_1 r^{(1/2)+\theta_1}
\langle t-r\rangle^{1-\theta_1}
\partial^2 u_1(t)
\|_{L^\infty({\mathbb R}^3)}
$ 
and obtain 
%%%%%%%%%%%%%%%%%%%%%%%%%%%%%%%%%%%%
\begin{equation}\label{95}
\begin{split}
&\|
\chi_1
(\partial{\bar Z}^{a'}Su_2(t))
(\partial^2 u_1(t))
(\partial_t{\bar Z}^a Su_1(t))
\|_{L^1({\mathbb R}^3)}\\
\leq&
C\langle t\rangle^{-(1/2)-2\mu+\delta}
\langle\!\langle u(t)\rangle\!\rangle
L(u_1(t))L(u_2(t)).
\end{split}
\end{equation}
%%%%%%%%%%2018/1/19, 14:00%%%%%%%%%%%%%%%%%%%%%
\noindent\underline{Case 2. $d'=0$, $d''=1$.} Recall 
that we are discussing the case $|a|+d=3$. 
Since we always have $|a''|+d''<|a|+d$, 
we know $|a''|\leq 1$ under the condition $d''=1$. 

\noindent\underline{Case $2$-$1$. $|a'|,\,|a''|\leq 1$.} We employ 
(\ref{eqn:3yokutsukau3}) to deal with 
$\partial_t^2{\bar Z}S u_1(t,x)$. 
We get
%%%%%%%%%%%%%%%%%%%%%%%%%%%%%%%
\begin{equation}\label{96}
\begin{split}
&\|
\chi_1
(\partial{\bar Z}^{a'} u_2(t))
(\partial^2{\bar Z}^{a''}S u_1(t))
(\partial_t{\bar Z}^{a}Su_1(t))
\|_{L^1({\mathbb R}^3)}\\
\leq&
C\langle t\rangle^{-(1/2)-2\mu}
\|\chi_1
r^{(1/2)+\theta_1}
\langle t-r\rangle^{1-\theta_1}
\partial{\bar Z}^{a'} u_2(t)
\|_{L^\infty}\\
&
\hspace{2.1cm}
\times
\|
r^{-(1/2)+\mu}\partial^2{\bar Z}^{a''}S u_1(t)
\|_{L^2}
\|
r^{-(1/2)+\mu}\partial_t{\bar Z}^{a}Su_1(t)
\|_{L^2}\\
\leq&
C\langle t\rangle^{-(1/2)-2\mu+\delta}
\langle\!\langle u(t)\rangle\!\rangle\\
&
\hspace{1cm}
\times
\biggl\{
\sum_{|b|,\,d\leq 1}
\biggl(
\|
r^{-(1/2)+\mu}\partial\partial_x {\bar Z}^b S^d u_1(t)
\|_{L^2}\\
&
\hspace{2.8cm}
+
\sum_{k=1}^2
\|
r^{-(1/2)+\mu}\partial{\bar Z}^b S^d u_k(t)
\|_{L^2}\\
&
\hspace{2.8cm}
+
\sum_{k=1}^2
\|\partial Su_k(t)\|_{L^\infty}
\|r^{-(1/2)+\mu}\partial\partial_x {\bar Z}u_1(t)\|_{L^2}
\biggr)
\biggr\}
L(u_1(t))\\
\leq&
C\langle t\rangle^{-(1/2)-2\mu+2\delta}
\langle\!\langle u(t)\rangle\!\rangle
\bigl(
L(u_1(t))+L(u_2(t))
\bigr)
L(u_1(t)).
\end{split}
\end{equation}
%%%%%%%%%%%%%%%%%%%%%%%%%%%%
At the last inequality, 
we have used 
$\|\partial Su_2(t)\|_{L^\infty}
\leq
\langle t\rangle^\delta\langle\!\langle u(t)\rangle\!\rangle
$. 
We have also used 
$\langle\!\langle u(t)\rangle\!\rangle^2
\leq
\langle\!\langle u(t)\rangle\!\rangle
$ 
because we are assuming 
smallness of $\langle\!\langle u(t)\rangle\!\rangle$ 
(see (\ref{eqn:daiji1})). 

%%%%%%%%%%%%%%%%%%%%%%%%%%%%%%
\noindent\underline{Case $2$-$2$. $|a'|\leq 2$, $|a''|=0$.} Using 
$
\|
\chi_1 r^{(1/2)+\theta_1}
\langle t-r\rangle^{1-\theta_1}
\partial{\bar Z}^{a'}u_2(t)
\|_{L^\infty_r L^4_\omega({\mathbb R}^3)}
$ 
and (\ref{eqn:3yokutsukau2}), we get
\begin{equation}\label{97}
\begin{split}
&\|
\chi_1
(\partial{\bar Z}^{a'} u_2(t))
(\partial^2 S u_1(t))
(\partial_t{\bar Z}^{a}Su_1(t))
\|_{L^1({\mathbb R}^3)}\\
\leq&
C\langle t\rangle^{-(1/2)-2\mu+\delta}
\langle\!\langle u(t)\rangle\!\rangle
\bigl(
L(u_1(t))+L(u_2(t))
\bigr)
L(u_1(t)).
\end{split}
\end{equation}
We have finished the estimate for $\chi_1J_{1,2}$. 
%%%%%%%2018/1/19, 14:57%%%%%%%%%%%%%%%%%%%%%

We turn our attention to the estimate for the semi-linear term 
$\chi_1J_{1,5}$. 
It suffices to handle 
$
\|
\chi_1
(\partial{\bar Z}^{a'}u_2(t))
(\partial{\bar Z}^{a''}Su_2(t))
(\partial_t{\bar Z}^a Su_1(t))
\|_{L^1({\mathbb R}^3)}
$ 
for $|a'|+|a''|\leq 2$. 
We use 
$
\|
\chi_1
r^{(1/2)+\theta_1}
\langle t-r\rangle^{1-\theta_1}
\partial{\bar Z}^{a'}u_2(t)
\|_{L^\infty({\mathbb R}^3)}
$
for $|a'|\leq 1$ 
and 
$
\|
\chi_1
r^{(1/2)+\theta_1}
\langle t-r\rangle^{1-\theta_1}
\partial{\bar Z}^{a'}u_2(t)
\|_{L^\infty_r L^4_\omega({\mathbb R}^3)}
$
for $|a'|=2$ 
to get 
%%%%%%%%%%%%%%%%%%%%%%%%%%%%%%%
\begin{equation}\label{98}
\begin{split}
&\|
\chi_1
(\partial{\bar Z}^{a'}u_2(t))
(\partial{\bar Z}^{a''}Su_2(t))
(\partial_t{\bar Z}^a Su_1(t))
\|_{L^1({\mathbb R}^3)}\\
\leq&
C\langle t\rangle^{-(1/2)-2\mu+\delta}
\langle\!\langle u(t)\rangle\!\rangle
L(u_1(t))L(u_2(t)).
\end{split}
\end{equation}
We have finished the estimate for 
$\chi_1J_{1,k}$, $k=1,\dots,5$. 
%%%%%2018/1/19 15:11%%%%%%%%%%%%%%%%%%%%%
%%%%%%%%%%%%%%%%%%%%%%%%

\noindent{\bf Estimate over 
\boldmath$\{x\in{\mathbb R}^3\,|\,|x|>(t+1)/2$\}}. 

\noindent{\bf$\cdot$ Estimate of \boldmath$\chi_2 q$}. 
By virtue of the null condition, we can use (\ref{eqn:null3}) 
together with (\ref{eqn:Zha}), (\ref{eqn:3c}) and obtain 
similarly to (\ref{64})--(\ref{65}) 
%%%%%%%%%%%%%%%%%%%%%%%%%%%%%%%%%%%%
\begin{equation}\label{99}
\begin{split}
&\|
\chi_2
G_1^{i1,\alpha\beta\gamma}
(\partial_{t\gamma}^2 u_i(t))
(\partial_\beta{\bar Z}^aS^d u_1(t))
(\partial_\alpha{\bar Z}^aS^d u_1(t))
\|_{L^1({\mathbb R}^3)}\\
\leq&
C
\langle t\rangle^{-1}
\sum_{k=1}^2
\biggl(
\sum_{|b|=1}
\|
\chi_2
\partial_t{\bar Z}^bu_k(t)
\|_{L^\infty}
+
\sum_{d'\leq 1}
\|
\chi_2\partial_t S^{d'}u_k(t)
\|_{L^\infty}\\
&
\hspace{1.8cm}
+
\|
\chi_2
\langle t-r\rangle
\partial_t\partial_x u_k(t)
\|_{L^\infty}
\biggr)
\|
\partial{\bar Z}^aS^d u_1
\|_{L^2}^2\\
&+
C
\sum_{k=1}^2
\|
\chi_2
(|\partial_t\partial_x u_k(t)|+
|\partial u_k(t)|)
(T{\bar Z}^aS^d u_1(t))
\partial{\bar Z}^aS^d u_1(t)
\|_{L^1}\\
\leq&
C
\langle t\rangle^{-(3/2)+\delta}
\langle\!\langle u(t)\rangle\!\rangle
N(u_1(t))^2
+
C
\langle t\rangle^{-1+\eta+\delta}
\langle\!\langle u(t)\rangle\!\rangle
G(u_1(t))N(u_1(t)).
\end{split}
\end{equation}
%%%%%%%%%%%%%%%%%%%%%%%%%%%%%%%%%%%%5
Here we have employed the norm 
$\||x|^{(1/2)+\theta_2}\langle t-r\rangle^{1-\theta_2}
\partial u_k(t)\|_{L^\infty}$, 
$\theta_2=(1/2)-\eta$. 
%%%%%%2018/1/21, 14:10%%%%%%%%%%%%%%%%%%%%%%%%%
Naturally, by using (\ref{eqn:null2}) instead of (\ref{eqn:null3}), 
we have a similar estimate for 
the second term on the right-hand side of (\ref{eqn:q1}). 

As for the treatment of $-(1/2)g'(t-r)q_2$, which is 
\begin{equation}\label{100}
\begin{split}
&\frac12
\langle t-r\rangle^{-1-2\eta}
\sum_{j=1}^3
(T_j{\bar Z}^{a}S^{d}u_1)^2\\
&
-
\langle t-r\rangle^{-1-2\eta}
G_1^{i1,\alpha\beta\gamma}
(\partial_{\gamma} u_i)
(\partial_\beta{\bar Z}^{a}S^{d}u_1)
(\partial_\alpha{\bar Z}^{a}S^{d}u_1)\\
&
+
2
\langle t-r\rangle^{-1-2\eta}
G_1^{i1,\alpha\beta\gamma}
(\partial_{\gamma} u_i)
(\partial_\beta{\bar Z}^{a}S^{d}u_1)
(-\omega_\alpha)
(\partial_t{\bar Z}^{a}S^{d}u_1),
\end{split}
\end{equation}
(see (\ref{eqn:q2})) we proceed as above, in order to 
treat the second and third terms on the right-hand side above.  
Namely, we first employ (\ref{eqn:null3}). 
We then use (\ref{eqn:Zha}), the simple inequality 
$
\langle t-r\rangle^{-1-2\eta}
\langle t-r\rangle
\leq 1
$
to get
%%%%%%%%%%%%%%%%%%%%%%%%%%%%%%%%
\begin{equation}\label{101}
\begin{split}
&\langle t-r\rangle^{-1-2\eta}
|T u_k|\\
\leq&
C
\langle t-r\rangle^{-1-2\eta}
\langle t\rangle^{-1}
\bigl(
|\partial u_k|
+
\sum_{|b|=1}
|\Omega^b u_k|
+|S u_k|
+
\langle t-r\rangle
|\partial_x u_k|
\bigr)\\
\leq&
C
\langle t\rangle^{-1}
\bigl(
|\partial u_k|
+
\sum_{|b|=1}
|\Omega^b u_k|
+|S u_k|
\bigr)
\leq
C
|x|^{-1/2}
\langle t\rangle^{-1+\delta}
\langle\!\langle u(t)\rangle\!\rangle.
\end{split}
\end{equation}
%%%%%%%%%%%%%%%%%%%%%%%%%%%%%%%%%%%%%
Using this inequality, we easily obtain
%%%%%%%%%%%%%%%%%%%%%%%%%
\begin{equation}\label{102}
\begin{split}
&\|
\chi_2
\langle t-r\rangle^{-1-2\eta}
G_1^{i1,\alpha\beta\gamma}
(\partial_{\gamma} u_i(t))
(\partial_\beta{\bar Z}^{a}S^{d}u_1(t))
(\partial_\alpha{\bar Z}^{a}S^{d}u_1(t))
\|_{L^1({\mathbb R}^3)}\\
\leq&
C\langle t\rangle^{-(3/2)+\delta}
\langle\!\langle u(t)\rangle\!\rangle
N(u_1(t))^2\\
&+
C\langle t\rangle^{-1+\delta}
\langle\!\langle u(t)\rangle\!\rangle
\|
\langle t-r\rangle^{-1-2\eta}
T{\bar Z}^{a}S^{d}u_1(t)
\|_{L^2}
N(u_1(t)).
\end{split}
\end{equation}
Essentially the same estimate as above remains true 
for the third term on the right-hand side of (\ref{eqn:q2}). 
%%%%%%%2018/1/21, 16:43%%%%%%%%%%%%%%%%%%%%

\noindent{\bf$\cdot$ Estimate of \boldmath$\chi_2 J_{1,k}$.} 
For the estimate of $\chi_2 J_{1,k}$ $(k=1,\dots,5)$ 
we must proceed carefully. 
%%%%%%%2018/1/21, 16:49%%%%%%%%%%%%%%%%%%%%%%%%%%%%%%%%%%%%%
As we did for $\chi_1 J_{1,k}$, 
we may focus on $\chi_2J_{1,2}$ and $\chi_2J_{1,5}$. 
Using (\ref{eqn:null1}), 
we first obtain for $\chi_2J_{1,2}$ 
%%%%%%%%%%%%%%%%%%%%%%%%%%%%%%%%%%%%%%
\begin{equation}\label{103}
\begin{split}
&\|
\chi_2
{\tilde G}^{\alpha\beta\gamma}
(\partial_\gamma{\bar Z}^{a'}S^{d'}u_2(t))
(\partial_{\alpha\beta}^2{\bar Z}^{a''}S^{d''}u_1(t))
(\partial_t{\bar Z}^a S^du_1(t))
\|_{L^1({\mathbb R}^3)}\\
\leq&
C
\bigl(
\|
\chi_2
(T{\bar Z}^{a'}S^{d'}u_2(t))
(\partial^2{\bar Z}^{a''}S^{d''}u_1(t))
\|_{L^2}\\
&
\hspace{0.4cm}
+
\|
\chi_2
(\partial{\bar Z}^{a'}S^{d'}u_2(t))
(T\partial{\bar Z}^{a''}S^{d''}u_1(t))
\|_{L^2}
\bigr)
\|
\partial_t{\bar Z}^aS^du_1(t)
\|_{L^2}\\
=:&
C(K_1+K_2)
\|
\partial_t{\bar Z}^aS^du_1(t)
\|_{L^2}.
\end{split}
\end{equation}
%%%%%%%%%%%%%%%%%%%%%%%%%%%%%%%%%%%
Recall that we are considering the highest-order energy 
$|a|+d=3$, which means that 
we are discussing the case 
$|a'|+|a''|+d'+d''\leq 3$, 
$|a''|+d''\leq 2$, 
and $d'+d''\leq d$. 
Again, we have only to deal with the case 
$d'+d''=d=1$; in other cases, the argument becomes much easier. 

\noindent\underline{Case 1. $d'=1$ and $d''=0$.} 

\noindent\underline{Case 1-1.\,$|a'|=0$ and $|a''|\leq 2$.} 
Recall $\theta_2=(1/2)-\eta$. We get by (\ref{eqn:3yokutsukau}) 
%%%%%%%%%%%%%%%%%%%%%%%%%%%%%%%%%%%%%5
\begin{equation}\label{104}
\begin{split}
K_1
&
=
\|
\chi_2
(TSu_2(t))
(\partial^2{\bar Z}^{a''}u_1(t))
\|_{L^2}\\
&
\leq
C
\langle t\rangle^{-1}
\|
r\langle t-r\rangle^{-1}TSu_2(t)
\|_{L^\infty}
\|
\langle t-r\rangle\partial\partial_x{\bar Z}^{a''}u_1(t)
\|_{L^2}\\
&
\hspace{0.38cm}
+
C
\langle t\rangle^{-1+\eta}
\|
\langle t-r\rangle^{-(1/2)-\eta}TSu_2(t)
\|_{L_r^2 L_\omega^4}\\
&
\hspace{2.2cm}
\times
\biggl(
\sum_{k=1,2}
\|
r^{1-\eta}
\langle t-r\rangle^{(1/2)+\eta}
\partial{\bar Z}^{a''}u_k(t)
\|_{L_r^\infty L_\omega^4}
\biggr)\\
&
\leq
C\langle t\rangle^{-1}
G(u_2(t))M(u_1(t))
+
C\langle t\rangle^{-1+\eta+\delta}
G(u_2(t))\langle\!\langle u(t)\rangle\!\rangle.
\end{split}
\end{equation}
%%%%%%%%%2018/1/25, 16:55%%%%%%%%%%%%%%%%%%
Here, 
to handle 
$\|
r\langle t-r\rangle^{-1}TSu_2(t)
\|_{L^\infty}$, 
we have used the following (see, e.g., (27), (28) in \cite{Zha})
%%%%%%%%%%%%%%%%%%%%%%%%%%%%%%%%
\begin{equation}
[\Omega_{ij},T_k]
=
\delta_{kj}T_i
-
\delta_{ki}T_j,
\quad
\partial_r T_i
=
\sum_{k=1}^3\frac{x_k}{r}T_i\partial_k
\label{eqn:Tcom}
\end{equation}
%%%%%%%%%%%%%%%%%%%%%%%%%%%%%%%%%%
and the Sobolev-type inequality (see, e.g., (3.19) and (3.14b) in \cite{Sideris2000})
%%%%%%%%%%%%%%%%%%%%%%%%%%%%%%%%
\begin{equation}
r\|v(r\cdot)\|_{L^4_\omega}
\leq
C
\|
\partial_r v
\|_{L^2({\mathbb R}^3)}^{1/2}
\biggl(
\sum_{|b|\leq 1}
\|
\Omega^b v
\|_{L^2({\mathbb R}^3)}^{1/2}
\biggr)
\label{TomKey}
\end{equation}
%%%%%%%%%%%%%%%%%%%%%%%%%%%%%%%%%%%%%%%%%
together with the Sobolev embedding on $S^2$. 
%%%%%%%%%%%2018/1/25, 17:27%%%%%%%%%%%%%%%%%%%%%%%%%%%%%%%%%
We also get by (\ref{eqn:Zha}), (\ref{eqn:3yokutsukau}), 
and (\ref{eqn:mn})
%%%%%%%%%%%%%%%%%%%%%%%%%%%%%%%%%%%%%%%%%
\begin{equation}\label{107}
\begin{split}
K_2
&=
\|
\chi_2
(\partial Su_2(t))
(T\partial{\bar Z}^{a''}u_1(t))
\|_{L^2}\\
&
\leq
C
\langle t\rangle^{-1}
\|
\langle r\rangle\partial Su_2(t)
\|_{L^\infty}
\|
T\partial{\bar Z}^{a''}u_1(t)
\|_{L^2}
\leq
C
\langle t\rangle^{-2+2\delta}
\langle\!\langle u(t)\rangle\!\rangle
{\mathcal N}(u(t)).
\end{split}
\end{equation}
%%%%%%%%%%%%%%%%%%%%%%%%%%%%%%%%%%%%%%%%%%%%%

\noindent\underline{Case 1-2. $|a'|\leq 1$ and $|a''|\leq 1$.} 
Employing 
$\|r\langle t-r\rangle^{-1}T{\bar Z}^{a'}Su_2(t)\|_{L^\infty_r L^4_\omega}$ 
and 
$\|\langle t-r\rangle \partial\partial_x{\bar Z}^{a''}u_1(t)\|_{L^2_r L^4_\omega}$, 
we can suitably modify the argument in Case 1-1 above to get
%%%%%%%%%%%%%%%%%%%%%%%%%%%%%
\begin{equation}\label{108}
\begin{split}
K_1
&=
\|
\chi_2
(T{\bar Z}^{a'}Su_2(t))
(\partial^2{\bar Z}^{a''}u_1(t))
\|_{L^2}\\
&
\leq
C\langle t\rangle^{-1}
G(u_2(t))M(u_1(t))
+
\langle t\rangle^{-1+\eta+\delta}
G(u_2(t))
\langle\!\langle u(t)\rangle\!\rangle.
\end{split}
\end{equation}
%%%%%%2018/1/26, 15:36%%%%%%%%%%%%%%%%%%%
Similarly, we obtain
%%%%%%%%%%%%%%%%%%%%%%%%%%%%%%%%%%%
\begin{equation}\label{109}
\begin{split}
K_2
&\leq
C
\langle t\rangle^{-1}
\|
\langle r\rangle
\partial{\bar Z}^{a'} Su_2(t)
\|_{L^\infty_r L^4_\omega}
\|
T\partial{\bar Z}^{a''}u_1(t)
\|_{L_r^2 L_\omega^4}\\
&
\leq
C
\langle t\rangle^{-2+2\delta}
\langle\!\langle u(t)\rangle\!\rangle 
{\mathcal N}(u(t)).
\end{split}
\end{equation}
%%%%%%%2018/1/26, 15:56%%%%%%%%%%%%%%%%%%%%

\noindent\underline{Case 1-3. $|a'|\leq 2$ and $|a''|=0$.} 
We easily get by using (\ref{eqn:3c})
%%%%%%%%%%%%%%%%%%%%%%%%%%%%%%
\begin{equation}\label{110}
\begin{split}
&K_1\\
\leq&
C
\langle t \rangle^{-1+\eta}
\|
\langle t-r\rangle^{-(1/2)-\eta}
T{\bar Z}^{a'}S u_2(t)
\|_{L^2}
\|
r^{1-\eta}\langle t-r\rangle^{(1/2)+\eta}
\partial^2 u_1(t)
\|_{L^\infty}\\
\leq&
C
\langle t \rangle^{-1+\eta+\delta}
G(u_2(t))
\langle\!\langle u(t)\rangle\!\rangle.
\end{split}
\end{equation}
%%%%%%%%%%%%%%%%%%%%%%%%%%%%%%
We also get by (\ref{eqn:Zha})
%%%%%%%%%%%%%%%%%%%%%%%%%%%%%
\begin{equation}\label{111}
K_2
\leq
\|
\partial{\bar Z}^{a'}S u_2(t)
\|_{L^2}
\|
\chi_2
T\partial u_1(t)
\|_{L^\infty}
\leq
C
\langle t \rangle^{-(3/2)+2\delta}
{\mathcal N}(u(t))
\langle\!\langle u(t)\rangle\!\rangle.
\end{equation}
%%%%%%2018/1/26, 16:12%%%%%%%%%%%%%%%%%%

\noindent\underline{Case 2. $d'=0$ and $d''=1$.} 
We note $|a''|\leq 1$ in this case. 

\noindent\underline{Case 2-1. $|a'|, |a''|\leq 1$.} 
Using (\ref{eqn:Zha}), (\ref{eqn:3yokutsukau2}), (\ref{eqn:3yokutsukau3}), 
we get
%%%%%%%%%%%%%%%%%%%%%%%%%%%%%
\begin{equation}\label{112}
\begin{split}
K_1
&=
\|
\chi_2
(T{\bar Z}^{a'} u_2(t))
(\partial^2{\bar Z}^{a''}S u_1(t))
\|_{L^2}\\
&
\leq
C
\langle t \rangle^{-1/2}
\|
r^{1/2}
T{\bar Z}^{a'}u_2(t)
\|_{L^\infty}
\|
\chi_2
\partial^2
{\bar Z}^{a''}Su_1(t)
\|_{L^2}\\
&
\leq
C
\langle t \rangle^{-(3/2)+\delta}
\langle\!\langle u(t)\rangle\!\rangle
{\mathcal N}(u(t)),
\end{split}
\end{equation}
%%%%%%%%%%%%%%%%%%%%%%%%%%%%%%%%%%%
where we have used (\ref{eqn:3yokutsukau3}) 
along with 
$\|\chi_2\partial Su_2(t)\|_{L^\infty}
\leq
C\langle t\rangle^{-1}
\|\langle r\rangle\partial Su_2(t)\|_{L^\infty}
\leq
C\langle t\rangle^{-1+\delta}
\langle\!\langle u(t)\rangle\!\rangle
\leq
C\langle\!\langle u(t)\rangle\!\rangle
$ 
and 
$\langle\!\langle u(t)\rangle\!\rangle^2
\leq
\langle\!\langle u(t)\rangle\!\rangle$ 
as before. 
We also get
%%%%%%%%%%%%%%%%%%%%%%%%%%%%%%%%%%%%%%
\begin{equation}\label{113}
\begin{split}
K_2=&
\|
\chi_2
(\partial{\bar Z}^{a'} u_2(t))
(T\partial{\bar Z}^{a''}S u_1(t))
\|_{L^2}\\
\leq&
C
\langle t \rangle^{-1+\eta}
\|
r^{1-\eta}
\langle t-r\rangle^{(1/2)+\eta}
\partial{\bar Z}^{a'} u_2(t)
\|_{L^\infty}\\
&
\hspace{1.5cm}
\times
\|
\langle t-r\rangle^{-(1/2)-\eta}
T\partial{\bar Z}^{a''}S u_1(t)
\|_{L^2}\\
\leq&
C
\langle t \rangle^{-1+\eta+\delta}
\langle\!\langle u(t)\rangle\!\rangle
G(u_1(t)).
\end{split}
\end{equation}
%%%%%%%%%%%%%%%%%%%%%%%%%%
We note that 
this is one of the places where 
we encounter a little troublesome norm 
$\|
\langle t-r\rangle^{-(1/2)-\eta}
T_j\partial_t{\bar Z}^aS^d u_i(t)
\|_{L^2}$ 
$(|a|+d\leq 2)$. 
%%%%%%2018/1/27, 15:45%%%%%%%%%%%%%%%%%%%%%

\noindent\underline{Case 2-2. $|a'|\leq 2$ and $|a''|=0$.} 
We naturally modify the argument in Case 2-1 and obtain 
%%%%%%%%%%%%%%%%%%%%%%%%%%%%%%%%%%%%%%%%%%%%
\begin{equation}\label{114}
\begin{split}
K_1
&
\leq
C
\langle t \rangle^{-1/2}
\|
r^{1/2}
T{\bar Z}^{a'}u_2(t)
\|_{L^\infty_r L^4_\omega}
\|
\chi_2
\partial^2 Su_1(t)
\|_{L^2_r L^4_\omega}\\
&
\leq
C
\langle t \rangle^{-(3/2)+\delta}
\langle\!\langle u(t)\rangle\!\rangle
{\mathcal N}(u(t))
\end{split}
\end{equation}
%%%%%%%%%%%%%%%%%%%%%%%%%%%%%%%%%%
and
%%%%%%%%%%%%%%%%%%%%%%%%%%%%%%
\begin{equation}\label{115}
\begin{split}
K_2
&\leq
C
\langle t \rangle^{-1+\eta}
\|
r^{1-\eta}
\langle t-r\rangle^{(1/2)+\eta}
\partial{\bar Z}^{a'} u_2(t)
\|_{L^\infty_r L^4_\omega}\\
&
\hspace{1.8cm}
\times
\|
\langle t-r\rangle^{-(1/2)-\eta}
T\partial S u_1(t)
\|_{L^2_r L^4_\omega}\\
&
\leq
C
\langle t \rangle^{-1+\eta+\delta}
\langle\!\langle u(t)\rangle\!\rangle
G(u_1(t)).
\end{split}
\end{equation}
%%%%%%%%%%%%%%%%%%%%%%%%%%%%%%%%%%%
We have obtained the estimate of 
$\chi_2 J_{1,2}$. As for the semi-linear terms $\chi_2J_{1,5}$, 
%%%%%%%%2018/1/27, 16:21%%%%%%%%%%%%%%%%%%%%%%%%%%%%%
we first note that due to (\ref{eqn:null4}), 
the inequality
%%%%%%%%%%%%%%%%%%%%%%%%%%%%%%%%%%%%
\begin{equation}\label{116}
\begin{split}
&\|
\chi_2
{\bar H}^{\alpha\beta}
(\partial_\alpha{\bar Z}^{a'}S^{d'}u_2(t))
(\partial_\beta{\bar Z}^{a''}S^{d''}u_2(t))
\|_{L^2({\mathbb R}^3)}\\
\leq&
C
\|
\chi_2
(T{\bar Z}^{a'}S^{d'}u_2(t))
(\partial{\bar Z}^{a''}S^{d''}u_2(t))
\|_{L^2}\\
&+
C
\|
\chi_2
(\partial{\bar Z}^{a'}S^{d'}u_2(t))
(T{\bar Z}^{a''}S^{d''}u_2(t))\|_{L^2}
\end{split}
\end{equation}
%%%%%%%%%%%%%%%%%%%%%%%%%%%%%%
holds. Due to symmetry, 
we may suppose $d'=1$, $d''=0$. 
When $|a'|=0$ and $|a''|\leq 2$ or 
$|a'|\leq 1$ and $|a''|\leq 1$, 
we get as in (\ref{115}) 
%%%%%%%%%%%%%%%%%%%%%%%%%%%%%
\begin{equation}\label{117}
\begin{split}
&\|
\chi_2
(T{\bar Z}^{a'}S u_2(t))
(\partial{\bar Z}^{a''}u_2(t))
\|_{L^2}\\
\leq&
C\langle t \rangle^{-1+\eta}
\|
\langle t-r\rangle^{-(1/2)-\eta}T{\bar Z}^{a'}S u_2(t)
\|_{L_r^2 L_\omega^4}\\
&
\hspace{1.4cm}
\times
\|
r^{1-\eta}
\langle t-r\rangle^{(1/2)+\eta}
\partial{\bar Z}^{a''}u_2(t)
\|_{L_r^\infty L_\omega^4}\\
\leq&
C\langle t \rangle^{-1+\eta+\delta}
G(u_2(t))
\langle\!\langle u(t)\rangle\!\rangle.
\end{split}
\end{equation}
%%%%%%%%%%%%%%%%%%%%%%%%%%%%%
Also, by (\ref{eqn:hoshiro}), (\ref{eqn:Zha}), (\ref{eqn:mn}) 
and the commutation relation
%%%%%%%%%%%%%%%%%%%%%%%%%%%%%
\begin{equation}\label{118}
[\partial_j,T_k]
=
\frac{1}{r}
\biggl(
\delta_{jk}
-
\frac{x_jx_k}{r^2}
\biggr)\partial_t,
\end{equation}
%%%%%%%%%%%%%%%%%%%%%%%%%%%
we get
%%%%%%%%%%%%%%%%%%%%%%%%%
\begin{equation}\label{119}
\begin{split}
&\|
\chi_2
(\partial{\bar Z}^{a'}S u_2(t))
(T{\bar Z}^{a''}u_2(t))
\|_{L^2({\mathbb R}^3)}\\
\leq&
C\langle t \rangle^{-1/2}
\|
\partial{\bar Z}^{a'}S u_2(t)
\|_{L_r^2 L_\omega^4}
\|
r^{1/2}
{\tilde \chi}_2
T{\bar Z}^{a''}u_2(t)
\|_{L_r^\infty L_\omega^4}\\
\leq&
C\langle t \rangle^{-(1/2)+\delta}
{\mathcal N}(u(t))
\|
{\tilde \chi}_2
T{\bar Z}^{a''}u_2(t)
\|_{{\dot H}^1}
\leq
C\langle t \rangle^{-(3/2)+2\delta}
{\mathcal N}(u(t))^2.
\end{split}
\end{equation}
%%%%%%%%%%%%%%%%%%%%%%%%%%%%%%%%%%%%%%
Here, we have used ${\tilde \chi}_2$ which is defined as 
${\tilde \chi}_2(x):=\chi((2/(t+1))x)$ 
for a smooth, radially symmetric function $\chi(x)$ 
such that $\chi(x)=0$ for $|x|\leq 1/2$, 
$\chi(x)=1$ for $|x|\geq 1$. 

When $|a'|\leq 2$ and $|a''|=0$, 
we obtain
%%%%%%%%%%%%%%%%%%%%%%%%%%%%%%
\begin{equation}\label{120}
\begin{split}
&\|
\chi_2
(T{\bar Z}^{a'}S u_2(t))
(\partial u_2(t))
\|_{L^2({\mathbb R}^3)}\\
\leq&
C\langle t \rangle^{-1+\eta}
G(u_2(t))
\|
r^{1-\eta}\langle t-r\rangle^{(1/2)+\eta}\partial u_2(t)
\|_{L^\infty}\\
\leq&
C\langle t \rangle^{-1+\eta+\delta}
G(u_2(t))
\langle\!\langle u(t)\rangle\!\rangle.
\end{split}
\end{equation}
%%%%%%%%%%%%%%%%%%%%%%%%%%%%%%%
We also get in the same way as in (\ref{119})
%%%%%%%%%%%%%%%%%%%%%%%%%%%%%%%%
\begin{equation}\label{121}
\begin{split}
&\|
\chi_2
(\partial{\bar Z}^{a'}S u_2(t))
(Tu_2(t))
\|_{L^2({\mathbb R}^3)}\\
\leq&
C
\langle t \rangle^{-1/2}
N(u_2(t))
\|
r^{1/2}{\tilde\chi}_2 Tu_2(t)
\|_{L^\infty}
\leq
C\langle t \rangle^{-(3/2)+2\delta}
{\mathcal N}(u(t))^2.
\end{split}
\end{equation}
%%%%%%%%%%%%%%%%%%%%%%%%%%
We are ready to complete the proof of 
Proposition \ref{proposition2}. 
First, we note that owing to (\ref{91}), 
the function $g(t-r)$ is bounded, 
which means that 
the function $e^{g(t-r)}$ appearing in (\ref{76}) satisfies 
$c\leq e^{g(t-r)}\leq C$ for some positive constants $c$ and $C$. 
Second, we must mention how to deal with 
rather troublesome terms 
%%%%%%%%%%%%%%%%%%%%%%%%%%%%%%%%%%
\begin{align}
&
\int_0^t
\langle\tau\rangle^{-(1/2)-2\mu+2\delta}
L(u_1(\tau))L(u_2(\tau))d\tau,\label{122}\\
&
\int_0^t
\langle\tau\rangle^{-1+\eta+\delta}G(u_2(\tau))d\tau,\label{123}
\end{align}
%%%%%%%%%%%%%%%%%%%%%%%%%%%%%%%%%%%
which naturally come from the integration of such terms as in 
(\ref{96}) and (\ref{108}) with respect to the time variable. 
As in \cite[p.\,363]{Sogge2003}, 
the idea of dyadic decomposition of the interval $(0,t)$ 
plays a useful role. Without loss of generality, 
we suppose $T>1$. 
For any $t\in(1,T)$, we see
%%%%%%%%%%%%%%%%%%%%%%%%%%%%%%%%%%%%%
\begin{equation}\label{124}
\begin{split}
&\int_1^t
\langle\tau\rangle^{-(1/2)-2\mu+2\delta}
L(u_1(\tau))L(u_2(\tau))d\tau
=
\sum_{j=0}^N
\int_{2^j}^{2^{j+1}}\cdots d\tau\\
\leq&
\sum_{j=0}^N
(2^j)^{-(1/2)+(9\delta/2)}
\biggl\{
(2^j)^{-\mu-\delta}
\biggl(
\int_{2^j}^{2^{j+1}}
L(u_1(\tau))^2d\tau
\biggr)^{1/2}
\biggr\}\\
&
\hspace{3cm}
\times
\biggl\{
(2^j)^{-\mu-(3\delta/2)}
\biggl(
\int_{2^j}^{2^{j+1}}
L(u_2(\tau))^2d\tau
\biggr)^{1/2}
\biggr\}\\
\leq&
C
\biggl(
\sum_{j=0}^\infty
(2^j)^{-(1/2)+(9\delta/2)}
\biggr)
\sup_{1<\sigma<T}
\langle \sigma\rangle^{-\mu-\delta}
\biggl(
\int_1^\sigma
L(u_1(\tau))^2d\tau
\biggr)^{1/2}\\
&\hspace{3.8cm}\times
\sup_{1<\sigma<T}
\langle \sigma\rangle^{-\mu-(3\delta/2)}
\biggl(
\int_1^\sigma
L(u_2(\tau))^2d\tau
\biggr)^{1/2}.
\end{split}
\end{equation}
%%%%%%%%%%%%%%%%%%%%%%%%%%%%%%
Here, and later on as well, 
we abuse the notation 
to mean $t$ by $2^{N+1}$. 
Because $\delta$ is a sufficiently small 
positive number, 
we are able to obtain the desired estimate. 
Also, 
%%%%%%%%%%%%%%%%%%%%%%%%%%%%%
\begin{equation}\label{125}
\begin{split}
&\int_1^t
\langle\tau\rangle^{-1+\eta+\delta}
G(u_2(\tau))d\tau\\
\leq&
\sum_{j=0}^N
\biggl(
\int_{2^j}^{2^{j+1}}\tau^{-2+2(\eta+\delta)}d\tau
\biggr)^{1/2}
\biggl(
\int_{2^j}^{2^{j+1}}G(u_2(\tau))^2d\tau
\biggr)^{1/2}\\
\leq&
C\sum_{j=0}^N(2^j)^{-(1/2)+\eta+2\delta}
(2^j)^{-\delta}
\biggl(
\int_{2^j}^{2^{j+1}}G(u_2(\tau))^2d\tau
\biggr)^{1/2}\\
\leq&
C\sup_{1<\sigma<T}
\langle\sigma\rangle^{-\delta}
\biggl(
\int_1^\sigma
G(u_2(\tau))^2d\tau
\biggr)^{1/2},
\end{split}
\end{equation}
because $\delta$ and $\eta$ are sufficiently small positive numbers. 
The estimate of the integration from $0$ to $1$ is much easier, 
thus we omit it.  
Integrating (\ref{76}) over $(0,t)\times{\mathbb R}^3$, 
we are now able to obtain (\ref{89}), 
except the required estimate for 
$\|
\langle t-r\rangle^{-(1/2)-\eta}
T_j\partial_t{\bar Z}^aS^d u_1
\|_{L^2((0,T)\times{\mathbb R}^3)}$, 
$|a|+d\leq 2,\,d\leq 1$. 
To estimate it, 
we need to replace 
${\bar Z}^aS^d$ 
$(|a|+d=3)$ 
in (\ref{76})--(\ref{83}) 
by 
$\partial_t{\bar Z}^aS^d$ 
$(|a|+d=2)$, 
accordingly modifying 
${\bar Z}^{a'}S^{d'}$, ${\bar Z}^{a''}S^{d''}$ 
in (\ref{79})--(\ref{83}). 
Firstly, note that we then encounter a little troublesome 
$\partial_t^2{\bar Z}^aS^d u_k$ 
$(|a|+d=2)$ 
and 
$\partial_t^3 Su_k$. 
It is safe to say that 
we have already seen how to handle these. 
See, e.g., 
(\ref{96}) and (\ref{112}). 
Namely, it suffices to employ 
(\ref{eqn:3yokutsukau3}) and (\ref{eqn:ad6}) 
together with 
$\|\chi_1\partial S u_k(t)\|_{L^\infty}
\leq
C\langle t\rangle^\delta
\langle\!\langle u(t)\rangle\!\rangle$ and 
$\|\chi_2\partial S u_k(t)\|_{L^\infty}
\leq
C\langle\!\langle u(t)\rangle\!\rangle$. 
Secondly, note that, when using Lemma \ref{lemma2.2}, 
we then encounter 
$T_j\partial_t^2 u_k$, 
$T_j\partial_t^2{\bar Z}^au_k$ $(|a|=1)$, 
and $T_j\partial_t^2Su_k$. 
It is also safe to say that 
we have no trouble dealing with these; 
after employing (\ref{eqn:ad3}), (\ref{eqn:adad1}) 
(together with 
$\|\chi_2\partial\partial_x{\bar Z}^b u_k(t)\|
_{L^\infty_r L^4_\omega}
\leq
C\langle\!\langle u(t)\rangle\!\rangle$, 
$|b|\leq 1$, 
$k=1,2$), 
and (\ref{eqn:adad2}) 
(together with $\|\chi_2\partial S u_k(t)\|
_{L^\infty}
\leq
C\langle\!\langle u(t)\rangle\!\rangle$, 
$\|\chi_2\partial\partial_x S u_k(t)\|
_{L^\infty_r L^4_\omega}
\leq
C\langle\!\langle u(t)\rangle\!\rangle$), 
we can proceed in exactly the same way as we have done above. 
The proof of Proposition \ref{proposition2} has been finished. 
%%%%%%%%%%%%%%%%%%%%%%%%%%%%%%%%%%%%%%%%
%%%%%%%%%%%%%%%%%%%%%%%%%%%%%%%%%%%%%%%
%%%%%%%%%%%%%%%%%%%%%%%%%%%%%%%%%%%%%%
\section{Energy estimate of $u_2$}
This section is devoted to the energy estimate of $u_2$. 
We will show:
\begin{proposition}\label{proposition3}
The following inequality holds 
for smooth local solutions to $(\ref{1})$ 
$u=(u_1,u_2)$, 
as long as they satisfy $(\ref{eqn:daiji1})$ for some time interval 
$(0,T):$
%%%%%%%%%%%%%%%%%%%%%%%%%%%%%
\begin{equation}\label{126}
\begin{split}
&\sup_{0<t<T}
\langle t\rangle^{-2\delta}N(u_2(t))^2
+
\sup_{0<t<T}
\langle t\rangle^{-2\delta}
\int_0^t G(u_2(\tau))^2 d\tau\\
\leq&
CN(u_2(0))^2
+
C\sup_{0<t<T}
\langle\!\langle u(t)\rangle\!\rangle
\bigl(
{\mathcal N}_T(u)^2
+
{\mathcal G}_T(u)^2
+
{\mathcal L}_T(u)^2
\bigr)\\
&+
C{\mathcal N}_T(u)^3
+
C{\mathcal G}_T(u){\mathcal N}_T(u)^2.
\end{split}
\end{equation}
\end{proposition}
%%%%%%%%%%%%%%%%%%%%%%%%%%%%%%%%%%%%%%%
The rest of this section is devoted to the proof of this proposition. 
As in the previous section, 
we have only to deal with the highest-order energy. 
In the same way as in (\ref{76}), we get
\begin{equation}\label{127}
\begin{split}
&\frac12
\partial_t
\bigl\{
e^g
\bigl(
(\partial_t{\bar Z}^{a}S^{d}u_2)^2
+
|\nabla{\bar Z}^{a}S^{d}u_2|^2\\
&
\hspace{1.3cm}
-
G_2^{i2,\alpha\beta\gamma}
(\partial_\gamma u_i)
(\partial_\beta{\bar Z}^{a}S^{d}u_2)
(\partial_\alpha{\bar Z}^{a}S^{d}u_2)\\
&
\hspace{1.3cm}
+2
G_2^{i2,0\beta\gamma}
(\partial_\gamma u_i)
(\partial_\beta{\bar Z}^{a}S^{d}u_2)
(\partial_t{\bar Z}^{a}S^{d}u_2)
\bigr)
\bigr\}\\
&
+\nabla\cdot\{\cdots\}
+
e^g{\tilde q}
+
e^g(J_{2,1}+J_{2,2}+\cdots+J_{2,5})=0.
\end{split}
\end{equation}
%%%%%%%%%%%%%%%%%%%%%%%%%%%%%%%%
Here, $g=g(t-r)$ is the same as in (\ref{91}), 
${\tilde q}=q_3-(1/2)g'(t-r)q_4$, 
%%%%%%%%%%%
\begin{equation}\label{128}
\begin{split}
q_3
=&
\frac12
G_2^{i2,\alpha\beta\gamma}
(\partial_{t\gamma}^2 u_i)
(\partial_\beta{\bar Z}^{a}S^{d}u_2)
(\partial_\alpha{\bar Z}^{a}S^{d}u_2)\\
&
-
G_2^{i2,\alpha\beta\gamma}
(\partial_{\alpha\gamma}^2 u_i)
(\partial_\beta{\bar Z}^{a}S^{d}u_2)
(\partial_t{\bar Z}^{a}S^{d}u_2),
\end{split}
\end{equation}
%%%%%%%%%%%%%%%%%%%%%
\begin{equation}\label{129}
\begin{split}
q_4
=
&\sum_{j=1}^3
(T_j{\bar Z}^{a}S^{d}u_2)^2%\\
%&
-
G_2^{i2,\alpha\beta\gamma}
(\partial_{\gamma} u_i)
(\partial_\beta{\bar Z}^{a}S^{d}u_2)
(\partial_\alpha{\bar Z}^{a}S^{d}u_2)\\
&
+
2
G_2^{i2,\alpha\beta\gamma}
(\partial_{\gamma} u_i)
(\partial_\beta{\bar Z}^{a}S^{d}u_2)
(-\omega_\alpha)
(\partial_t{\bar Z}^{a}S^{d}u_2),
\end{split}
\end{equation}
%%%%%%%%%%%%%%%%%%%%%%%%%%%%%%%
and 
%%%%%%%%%%%%%%%%%%%%%%%%%%%%%%%%%%%%%5
\begin{align}
&
J_{2,1}
=
\sum\!{}^{'}
{\tilde G}^{\alpha\beta\gamma}
(\partial_\gamma{\bar Z}^{a'}S^{d'}u_1)
(\partial_{\alpha\beta}^2{\bar Z}^{a''}S^{d''}u_2)
(\partial_t{\bar Z}^{a}S^{d}u_2),\label{130}\\
&
J_{2,2}
=
\sum\!{}^{'}
{\hat G}^{\alpha\beta\gamma}
(\partial_\gamma{\bar Z}^{a'}S^{d'}u_2)
(\partial_{\alpha\beta}^2{\bar Z}^{a''}S^{d''}u_2)
(\partial_t{\bar Z}^{a}S^{d}u_2),\label{131}\\
&
J_{2,3}
=
\sum\!{}^{''}
{\tilde H}^{\alpha\beta}
(\partial_\alpha{\bar Z}^{a'}S^{d'}u_1)
(\partial_\beta{\bar Z}^{a''}S^{d''}u_2)
(\partial_t{\bar Z}^{a}S^{d}u_2),\label{132}\\
&
J_{2,4}
=
\sum\!{}^{''}
{\hat H}^{\alpha\beta}
(\partial_\alpha{\bar Z}^{a'}S^{d'}u_1)
(\partial_\beta{\bar Z}^{a''}S^{d''}u_1)
(\partial_t{\bar Z}^{a}S^{d}u_2),\label{133}\\
&
J_{2,5}
=
\sum\!{}^{''}
{\bar H}^{\alpha\beta}
(\partial_\alpha{\bar Z}^{a'}S^{d'}u_2)
(\partial_\beta{\bar Z}^{a''}S^{d''}u_2)
(\partial_t{\bar Z}^{a}S^{d}u_2).\label{134}
\end{align}
%%%%%%%%%%%%%%%%%%%%%%%%%%%%%%%%%%%%%%%%
Recall that 
we have dealt with $\chi_1 q$ and 
$\chi_1 J_{1,1},\dots,\chi_1 J_{1,5}$ 
without relying upon 
the null condition. 
Therefore, it is possible to handle 
$\chi_1{\tilde q}$ and $\chi_1 J_{2,1},\dots,\chi_1 J_{2,5}$ 
as before. 
We may thus focus on the estimate of 
$\chi_2{\tilde q}$ and $\chi_2 J_{2,1},\dots,\chi_2 J_{2,5}$. 
For the estimate of $\chi_2{\tilde q}$, 
it suffices to show 
how to handle the terms with 
the coefficients $\{G_2^{12,\alpha\beta\gamma}\}$, 
because the coefficients $\{G_2^{22,\alpha\beta\gamma}\}$ 
satisfy the null condition 
and thus we are able to 
treat all the terms with the coefficients $\{G_2^{22,\alpha\beta\gamma}\}$ 
in the same way as before. 

Using the first equation in (\ref{1}) to represent 
$\partial_t^2 u_1$ as 
$\Delta u_1+\mbox{(higher-order terms)}$ 
and then use (\ref{eqn:3c}) to represent 
$\partial_t^2 u_1$ appearing in these higher-order terms, 
we obtain
%%%%%%%%%%%%%%%%%%%%%%%%%%%%%
\begin{equation}\label{135}
\begin{split}
&\|
\chi_2G_2^{12,\alpha\beta\gamma}
(\partial_{t\gamma}^2 u_1(t))
(\partial_\beta{\bar Z}^a S^d u_2(t))
(\partial_\alpha{\bar Z}^a S^d u_2(t))
\|_{L^1({\mathbb R}^3)}\\
\leq&
C
\langle t\rangle^{-1}
\|
\chi_2 |x| 
(\partial_t^2 u_1(t))
(\partial{\bar Z}^a S^d u_2(t))
(\partial{\bar Z}^a S^d u_2(t))
\|_{L^1({\mathbb R}^3)}\\
&+
C
\langle t\rangle^{-1}
\|
\chi_2 |x| 
(\partial_t\partial_x u_1(t))
(\partial{\bar Z}^a S^d u_2(t))
(\partial{\bar Z}^a S^d u_2(t))
\|_{L^1({\mathbb R}^3)}\\
\leq&
C
\langle t\rangle^{-1+2\delta}
\langle\!\langle u\rangle\!\rangle
{\mathcal N}(u(t))^2.
\end{split}
\end{equation}
%%%%%%%%%%%%%%%%%%%%%%%%%%%%%%%%%
(If we employed (\ref{eqn:3c}) directly, 
it would meet with the troublesome factor 
$\langle t\rangle^{-1+3\delta}$ on the right-hand side above. 
This is the reason why 
we have used the first equation in (\ref{1}) to represent 
$\partial_t^2 u_1$ as 
$\Delta u_1+\mbox{(higher-order terms)}$.)  
Here, we have used the assumption that 
$\langle\!\langle u\rangle\!\rangle$ is small, 
so that we have 
$\langle\!\langle u\rangle\!\rangle^2
\leq\langle\!\langle u\rangle\!\rangle$. 
In the same way, 
we get
%%%%%%%%%%%%%%%%%%%%%%%%%%%%%%%%%
\begin{equation}\label{136}
\begin{split}
&\|
\chi_2G_2^{12,\alpha\beta\gamma}
(\partial_{\alpha\gamma}^2 u_1(t))
(\partial_\beta{\bar Z}^a S^d u_2(t))
(\partial_t{\bar Z}^a S^d u_2(t))
\|_{L^1({\mathbb R}^3)}\\
\leq&
C
\langle t\rangle^{-1+2\delta}
\langle\!\langle u\rangle\!\rangle
{\mathcal N}(u(t))^2.
\end{split}
\end{equation}
%%%%%%%%%%%%%%%%%%%%%%%%%%%%%%%%
It is easy to show
%%%%%%%%%%%%%%%%%%%%%%%%%%%%%
\begin{equation}\label{137}
\begin{split}
&\|
\chi_2G_2^{12,\alpha\beta\gamma}
(\partial_{\gamma} u_1(t))
(\partial_\beta{\bar Z}^{a}S^{d}u_2(t))
(\partial_\alpha{\bar Z}^{a}S^{d}u_2(t))
\|_{L^1({\mathbb R}^3)},\\
&\|
\chi_2
G_2^{12,\alpha\beta\gamma}
(\partial_{\gamma} u_1(t))
(\partial_\beta{\bar Z}^{a}S^{d}u_2(t))
(-\omega_\alpha)
(\partial_t{\bar Z}^{a}S^{d}u_2(t))
\|_{L^1({\mathbb R}^3)}\\
\leq&
C
\langle t\rangle^{-1+2\delta}
\langle\!\langle u\rangle\!\rangle
{\mathcal N}(u(t))^2.
\end{split}
\end{equation}
%%%%%%%%%%%%%%%%%%%%%%%%%%%%
We have finished the estimate of 
$\chi_2{\tilde q}$. 

We next deal with $\chi_2 J_{2,1},\dots,\chi_2J_{2,5}$. 
We may focus on $\chi_2 J_{2,1}$, $\chi_2 J_{2,3}$, and 
$\chi_2 J_{2,4}$ because the coefficients 
$\{{\hat G}^{\alpha\beta\gamma}\}$ and 
$\{{\bar H}^{\alpha\beta}\}$ satisfy the null condition 
and it is therefore possible to handle 
$\chi_2J_{2,2}$ and $\chi_2J_{2,5}$ in the same way as before. 
Let us first deal with $\chi_2 J_{2,1}$. 
When $|a''|+d''=2$ 
(and thus $|a'|+d'\leq 1$), 
we get by (\ref{eqn:3yokutsukau}), (\ref{eqn:3yokutsukau3})
%%%%%%%%%%%%%%%%%%%%%%%%%%%%%%%%%%%%%%%%
\begin{equation}\label{138}
\begin{split}
&\|
\chi_2
(\partial_\gamma{\bar Z}^{a'}S^{d'}u_1(t))
(\partial_{\alpha\beta}^2{\bar Z}^{a''}S^{d''}u_2(t))
(\partial_t{\bar Z}^a S^d u_2(t))
\|_{L^1({\mathbb R}^3)}\\
\leq&
C\langle t\rangle^{-1}
\|
\chi_2 |x| \partial_\gamma{\bar Z}^{a'}S^{d'}u_1(t)
\|_{L^\infty({\mathbb R}^3)}\\
&
\hspace{1cm}
\times
\|
\chi_2
\partial_{\alpha\beta}^2{\bar Z}^{a''}S^{d''}u_2(t)
\|_{L^2({\mathbb R}^3)}
\|
\partial_t{\bar Z}^a S^d u_2(t)
\|_{L^2({\mathbb R}^3)}\\
\leq&
C\langle t\rangle^{-1+2\delta}
\langle\!\langle u(t)\rangle\!\rangle
{\mathcal N}(u(t))^2.
\end{split}
\end{equation}
%%%%%%%%%%%%%%%%%%%%%%%%%%%%%%%%%%%
Note that, 
to handle 
$\|
\chi_2
\partial_t^2{\bar Z}^{a''}S u_2(t)
\|_{L^2({\mathbb R}^3)}$ ($|a''|=1$), 
we have again used (\ref{eqn:3yokutsukau3}) 
along with 
$\|\chi_2\partial S u_2(t)\|_{L^\infty}
\leq
C\langle\!\langle u(t)\rangle\!\rangle$ 
(see (\ref{112})) and 
smallness of 
$\langle\!\langle u(t)\rangle\!\rangle$. 

On the other hand, 
when $|a''|+d''\leq 1$ 
(and thus $|a'|+d'\leq 2$), 
we get
%%%%%%%%%%%%%%%%%%%%%%%%%%%%%%%
\begin{equation}\label{139}
\begin{split}
&\|
\chi_2
(\partial_\gamma{\bar Z}^{a'}S^{d'}u_1(t))
(\partial_{\alpha\beta}^2{\bar Z}^{a''}S^{d''}u_2(t))
(\partial_t{\bar Z}^a S^d u_2(t))
\|_{L^1({\mathbb R}^3)}\\
\leq&
C\langle t\rangle^{-1}
\|
\chi_2 |x| \partial_\gamma{\bar Z}^{a'}S^{d'}u_1(t)
\|_{L^\infty_r L^4_\omega({\mathbb R}^3)}\\
&
\hspace{1cm}
\times
\|
\partial_{\alpha\beta}^2{\bar Z}^{a''}S^{d''}u_2(t)
\|_{L^2_r L^4_\omega({\mathbb R}^3)}
\|
\partial_t{\bar Z}^a S^d u_2(t)
\|_{L^2({\mathbb R}^3)}\\
\leq&
C\langle t\rangle^{-1+2\delta}
\langle\!\langle u(t)\rangle\!\rangle
{\mathcal N}(u(t))^2.
\end{split}
\end{equation}
%%%%%%%%%%%%%%%%%%%%%%%%%%%%%%%%
It is easy to obtain a similar estimate for 
$\chi_2 J_{2,3}$ and $\chi_2 J_{2,4}$. 

Using the basic fact that the integration of 
$(1+\tau)^{-1+2\delta}$ from $0$ to $t$ 
is $O(t^{2\delta})$ for large $t$, 
we can now complete the proof of Proposition \ref{proposition3}, 
except the required estimate for 
$\|
\langle t-r\rangle^{-(1/2)-\eta}
T_j\partial_t{\bar Z}^aS^d u_2
\|_{L^2((0,T)\times{\mathbb R}^3)}$, 
$|a|+d\leq 2,\,d\leq 1$. 
How to handle the similar norm for $u_1$ has been dwelt on 
at the end of the last section, 
and we have only to follow the same approach as there. 
The proof of Proposition \ref{proposition3} has been finished. 
%%%%%%%%%%%%%%%%%%%%%%%%%%%%%%%%%%%
%%%%%%%%%%%%%%%%%%%%%%%%%%%%%%%%%%
%%%%%%%%%%%%%%%%%%%%%%%%%%%%%%%%%
\section{Space-time $L^2$ estimates of $u_1$}
In this section, we will prove the following:
\begin{proposition}\label{proposition4}
The following inequality holds 
for smooth local solutions to $(\ref{1})$ 
$u=(u_1,u_2)$, 
as long as they satisfy $(\ref{eqn:daiji1})$ for some time interval $(0,T):$
%%%%%%%%%%%%%%%%%%%%%%%%%%%%%%%%%
\begin{equation}\label{140}
\begin{split}
&\left(
\sup_{0<t<T}
\langle t\rangle^{-\mu-\delta}
\biggl(
\int_0^t
L(u_1(\tau))^2
d\tau
\biggr)^{1/2}
\right)^2\\
\leq&
C
N(u_1(0))^2
+
C
\biggl(
\sup_{0<t<T}
N(u_1(t))
\biggr)
{\mathcal N}_T(u)^2\\
&+
C
\biggl(
\sup_{0<t<T}
\langle\!\langle u(t)\rangle\!\rangle
\biggr)
\biggl(
\sup_{0<t<T}N(u_1(t))
\biggr)^2\\
&+
C\biggl(
\sup_{0<t<T}
\langle\!\langle u(t)\rangle\!\rangle
\biggr)
\left(
\sup_{0<t<T}
\langle t\rangle^{-\mu-\delta}
\biggl(
\int_0^t
L(u_1(\tau))^2
d\tau
\biggr)^{1/2}
\right)^2.
\end{split}
\end{equation}
%%%%%%%%%%%%%%%%%%%%%%%%%%%%%%%%%
\end{proposition}
%%%%%%%%%%%%%%%%%%%%%%%%%%%
The rest of this section is devoted to the proof of this proposition. 
We may focus on the most troublesome case 
$|a|=2$ and $d=1$, 
when considering the estimate of 
$\|
r^{-(3/2)+\mu}{\bar Z}^a S^d u_1
\|_{L^2((0,t)\times{\mathbb R}^3)}$ 
and 
$\|
r^{-(1/2)+\mu}\partial{\bar Z}^a S^d u_1
\|_{L^2((0,t)\times{\mathbb R}^3)}$. 
By Lemma \ref{lemma2.7}, we see that for $|a|=2$
%%%%%%%%%%%%%%%%%%%%%%%%%%%%%%%%%%%%%%%%%
\begin{equation}\label{141}
\begin{split}
&(1+t)^{-2\mu}
\bigl(
\|
r^{-(3/2)+\mu}{\bar Z}^a S u_1
\|_{L^2((0,t)\times{\mathbb R}^3)}^2\\
&
\hspace{2cm}
+
\|
r^{-(1/2)+\mu}\partial{\bar Z}^a S u_1
\|_{L^2((0,t)\times{\mathbb R}^3)}^2
\bigr)\\
\leq&
C\|
\partial{\bar Z}^a S u_1(0)
\|_{L^2({\mathbb R}^3)}^2\\
&+
C\sum_{k=1}^2
\sum\!{}^{'}
\int_0^t\!\!\!\int_{{\mathbb R}^3}
|\partial{\bar Z}^a S u_1|
|\partial{\bar Z}^{a'} S^{d'} u_k|
|\partial^2{\bar Z}^{a''} S^{d''} u_1|dxd\tau\\
&+
C\sum_{k,j=1}^2
\sum\!{}^{''}
\int_0^t\!\!\!\int_{{\mathbb R}^3}
|\partial{\bar Z}^a S u_1|
|\partial{\bar Z}^{a'} S^{d'} u_j|
|\partial{\bar Z}^{a''} S^{d''} u_k|dxd\tau\\
&+
C\sum_{k=1}^2
\sum\!{}^{'}
\int_0^t\!\!\!\int_{{\mathbb R}^3}
r^{-1+2\mu}
\langle r\rangle^{-2\mu}
|{\bar Z}^a S u_1|
|\partial{\bar Z}^{a'} S^{d'} u_k|
|\partial^2{\bar Z}^{a''} S^{d''} u_1|dxd\tau\\
&+
C\sum_{k,j=1}^2
\sum\!{}^{''}
\int_0^t\!\!\!\int_{{\mathbb R}^3}
r^{-1+2\mu}
\langle r\rangle^{-2\mu}
|{\bar Z}^a S u_1|
|\partial{\bar Z}^{a'} S^{d'} u_j|
|\partial{\bar Z}^{a''} S^{d''} u_k|dxd\tau\\
&+
C
\sum_{k=1}^2
\int_0^t\!\!\!\int_{{\mathbb R}^3}
|\partial^2 u_k|
|\partial{\bar Z}^a S u_1|^2dxd\tau\\
&+
C
\sum_{k=1}^2
\int_0^t\!\!\!\int_{{\mathbb R}^3}
r^{-1+2\mu}
\langle r\rangle^{-2\mu}
|\partial^2 u_k|
|{\bar Z}^a S u_1|
|\partial{\bar Z}^a S u_1|dxd\tau\\
&+
C
\sum_{k=1}^2
\int_0^t\!\!\!\int_{{\mathbb R}^3}
r^{-1+2\mu}
\langle r\rangle^{-2\mu}
|\partial u_k|
|\partial{\bar Z}^a S u_1|^2dxd\tau\\
&+
C
\sum_{k=1}^2
\int_0^t\!\!\!\int_{{\mathbb R}^3}
r^{-2+2\mu}
\langle r\rangle^{-2\mu}
|\partial u_k|
|{\bar Z}^a S u_1|
|\partial{\bar Z}^a S u_1|dxd\tau\\
:=&
C\|
\partial{\bar Z}^a S u_1(0)
\|_{L^2((0,t)\times{\mathbb R}^3)}^2
+
\int_0^t\!\!\!\int_{{\mathbb R}^3}
K_{1,1}dxd\tau
+
\cdots
+
\int_0^t\!\!\!\int_{{\mathbb R}^3}
K_{1,8}dxd\tau.
\end{split}
\end{equation}
%%%%%%%%%%%%%%%%%%%%%%%%%%%%%%%%%%%%%
We again separate 
${\mathbb R}^3$ into 
the two pieces 
$\{x\in{\mathbb R}^3:|x|<(\tau+1)/2\}$ and its complement 
for the estimate of $K_{1,l}$ for $l=1,\dots,6$. 

\noindent{\bf$\cdot$ Estimate of \boldmath$\chi_1 K_{1,l}$ for 
\boldmath$l=1,\dots,6$}. 
By $\chi_1$, we mean the characteristic function 
of the set $\{x\in{\mathbb R}^3:|x|<(\tau+1)/2\}$ 
for any fixed $\tau\in (0,t)$. 

\noindent\underline{Estimate of $\chi_1 K_{1,1}$} 
Let us first consider the case $d'=0$ 
(and hence $d''\leq 1$). 
If $|a'|=2$ (and hence $|a''|=0$), 
then 
we get by (\ref{eqn:ell6}), (\ref{eqn:3yokutsukau2})
%%%%%%%%%%%%%%%%%%%%%%%%%%%%%%%%%%
\begin{equation}\label{142}
\begin{split}
&\|\chi_1K_{1,1}\|_{L^1({\mathbb R}^3)}\\
\leq&
C\langle\tau\rangle^{-1}
\sum_{k=1}^2
\sum\!{}^{'}
N(u_1(\tau))
\|
\langle \tau-r\rangle
\partial{\bar Z}^{a'}u_k(\tau)
\|_{L^6}
\|
\partial^2 S^{d''}u_1(\tau)
\|_{L^3}\\
\leq&
C\langle\tau\rangle^{-1+2\delta}
N(u_1(\tau))
{\mathcal N}(u(\tau))^2.
\end{split}
\end{equation}
%%%%%%%%%%%%%%%%%%%%%%%%%%%%%%%%%%%
If $|a'|\leq 1$ (and hence $|a''|\leq 1$), 
then we get by (\ref{eqn:ellinfty}) and (\ref{eqn:3yokutsukau3})
%%%%%%%%%%%%%%%%%%%%%%%%%%%%%%%%%%%%
\begin{equation}\label{143}
\begin{split}
&\|\chi_1K_{1,1}\|_{L^1({\mathbb R}^3)}\\
\leq&
C\langle\tau\rangle^{-1}
\sum_{k=1}^2
\sum\!{}^{'}
N(u_1(\tau))
\|
\langle \tau-r\rangle
\partial{\bar Z}^{a'}u_k(\tau)
\|_{L^\infty}
\|
\partial^2{\bar Z}^{a''}S^{d''}u_1(\tau)
\|_{L^2}\\
\leq&
C\langle\tau\rangle^{-1+2\delta}
N(u_1(\tau))
{\mathcal N}(u(\tau))^2.
\end{split}
\end{equation}
%%%%%%%%%%%%%%%%%%%%%%%%%%%%%%%%%%%%
Here, we have used (\ref{eqn:3yokutsukau3}) 
together with 
$\|\partial S u_k(\tau)\|_{L^\infty}
\leq
\langle\tau\rangle^\delta
\langle\!\langle u(\tau)\rangle\!\rangle
\leq
\varepsilon_3^*\langle\tau\rangle^\delta$. 
(See (\ref{eqn:daiji1}) for $\varepsilon_3^*$.)

Let us next consider the case $d''=0$ (and hence $d'\leq 1$). 
By considering the case 
$|a''|=2$ (and hence $|a'|=0$) 
and 
$|a''|\leq 1$ (and hence $|a'|\leq 1$) separately, 
we are able to obtain
%%%%%%%%%%%%%%%%%%%%%%%%%%%%%%%%%%
\begin{equation}\label{144}
\|\chi_1K_{1,1}\|_{L^1({\mathbb R}^3)}
\leq
C\langle\tau\rangle^{-1+2\delta}
N(u_1(\tau))
{\mathcal N}(u(\tau))^2
\end{equation}
%%%%%%%%%%%%%%%%%%%%%%%%%%%%%%%%%
in the same way as above.

\noindent\underline{Estimate of $\chi_1 K_{1,2}$} 
We can obtain
%%%%%%%%%%%%%%%%%%%%%%%%%%%%%%%%
\begin{equation}\label{145}
\|\chi_1K_{1,2}\|_{L^1({\mathbb R}^3)}
\leq
C\langle\tau\rangle^{-1+2\delta}
N(u_1(\tau))
{\mathcal N}(u(\tau))^2
\end{equation}
%%%%%%%%%%%%%%%%%%%%%%%%%%%%%%%%
in a similar way. 

\noindent\underline{Estimate of $\chi_1 K_{1,3}$} 
Using the Hardy inequality and proceeding as above, 
we can obtain
%%%%%%%%%%%%%%%%%%%%%%%%%%%%%%%%
\begin{equation}\label{146}
\begin{split}
&\|\chi_1 K_{1,3}\|_{L^1({\mathbb R}^3)}\\
\leq&
\sum_{k=1}^2
\sum\!{}^{'}
\|
r^{-1}{\bar Z}^a S u_1
\|_{L^2}
\|
\chi_1
(\partial{\bar Z}^{a'}S^{d'}u_k)
(\partial^2{\bar Z}^{a''}S^{d''}u_1)
\|_{L^2}\\
\leq&
C\langle\tau\rangle^{-1+2\delta}
N(u_1(\tau))
{\mathcal N}(u(\tau))^2.
\end{split}
\end{equation}

\noindent\underline{Estimate of $\chi_1 K_{1,4}$}\,\,\,
In the same way as in (\ref{146}), we obtain 
%%%%%%%%%%%%%%%%%%%%%%%%%%%%%%%%
\begin{equation}\label{147}
\|\chi_1 K_{1,4}\|_{L^1({\mathbb R}^3)}
\leq
C\langle\tau\rangle^{-1+2\delta}
N(u_1(\tau))
{\mathcal N}(u(\tau))^2.
\end{equation}
%%%%%%%%%%%%%%%%%%%%%%%%%%%%%%%%%

\noindent\underline{Estimate of $\chi_1 K_{1,5}$}\,\,\,
Using (\ref{eqn:ellinfty}) and (\ref{eqn:3c}), we get
%%%%%%%%%%%%%%%%%%%%%%%%%%%%%
\begin{equation}\label{148}
\|\chi_1 K_{1,5}\|_{L^1({\mathbb R}^3)}
\leq
C\langle\tau\rangle^{-1+\delta}
{\mathcal N}(u(\tau))
N(u_1(\tau))^2.
\end{equation}
%%%%%%%%%%%%%%%%%%%%%%%%%%

\noindent\underline{Estimate of $\chi_1 K_{1,6}$}\,\,\,
Using the Hardy inequality, we can obtain 
%%%%%%%%%%%%%%%%%%%%%%%%%%%%%
\begin{equation}\label{149}
\|\chi_1 K_{1,6}\|_{L^1({\mathbb R}^3)}
\leq
C\langle\tau\rangle^{-1+\delta}
{\mathcal N}(u(\tau))
N(u_1(\tau))^2.
\end{equation}
%%%%%%%%%%%%%%%%%%%%%%%%%%%

\noindent{\bf$\cdot$ Estimate of \boldmath$\chi_2 K_{1,l}$ for 
\boldmath$l=1,\dots,6$}.\,\,\,
We next consider $\chi_2 K_{1,l}$ for $l=1,\dots,6$. 

\noindent\underline{Estimate of $\chi_2 K_{1,1}$}\,\,\,
If $|a''|+d''=2$ (and hence $|a'|+d'\leq1$), then 
we get by (\ref{eqn:j3}) and (\ref{eqn:3yokutsukau3})
%%%%%%%%%%%%%%%%%%%%%%%%%%%%%%%%%%
\begin{equation}\label{150}
\begin{split}
&\|
\chi_2 K_{1,1}
\|_{L^1({\mathbb R}^3)}\\
\leq&
C
\langle\tau\rangle^{-1}
\sum_{k=1}^2
\sum_{{|a'|+d'\leq1}\atop{|a''|+d''=2}}
N(u_1(\tau))
\|
\chi_2
\langle r\rangle
\partial{\bar Z}^{a'}S^{d'}u_k(\tau)
\|_{L^\infty}
\|
\chi_2
\partial^2{\bar Z}^{a''}S^{d''}u_1
\|_{L^2}\\
\leq&
C
\langle\tau\rangle^{-1+2\delta}
N(u_1(\tau))
{\mathcal N}(u(\tau))^2.
\end{split}
\end{equation}
%%%%%%%%%%%%%%%%%%%%%%%%%%%%%%%%%%%
Here we have used (\ref{eqn:3yokutsukau3}) 
along with 
$\|\chi_2\partial Su_2(\tau)\|_{L^\infty}
\leq
C\langle\!\langle u(\tau)\rangle\!\rangle
\leq
C\varepsilon_3^*$.
If $|a''|+d''=1$ (and hence $|a'|+d'\leq 2$), 
then we get by (\ref{eqn:j2})
%%%%%%%%%%%%%%%%%%%%%%%%%%%%%%
\begin{equation}\label{151}
\begin{split}
&\|
\chi_2 K_{1,1}
\|_{L^1({\mathbb R}^3)}\\
\leq&
C
\langle\tau\rangle^{-1}
\sum_{k=1}^2
\sum_{{|a'|+d'\leq 2}\atop{|a''|+d''=1}}
N(u_1(\tau))
\|
\chi_2
\langle r\rangle
\partial{\bar Z}^{a'}S^{d'}u_k(\tau)
\|_{L^\infty_r L^4_\omega}\\
&\hspace{4.5cm}
\times
\|
\partial^2{\bar Z}^{a''}S^{d''}u_1
\|_{L^2_r L^4_\omega}\\
\leq&
C
\langle\tau\rangle^{-1+2\delta}
N(u_1(\tau))
{\mathcal N}(u(\tau))^2.
\end{split}
\end{equation}
%%%%%%%%%%%%%%%%%%%%%%%%%%%%%
On the other hand, 
if $|a''|+d''=0$ (and hence $|a'|+d'\leq 3$), 
then we get by using (\ref{eqn:j3}) and (\ref{eqn:3c})
%%%%%%%%%%%%%%%%%%%%%%%%%%%%%%%%%%%
\begin{equation}\label{152}
\begin{split}
&\|
\chi_2 K_{1,1}
\|_{L^1({\mathbb R}^3)}\\
\leq&
C
\langle\tau\rangle^{-1}
\sum_{k=1}^2
\sum_{|a'|+d'\leq 3}
N(u_1(\tau))
\|
\partial{\bar Z}^{a'}S^{d'}u_k(\tau)
\|_{L^2}
\|
\chi_2
\langle r\rangle
\partial^2u_1(\tau)
\|_{L^\infty}\\
\leq&
C
\langle\tau\rangle^{-1+2\delta}
N(u_1(\tau))
{\mathcal N}(u(\tau))^2.
\end{split}
\end{equation}
%%%%%%%%%%%%%%%%%%%%%%%%%%%%%%%%%%
\noindent\underline{Estimate of $\chi_2 K_{1,2}$}\,\,\,
Using (\ref{eqn:j3}), we can easily get
%%%%%%%%%%%%%%%%%%%%%%%%%%%%%%%%
\begin{equation}\label{153}
\|
\chi_2 K_{1,2}
\|_{L^1({\mathbb R}^3)}
\leq
C
\langle\tau\rangle^{-1+2\delta}
N(u_1(\tau))
{\mathcal N}(u(\tau))^2.
\end{equation}
%%%%%%%%%%%%%%%%%%%%%%%%%%%%%%%
\noindent\underline{Estimate of $\chi_2 K_{1,3}$}\,\,\,
Arguing as in (\ref{150})--(\ref{152}) and using the Hardy inequality, we get
\begin{equation}\label{154}
\begin{split}
&\|
\chi_2 K_{1,3}
\|_{L^1({\mathbb R}^3)}\\
\leq&
C
\sum_{k=1}^2
\sum\!{}^{'}
\|r^{-1}{\bar Z}^aSu_1(\tau)\|_{L^2}
\|
\chi_2
(\partial{\bar Z}^{a'}S^{d'}u_k(\tau))
(\partial^2{\bar Z}^{a''}S^{d''}u_1(\tau))
\|_{L^2}\\
\leq&
C
\langle\tau\rangle^{-1+2\delta}
N(u_1(\tau))
{\mathcal N}(u(\tau))^2.
\end{split}
\end{equation}
%%%%%%%%%%%%%%%%%%%%%%%%%%%%%%%%%%
\noindent\underline{Estimate of $\chi_2 K_{1,4}$}\,\,\,
Using the Hardy inequality and proceeding as in (\ref{153}), we can obtain
%%%%%%%%%%%%%%%%%%%%%%%%%%%%%%%%%%%%%%%%%%%%%
\begin{equation}\label{155}
\|
\chi_2 K_{1,4}
\|_{L^1({\mathbb R}^3)}
\leq
C
\langle\tau\rangle^{-1+2\delta}
N(u_1(\tau))
{\mathcal N}(u(\tau))^2.
\end{equation}
%%%%%%%%%%%%%%%%%%%%%%%%%%%%%%%%%%
\noindent\underline{Estimate of $\chi_2 K_{1,5}$}\,\,\,
Using (\ref{eqn:j3}) and (\ref{eqn:3c}), we easily get
%%%%%%%%%%%%%%%%%%%%%%%%%%%%%%%
\begin{equation}\label{156}
\|
\chi_2 K_{1,5}
\|_{L^1({\mathbb R}^3)}
\leq
C
\langle\tau\rangle^{-1+\delta}
\langle\!\langle u(\tau)\rangle\!\rangle
N(u_1(\tau))^2.
\end{equation}
%%%%%%%%%%%%%%%%%%%%%%%%%%%%
\noindent\underline{Estimate of $\chi_2 K_{1,6}$}\,\,\,
By using the Hardy inequality, we can obtain
%%%%%%%%%%%%%%%%%%%%%%%%%%%
\begin{equation}\label{157}
\|
\chi_2 K_{1,6}
\|_{L^1({\mathbb R}^3)}
\leq
C
\langle\tau\rangle^{-1+\delta}
\langle\!\langle u(\tau)\rangle\!\rangle
N(u_1(\tau))^2
\end{equation}
%%%%%%%%%%%%%%%%%%%%%%%%%%%%
in the same way as in (\ref{156}). 

%%%%%%%%%%%%%%%%%%%%%%%%%%%%%%%%%%%%%%%
\noindent{\bf$\cdot$ Estimate of \boldmath$K_{1,7}$ and 
\boldmath$K_{1,8}$}.\,\,\,
It is easy to get by (\ref{eqn:ellinfty}) and (\ref{eqn:j3})
%%%%%%%%%%%%%%%%%%%%%%%%%%%%%%%%%%
\begin{equation}\label{158}
\|K_{1,7}\|_{L^1({\mathbb R}^3)}
\leq
C
\langle\tau\rangle^{-1+\delta}
\langle\!\langle u(\tau)\rangle\!\rangle
L(u_1(\tau))^2.
\end{equation}
%%%%%%%%%%%%%%%%%%%%%%%%%%%%%
We also get
%%%%%%%%%%%%%%%%%%%%%%%%%%
\begin{equation}\label{159}
\begin{split}
\|K_{1,8}\|_{L^1({\mathbb R}^3)}
&
\leq
C
\langle\tau\rangle^{-1+\delta}
\langle\!\langle u(\tau)\rangle\!\rangle
\|
r^{-(3/2)+\mu}
{\bar Z}^a S u_1(\tau)
\|_{L^2({\mathbb R}^3)}
L(u_1(\tau))\\
&
\leq
C
\langle\tau\rangle^{-1+\delta}
\langle\!\langle u(\tau)\rangle\!\rangle
L(u_1(\tau))^2.
\end{split}
\end{equation}
%%%%%%%%%%%%%%%%%%%%%%%%%%%%%
Now we are ready to complete the proof of 
Proposition \ref{proposition4}. 
In view of (\ref{141})--(\ref{159}), 
we have only to explain how to 
handle the integral over $(0,t)$ of 
$\|K_{1,7}(\tau)\|_{L^1({\mathbb R}^3)}$ 
and 
$\|K_{1,8}(\tau)\|_{L^1({\mathbb R}^3)}$. 
Without loss of generality, 
we may suppose $1<t<T$. 
It follows from (\ref{158}) that 
%%%%%%%%%%%%%%%%%%%%%%%%%%
\begin{equation}\label{160}
\int_0^1
\|K_{1,7}(\tau)\|_{L^1({\mathbb R}^3)}
d\tau
\leq
C
\sup_{0<\tau<T}
\langle\!\langle u(\tau)\rangle\!\rangle
\int_0^1
L(u_1(\tau))^2 d\tau,
\end{equation}
%%%%%%%%%%%%%%%%%%%%%%%%%%%
and
%%%%%%%%%%%%%%%%%%%%%
\begin{equation}\label{161}
\begin{split}
&\int_1^t
\|K_{1,7}(\tau)\|_{L^1({\mathbb R}^3)}
d\tau\\
\leq&
C
\sup_{0<\tau<T}
\langle\!\langle u(\tau)\rangle\!\rangle
\sum_{j=0}^N
(2^j)^{-1+\delta}
\int_{2^j}^{2^{j+1}}
L(u_1(\tau))^2 d\tau\\
\leq&
C
\sup_{0<\tau<T}
\langle\!\langle u(\tau)\rangle\!\rangle
\biggl(
\sum_{j=0}^\infty
(2^j)^{-1+\delta+2(\mu+\delta)}
\biggr)\\
&
\hspace{2.5cm}
\times
\biggl(
\sup_{1<\sigma<T}
\langle\sigma\rangle^{-\mu-\delta}
\biggl(
\int_1^\sigma
L(u_1(\tau))^2 d\tau
\biggr)^{1/2}
\biggr)^2.
\end{split}
\end{equation}
%%%%%%%%%%%%%%%%%%%%%%%%%%%%%%%%%%
Similarly, we get by (\ref{159})
%%%%%%%%%%%%%%%%%%%%%%%%%%%
\begin{equation}\label{162}
\int_0^1
\|K_{1,8}(\tau)\|_{L^1({\mathbb R}^3)}
d\tau
\leq
C
\sup_{0<\tau<T}
\langle\!\langle u(\tau)\rangle\!\rangle
\int_0^1
L(u_1(\tau))^2 d\tau
\end{equation}
%%%%%%%%%%%%%%%%%%%%%%%%
and 
%%%%%%%%%%%%%%%%%%%%%%%
\begin{equation}\label{163}
\begin{split}
\int_1^t
\|K_{1,8}(\tau)\|_{L^1({\mathbb R}^3)}
d\tau
\leq&
C
\sup_{1<\tau<T}
\langle\!\langle u(\tau)\rangle\!\rangle
\biggl(
\sum_{j=0}^\infty
(2^j)^{-1+\delta+2(\mu+\delta)}
\biggr)\\
&
\hspace{0.2cm}
\times
\left(
\sup_{1<\sigma<T}
\langle\sigma\rangle^{-\mu-\delta}
\biggl(
\int_1^\sigma
L(u_1(\tau))^2 d\tau
\biggr)^{1/2}
\right)^2.
\end{split}
\end{equation}
%%%%%%%%%%%%%%%%%%%%%%%%%%%%%%%%%%
Since $\delta$ and $\mu$ are 
sufficiently small positive numbers, 
we see that the series in (\ref{161}) and (\ref{163}) converges. 
Therefore we have finished the proof of (\ref{140}). 
%%%%%%%%%%%%%%%%%%%%%%%%%%%%%%%%%%%%%%%%%
%%%%%%%%%%%%%%%%%%%%%%%%%%%%%%%%%%
%%%%%%%%%%%%%%%%%%%%%%%%%%%%%%%%%%
\section{Space-time $L^2$ estimates of $u_2$}
In this section, we consider 
the space-time $L^2$ estimates of $u_2$. 
We can prove:
\begin{proposition}\label{proposition5}
The following inequality holds 
for smooth local solutions to $(\ref{1})$ 
$u=(u_1,u_2)$, 
as long as they satisfy $(\ref{eqn:daiji1})$ 
for some time interval $(0,T):$
%%%%%%%%%%%%%%%%%%%%%%%%%%%%%%%%%%%%%%%5
\begin{equation}\label{164}
\begin{split}
&\left(
\sup_{0<t<T}
\langle t\rangle^{-\mu-(3\delta/2)}
\biggl(
\int_0^t
L(u_2(\tau))^2
d\tau
\biggr)^{1/2}
\right)^2\\
\leq&
C
N(u_2(0))^2
+
C
\biggl(
\sup_{0<t<T}
\langle t\rangle^{-\delta}
N(u_2(t))
\biggr)
{\mathcal N}_T(u)^2\\
&+
C
\biggl(
\sup_{0<t<T}
\langle\!\langle u(t)\rangle\!\rangle
\biggr)
\biggl(
\sup_{0<t<T}
\langle t\rangle^{-\delta}
N(u_2(t))
\biggr)^2\\
&+
C\biggl(
\sup_{0<t<T}
\langle\!\langle u(t)\rangle\!\rangle
\biggr)
\left(
\sup_{0<t<T}
\langle t\rangle^{-\mu-(3\delta/2)}
\biggl(
\int_0^t
L(u_2(\tau))^2
d\tau
\biggr)^{1/2}
\right)^2.
\end{split}
\end{equation}
%%%%%%%%%%%%%%%%%%%%%%%%%%%%%%%%%
\end{proposition}
We have only to repeat essentially the same argument as in 
Section 6. We thus omit the proof.
%%%%%%%%%%%%%%%%%%%%%%%%%%%%%%%%%%%%%%%%
%%%%%%%%%%%%%%%%%%%%%%%%%%%%%%%%%%%%%
%%%%%%%%%%%%%%%%%%%%%%%%%%%%%%%%%%%%%%
\section{Proof of Theorem 1.3} 
So far, we have proved that 
local solutions to (\ref{1}) defined for 
$(t,x)\in [0,T)\times{\mathbb R}^3$ 
with compactly supported smooth data 
satisfy 
%%%%%%%%%%%%%%%%%%%%%%%%%%%%5
\begin{equation}\label{165}
\begin{split}
&
{\mathcal N}_T(u)^2
+
{\mathcal G}_T(u)^2
+
{\mathcal L}_T(u)^2\\
\leq&
C_0
\bigl(
N_4(u_1(0))^2
+
N_4(u_2(0))^2
\bigr)\\
&+
C_1
\biggl(
\sup_{0<t<T}
\langle\!\langle u(t)\rangle\!\rangle
\bigl(
{\mathcal N}_T(u)^2
+
{\mathcal G}_T(u)^2
+
{\mathcal L}_T(u)^2
\bigr)\\
&
\hspace{1.1cm}
+{\mathcal N}_T(u)^3
+{\mathcal G}_T(u){\mathcal N}_T(u)^2
\biggr)
\end{split}
\end{equation}
%%%%%%%%%%%%%%%%%%%%%%%%%%%%%%%%%%%%%%%%%%%%
for suitable constants $C_0$, $C_1>0$, provided that 
%%%%%%%%%%%%%%%%%%%%%%%%%%%%%%%%%%%%%
\begin{equation}\label{166}
\sup_{0<t<T}
\langle\!\langle u(t)\rangle\!\rangle
\leq
\varepsilon_3^*.
\end{equation}
%%%%%%%%%%%%%%%%%%%%%%%%%%%%%%%
See (\ref{eqn:daiji1}) for $\varepsilon_3^*$. 
In order to get the key a priori estimate (see (\ref{175}) below), 
we must show that 
$\langle\!\langle u(t)\rangle\!\rangle$ is small 
(at least for a short time interval), 
whenever $N_4(u_1(0))+N_4(u_2(0))$ is small enough. 
(See (\ref{174}) below.)

Since initial data belong to 
$C_0^\infty({\mathbb R}^3)\times C_0^\infty({\mathbb R}^3)$ 
and 
the uniqueness theorem of $C^2$-solutions 
and 
its corollary in \cite[p.\,53]{John1990} apply to 
the system (\ref{1}), 
smooth local solutions satisfy
%%%%%%%%%%%%%%%%%%%%%%%%%
\begin{equation}\label{167}
u_1(t,x)
=
u_2(t,x)
=0,
\quad
0\leq t<T,\,\,|x|\geq R+t,
\end{equation}
%%%%%%%%%%%%%%%%%%%%%%%%%%
where $R>0$ is a constant such that 
$u_i(0,x)=\partial_t u_i(0,x)=0$ $(i=1,2)$ for $|x|\geq R$. 
(Remark: All the constants $C$ appearing below 
will be independent of $R$.) 
Moreover, thanks to (\ref{167}), 
we can easily verify
%%%%%%%%%%%%%%%%%%%%%%%%
\begin{equation}\label{168}
{\mathcal N}(u(t))
\in
C([0,T)).
\end{equation}
%%%%%%%%%%%%%%%%%%%%%%
(Actually, the last property can be seen as a direct consequence 
of the fact $N_4(u_i(t))^2$
$\in C^\infty([0,T))$, $i=1,2$.) 
Due to (\ref{eqn:3a}) and (\ref{168}), 
we know 
%%%%%%%%%%%%%%%%%%%%%%%%%%
\begin{equation}\label{169}
{\mathcal N}(u(t))
\leq 
2A\varepsilon_0
\quad
(
A:=
\max
\{\sqrt{C_0},C_{KS},1\}
)
\end{equation}
%%%%%%%%%%%%%%%%%%%%%%%%%%%%
(see (\ref{165}), (\ref{eqn:KSineq}) for the constants $C_0$, $C_{KS}$) 
at least for a short time interval, which means 
\begin{align*}
\{\,T>0\,:&\,\mbox{For given data $(f_i,g_i)\in 
C_0^\infty({\mathbb R}^3)\times C_0^\infty({\mathbb R}^3)$ 
$(i=1,2)$ satisfying (\ref{eqn:3a}),}\\
&
\mbox{there exists a unique smooth solution $(u_1,u_2)$ to $(\ref{1})$ 
defined for }\\
&
\mbox{all $(t,x)\in [0,T)\times{\mathbb R}^3$ 
satisfying 
${\mathcal N}(u(t))\leq 2A\varepsilon_0$ 
for any $t\in [0,T)$}
\}\ne\emptyset.
\end{align*}
We define $T_*$ as the supremum of this non-empty set. 
In order to establish the key estimates (\ref{174}) and (\ref{175}), 
we must first prove:
\begin{proposition}\label{proposition6}
Suppose $(\ref{eqn:3a})$ for compactly supported smooth data. 
Then the local solution to $(\ref{1})$ satisfies 
%%%%%%%%%%%%%%%%%%%%%%%%%%%%%%%%%%
\begin{equation}\label{170}
{\mathcal M}(u(t))
\leq
2A{\mathcal N}(u(t)),
\quad
0<t<T_*.
\end{equation}
%%%%%%%%%%%%%%%%%%%%%%%%%%%%%%%%
\end{proposition}
For the constant $A$, see (\ref{169}) above. 
\begin{proof}
When the initial data is identically zero and 
hence the corresponding solution identically vanishes, 
we obviously get (\ref{170}). 
We may therefore suppose without loss of generality 
that the smooth initial data is not identically zero. 
We thus have ${\mathcal N}(u(0))>0$. 
Moreover, we actually know ${\mathcal N}(u(t))>0$ 
for all $t\in (0,T_*)$. Indeed, suppose 
${\mathcal N}(u(T_0))=0$ for some 
$T_0\in (0,T_*)$. 
Since $\partial u(T_0,x)$ is identically zero and 
$u(T_0,x)$ has compact support, 
$u(T_0,x)$ and $\partial_t u(T_0,x)$ are also 
identically zero. 
Define $w(t,x):=u(T_0-t,x)$. 
We then see that $w$ satisfies a system of quasi-linear 
wave equations to which the above-mentioned 
uniqueness theorem of $C^2$-solutions \cite{John1990} applies. 
Since $w(0,x)$ and $\partial_t w(0,x)$ are identically zero, 
we know by this uniqueness theorem that 
$w$ is a trivial solution, 
which in particular means 
$w(T_0,x)$ and $\partial_t w(T_0,x)$ are 
identically zero. 
This  contradicts the fact that 
the initial data $(u(0,x), \partial_t u(0,x))$ is non-trivial. 

Note that ${\mathcal N}(u(t))>0$ for $t\in [0,T_*)$ 
in the following discussion. 
Since Lemma \ref{lemma2.8} yields 
${\mathcal M}(u(0))\leq A{\mathcal N}(u(0))$, 
that is, 
${\mathcal M}(u(0))/{\mathcal N}(u(0))\leq A$,  
and ${\mathcal M}(u(t))$, ${\mathcal N}(u(t))$, and 
${\mathcal M}(u(t))/{\mathcal N}(u(t))$ 
are continuous on the interval $[0,T_*)$, 
we see 
$$
{\mathcal M}(u(t))/{\mathcal N}(u(t))\leq 2A,
$$
that is, 
%%%%%%%%%%%%%%%%%%%%%%%%%%%%%
\begin{equation}\label{171}
{\mathcal M}(u(t))
\leq
2A{\mathcal N}(u(t))
\end{equation}
%%%%%%%%%%%%%%%%%%%%%%%%
at least for a short time interval 
$\subset [0,T_*)$. 
It remains to show that in fact, the last inequality holds 
for {\it all} $t\in [0,T_*)$. 
Let
$$
{\hat T}
:=
\sup
\{\,T\in (0,T_*)\,:
{\mathcal M}(u(t))
\leq
2A{\mathcal N}(u(t)) 
\mbox{\,for all\,\,}
t\in [0,T)
\}.
$$
By definition we know ${\hat T}\leq T_*$. 
To show ${\hat T}=T_*$, we proceed as follows. 
Since we have 
${\mathcal N}(u(t))\leq 2A\varepsilon_0$ 
$(0<t<T_*)$, 
we obtain by Lemmas \ref{lemma2.4}--\ref{lemma2.6}
%%%%%%%%%%%%%%%%%%%%%%%%%%%%%%%%%%%%%
\begin{equation}\label{172}
\begin{split}
\langle\!\langle
u(t)
\rangle\!\rangle
&
\leq
C_2
\bigl(
{\mathcal N}(u(t))
+
{\mathcal M}(u(t))
\bigr)\\
&
\leq
C_2(1+2A){\mathcal N}(u(t))
\leq
2AC_2(1+2A)\varepsilon_0,
\quad
0<t<{\hat T}
\end{split}
\end{equation}
%%%%%%%%%%%%%%%%%%%%%%%%%%%%%%%%%%%
for a constant $C_2>0$. 
Because of $2AC_2(1+2A)\varepsilon_0\leq
\min\{\varepsilon_1^*,\varepsilon_2^*\}$ 
(see (\ref{eqn:3a})), 
we can use Proposition \ref{proposition1} with $T={\hat T}$ to get
$$
{\mathcal M}(u(t))
\leq
A{\mathcal N}(u(t))
+
2AC_2C_3(1+2A)\varepsilon_0
\bigl(
{\mathcal M}(u(t))+{\mathcal N}(u(t))
\bigr),
\quad
0<t<{\hat T}
$$
for a constant $C_3>0$, which yields owing to 
the definition of $\varepsilon_0$ (see (\ref{eqn:3a}))
%%%%%%%%%%%%%%%%%%%%%%%%%%%%%%%%
\begin{equation}\label{173}
{\mathcal M}(u(t))
\leq
\frac{A+2AC_2C_3(1+2A)\varepsilon_0}{1-2AC_2C_3(1+2A)\varepsilon_0}
{\mathcal N}(u(t))
\leq
\frac{3}{2}A{\mathcal N}(u(t)),
\quad
0<t<{\hat T}.
\end{equation}
%%%%%%%%%%%%%%%%%%%%%%%%%%%%%%%%%%%%
Since ${\mathcal M}(u(t))/{\mathcal N}(u(t))
\in C([0,T_*))$, 
we finally arrive at the conclusion ${\hat T}=T_*$. 
(If we assume ${\hat T}<T_*$, the estimate (\ref{173}) contradicts 
the definition of ${\hat T}$.) 
We have finished the proof of Proposition \ref{proposition6}. 
\end{proof}
Now we are in a position to complete the proof of 
the key a priori estimate (\ref{175}) below. 
As in (\ref{172}), 
we get by Proposition \ref{proposition6} and the definition of 
$\varepsilon_0$
%%%%%%%%%%%%%%%%%%%%%%%%%%%
\begin{equation}\label{174}
\langle\!\langle
u(t)
\rangle\!\rangle
\leq
2AC_2(1+2A)\varepsilon_0
\leq
\varepsilon_3^*,
\quad
0<t<T_*.
\end{equation}
%%%%%%%%%%%%%%%%%%
We can use (\ref{165}) with $T=T_*$ owing to (\ref{174}). 
Using the inequalities 
${\mathcal N}(u(t))\leq 2A\varepsilon_0$, 
$\langle\!\langle
u(t)
\rangle\!\rangle
\leq
2AC_2(1+2A)\varepsilon_0$ 
$(0<t<T_*)$, 
we get from (\ref{165})
\begin{align*}
&
{\mathcal N}_{T_*}(u)^2
+
{\mathcal G}_{T_*}(u)^2
+
{\mathcal L}_{T_*}(u)^2\\
&
\leq
\frac{C_0\bigl(N_4(u_1(0))^2+N_4(u_2(0))^2\bigr)}
{1-2AC_1C_2(1+2A)\varepsilon_0-3AC_1\varepsilon_0},
\end{align*}
which yields owing to the definition of 
$\varepsilon_0$ 
%%%%%%%%%%%%%%%%%%%%%%%%%%%%%%%
\begin{equation}\label{175}
{\mathcal N}(u(t))
\leq
\frac32
A
{\mathcal N}(u(0))
\leq
\frac32
A\varepsilon_0, 
\quad
0<t<T_*.
\end{equation}
%%%%%%%%%%%%%%%%%%%%%%%%%%%%%%%
Now we are in a position to show $T_*=\infty$. 
Assume $T_*<\infty$. 
By solving (\ref{1}) with data 
$(u_i(T_*-\delta,x),\partial_t u_i(T_*-\delta,x))
\in 
C_0^\infty({\mathbb R}^3)\times C_0^\infty({\mathbb R}^3)
$ 
(see (\ref{167})) given at $t=T_*-\delta$ 
($\delta>0$ is sufficiently small), 
we can extend this local solution smoothly 
to a larger strip, say, 
$\{(t,x)\,:\,0<t<{\tilde T},\,x\in{\mathbb R}^3\}$, 
where ${\tilde T}>T_*$. 
Such a smooth local solution defined for 
$(t,x)\in (0,{\tilde T})\times {\mathbb R}^3$ 
satisfies 
${\mathcal N}(u(t))\in C([0,{\tilde T}))$. 
Moreover, because of ${\mathcal N}(u(T_*))\leq 3A\varepsilon_0/2$ 
by (\ref{175}), 
we see that there exists $T'\in (T_*,{\tilde T}]$ such that 
${\mathcal N}(u(t))\leq 2A\varepsilon_0$, 
$0<t<T'$, 
which contradicts the definition of $T_*$. 
Hence we have $T_*=\infty$. 

To complete the proof of Theorem \ref{theorem1.2}, 
we must relax the regularity of data 
and eliminate compactness of the support of data. 
Naturally, 
we employ the standard mollifier and cut-off idea 
(see, e.g., \cite[p.\,12]{Friedman} and \cite[p.\,122]{Hor}). 
Then, we easily see that, 
for any $(f_i,g_i)$ $(i=1,2)$ 
satisfying $f_1,f_2\in L^6({\mathbb R}^3)$ and 
$$
C_D\sum_{i=1,2}
D(f_i,g_i)
\leq
\frac{\varepsilon_0}{2}
$$
(see (\ref{eqn:3a}) for the constant $C_D$), 
there exists a sequence 
$(f_{i,n},g_{i,n})\in 
C_0^\infty({\mathbb R}^3)\times C_0^\infty({\mathbb R}^3)$ 
$(n=1,2,\dots)$ such that 
%%%%%%%%%%%%%%%%%%%%%%%%%%%%%
\begin{equation}\label{176}
C_D\sum_{i=1,2}
D(f_{i,n},g_{i,n})\leq
\varepsilon_0
\end{equation}
%%%%%%%%%%%%%%%%%%%%%%%%%%%
for sufficiently large $n$, and 
%%%%%%%%%%%%%%%%%%%%%%%%%%%%%
\begin{equation}\label{177}
\sum_{i=1,2}
D(f_{i,n}-f_i,g_{i,n}-g_i)
\to 0\quad
(n\to\infty).
\end{equation}
%%%%%%%%%%%%%%%%%%%%%%%%%%
(We must keep in mind that this procedure becomes 
rather complicated 
when we employ $W_4$ (see (\ref{6})), as in \cite{HY2017}, 
to measure the size of data.) 
Thanks to (\ref{176}), 
we know that 
the Cauchy problem (\ref{1}) with data 
$(u_i(0),\partial_t u_i(0))
=(f_{i,n},g_{i,n})$ $(i=1,2)$ 
admits a unique solution, 
which is denoted by 
$u_n(t,x)=(u_{1,n}(t,x),u_{2,n}(t,x))$, 
for every large $n$. Also, we have
%%%%%%%%%%%%%%%%%%%%%%%%%%%%%%%%%%%
\begin{equation}\label{178}
{\mathcal N}_T(u_n)
+
{\mathcal G}_T(u_n)
+
{\mathcal L}_T(u_n)
\leq
C\sum_{i=1,2}
D(f_{i,n},g_{i,n})
\leq
C\varepsilon_0
\end{equation}
%%%%%%%%%%%%%%%%%%%%%%%%%%%%%
for all $T>0$, 
with a constant $C>0$ independent of 
$n$ and $T$. 
Furthermore, 
owing to (\ref{178}) and 
${\mathcal M}(u_n(t))\leq C{\mathcal N}(u_n(t))$ 
for $0<t<\infty$ (see (\ref{170})), we obtain 
by the same argument as 
(in fact, essentially simpler argument than) in Sections 4--7, 
with a few obvious modifications
\begin{equation}\label{179}
\begin{split}
&\sup_{t>0}
N_1(u_{1,m}(t)-u_{1,n}(t))
+
\sup_{t>0}
\langle t\rangle^{-\delta}
N_1(u_{2,m}(t)-u_{2,n}(t))\\
&
\hspace{0.1cm}
+
\biggl(
\int_0^\infty
\sum_{i=1}^3
\|
\langle t-r\rangle^{-(1/2)-\eta}
T_i
\bigl(
u_{1,m}(t)-u_{1,n}(t)
\bigr)
\|_{L^2({\mathbb R}^3)}^2dt
\biggr)^{1/2}\\
&
\hspace{0.1cm}
+
\sup_{t>0}
\langle t\rangle^{-\delta}
\biggl(
\int_0^t
\sum_{i=1}^3
\|
\langle \tau-r\rangle^{-(1/2)-\eta}
T_i
\bigl(
u_{2,m}(\tau)-u_{2,n}(\tau)
\bigr)
\|_{L^2({\mathbb R}^3)}^2d\tau
\biggr)^{1/2}\\
&
\hspace{0.1cm}
+
\sup_{t>0}
\langle t\rangle^{-\mu-\delta}
\biggl(
\int_0^t
\|
r^{-(3/2)+\mu}
\bigl(
u_{1,m}(\tau)-u_{1,n}(\tau)
\bigr)
\|_{L^2({\mathbb R}^3)}^2d\tau
\biggr)^{1/2}\\
&
\hspace{0.1cm}
+
\sup_{t>0}
\langle t\rangle^{-\mu-\delta}
\biggl(
\int_0^t
\|
r^{-(1/2)+\mu}
\partial
\bigl(
u_{1,m}(\tau)-u_{1,n}(\tau)
\bigr)
\|_{L^2({\mathbb R}^3)}^2d\tau
\biggr)^{1/2}\\
&
\hspace{0.1cm}
+
\sup_{t>0}
\langle t\rangle^{-\mu-(3\delta/2)}
\biggl(
\int_0^t
\|
r^{-(3/2)+\mu}
\bigl(
u_{2,m}(\tau)-u_{2,n}(\tau)
\bigr)
\|_{L^2({\mathbb R}^3)}^2d\tau
\biggr)^{1/2}\\
&
\hspace{0.1cm}
+
\sup_{t>0}
\langle t\rangle^{-\mu-(3\delta/2)}
\biggl(
\int_0^t
\|
r^{-(1/2)+\mu}
\partial
\bigl(
u_{2,m}(\tau)-u_{2,n}(\tau)
\bigr)
\|_{L^2({\mathbb R}^3)}^2d\tau
\biggr)^{1/2}\\
\leq&
C
\sum_{i=1,2}
\bigl(
\|
\nabla(f_{i,m}-f_{i,n})
\|_{L^2({\mathbb R}^3)}
+
\|
g_{i,m}-g_{i,n}
\|_{L^2({\mathbb R}^3)}
\bigr)
\end{split}
\end{equation}
for sufficiently large $m$, $n$, 
with a constant $C$ independent of $m$, $n$. 
(When showing (\ref{179}), 
we are supposed to choose $\varepsilon_0$ smaller than before, 
if necessary.) 
We thus see by the standard argument 
that $u_n=(u_{1,n},u_{2,n})$ has the limit that is the solution 
to (\ref{1}) with the data $(f_i,g_i)$ $(i=1,2)$ given 
at $t=0$. The proof of Theorem \ref{theorem1.2} has been completed. 
%%%%%%%%%%%%%%%%%%%%%%%%%%%%%%%%%%%%%
%%%%%%%%%%%%%%%%%%%%%%%%%%%%%%%%%%%
\section*{Acknowledgments} The authors are very grateful to the referees for 
careful reading of the manuscript and helpful comments. 
The first author was supported by 
JSPS KAKENHI Grant Numbers JP15K04955 and JP18K03365. 
The second author was supported by 
the National Natural Science Foundation of China (No.11801068), 
Shanghai Sailing Program (No.17YF1400700) 
and the Fundamental Research Funds for the Central Universities (No.17D110913). 

% You may incorporate your references as follows in your main tex file.
% Using BibTex is not recommended but can be handled.

\medskip
% The data information below will be filled by AIMS editorial staff
%Received xxxx 20xx; revised xxxx 20xx.
\medskip

\end{document}